\documentclass[a4paper, reqno]{amsart}

\usepackage[UKenglish]{babel}
\usepackage{amsmath, amssymb, amsthm}
\usepackage{keytheorems}                    % Mainly to fix cleveref names (otherwise it'll refer to a definition as a theorem, etc.)
\usepackage[hidelinks]{hyperref}
\usepackage[capitalise,nosort]{cleveref}    % The nosort option is because a bug in cleveref gets the sorting wrong sometimes, so we do it manually
\usepackage{xpatch}
\usepackage{color}
\usepackage{tikz-cd}
\usepackage{enumitem}
\usepackage{bbm}
\usepackage{quiver}

\usepackage{geometry}
\usepackage[inline,final]{showlabels}
\usepackage{comment}

\xpatchcmd{\proof}{\itshape}{\normalfont\bfseries}{}{}
\newtheoremstyle{repeat}{}{}{\itshape}{}{\bfseries}{.}{.5em}{#3, repeated}

\newtheorem{theorem}{Theorem}[section]
\newtheorem{proposition}[theorem]{Proposition}
\newtheorem{lemma}[theorem]{Lemma}
\newtheorem{corollary}[theorem]{Corollary}
\newtheorem{fact}[theorem]{Fact}

\theoremstyle{definition}
\newtheorem{definition}[theorem]{Definition}

\newtheorem{remark}[theorem]{Remark}

\newtheorem{notation}[theorem]{Notation}
\newtheorem{example}[theorem]{Example}

\theoremstyle{repeat}
\newtheorem*{repeated-theorem}{Repeat}

\makeatletter
\if@cref@capitalise
\crefname{fact}{Fact}{Facts}
\crefname{notation}{Notation}{Notations}
\else
\crefname{fact}{fact}{facts}
\crefname{notation}{notation}{notations}
\fi
\makeatother

\edef\PolishL{\L}

\newcommand{\A}{{\mathcal A}}
\newcommand{\B}{{\mathcal B}}
\newcommand{\F}{{\mathcal F}}
\newcommand{\C}{{\mathcal C}}

\newcommand{\E}{{\mathcal E}}
\newcommand{\G}{{\mathcal G}}

\newcommand{\I}{{\mathcal I}}

\renewcommand{\L}{{\mathcal L}}
\newcommand{\M}{{\mathcal M}}
\newcommand{\MM}{{\mathfrak M}}

\renewcommand{\P}{{\mathcal P}}

\newcommand{\U}{{\mathcal U}}
\newcommand{\V}{{\mathcal V}}

\newcommand{\X}{{\mathcal X}}

\newcommand{\Z}{{\mathbb Z}}

\newcommand{\LL}{{\mathbb L}}

\newcommand{\TT}{{\mathbb T}}
\newcommand{\NN}{{\mathbb N}}
\newcommand{\RR}{{\mathbb R}}

\DeclareMathOperator{\colim}{colim}
\DeclareMathOperator{\BareDiag}{BareDiag}
\DeclareMathOperator{\Diag}{Diag}
\DeclareMathOperator{\dom}{dom}
\DeclareMathOperator{\tp}{tp}
\newcommand{\tppp}{\tp_\textup{pp}}
\DeclareMathOperator{\Ob}{Ob}
\DeclareMathOperator{\id}{id}
\DeclareMathOperator{\Hom}{Hom}

\newcommand{\Ab}{\mathbf{Ab}}
\newcommand{\Cat}{\mathbf{Cat}}

\newcommand{\Met}{\mathbf{Met}}
\newcommand{\CMet}{\mathbf{CMet}}
\newcommand{\PMet}{\mathbf{PMet}}
\newcommand{\CMon}{\mathbf{CMon}}
\newcommand{\omegaCPO}{\omega\text{-}\mathbf{CPO}}
\newcommand{\Mod}{\mathbf{Mod}}
\newcommand{\Set}{\mathbf{Set}}
\newcommand{\Sh}{\mathbf{Sh}}
\newcommand{\Str}{\mathbf{Str}}
\newcommand{\Pos}{\mathbf{Pos}}
\newcommand{\Act}[1]{#1\text{-}\mathbf{Act}}

\renewcommand{\phi}{\varphi}
\newcommand{\op}{^{\textup{op}}} %I have added ^ so that we don't have to write it every time
\newcommand{\bbtwo}{{\mathbbm{2}}}

\newcommand{\fp}{{\textup{fp}}}
\newcommand{\coh}{{\textup{coh}}}

\title{The Compactness Theorem for Enriched Logic}
\author{M. Kamsma and G. Tendas}
\date{\today}
\thanks{
\begin{minipage}{0.6\textwidth}
The first-named and second-named authors acknowledge with gratitude the support of the Marie Sk\l{}odowska-Curie postdoctoral fellowship 101130801 and the EPSRC postdoctoral fellowship EP/X027139/1, respectively.
\end{minipage}%
\begin{minipage}{0.4\textwidth}
\begin{center}
    \includegraphics[height=1.1cm]{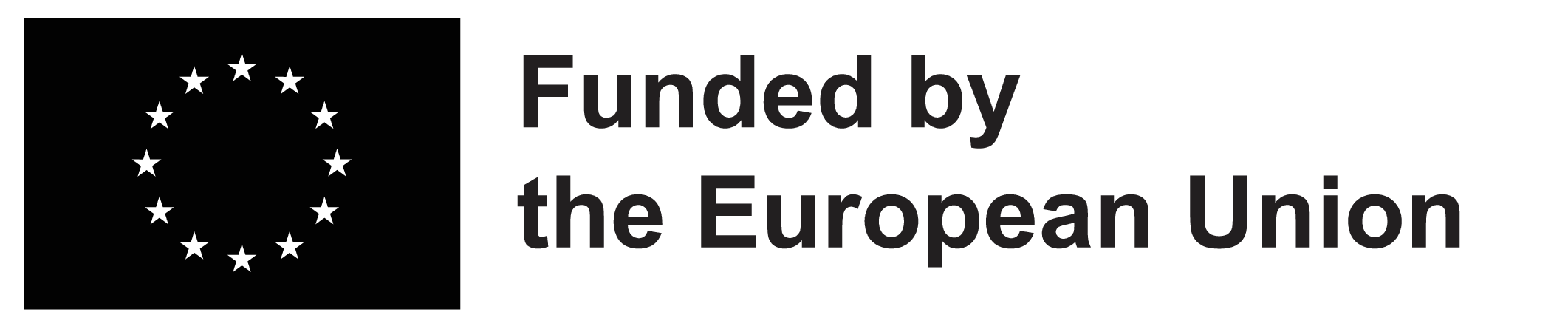}
\end{center}
\end{minipage}
\nopunct
}

\address[Mark Kamsma]{Department of Mathematics and Statistics, Masaryk University, Faculty of Sciences, Kotl\'{a}\v{r}sk\'{a} 2, 611 37, Brno, Czech Republic}
\urladdr{https://markkamsma.nl}
\email{mark@markkamsma.nl}

\address[Giacomo Tendas]{Department of Mathematics, University of Manchester, Faculty of Science and Engineering, Alan Turing Building, M13 9PL Manchester, UK}
\email{giacomo.tendas@manchester.ac.uk}

\begin{document}

\begin{abstract}
The framework of enriched logic extends positive logic (also known as coherent logic) from sets to structures internal to a symmetric monoidal closed category $\V$. This framework was developed by Rosick\'{y} and the second-named author. The main novelty of this logic is that the arities of the function and relation symbols are now objects in $\V$. We prove the compactness theorem for this enriched logic. That is, a set of $\V$-enriched sequents has a model (internal to $\V$) if every finite subset has a model. To achieve this, we define ultraproducts of enriched structures and prove enriched versions of {\PolishL}os' theorem. We then develop some model-theoretic applications. We start with an enriched version of the model-theoretic method of diagrams, establishing a correspondence between homomorphisms whose domain is some enriched structure $A$ and models of the set $\Diag(A)$ of all enriched atomic formulas with parameters from $A$ that are true in $A$. Combining this with our compactness theorem, we prove amalgamation results, laying the foundations for enriched syntactic types (in the model-theoretic sense). We finish the paper with a collection of examples of $\V$ that our results apply to. Besides recovering the classical case when $\V = \Set$, these include many presheaf toposes, categories of monoid acts, the category of abelian groups and related additive categories, categories of (small) enriched categories, and the categories of posets and pseudometric spaces.
\end{abstract}

\maketitle

\setcounter{tocdepth}{1} % Given the many subsections (particularly in the applications and examples section), not showing subsections looks a lot friendlier
\tableofcontents

\section{Introduction}
The compactness theorem is one of the cornerstones of model theory. Together with {\PolishL}os' theorem, it provides a powerful bridge between local and global properties, showing that structures satisfying arbitrarily large finite fragments of a theory can be assembled into genuine models. These ideas are behind some of the most important techniques in the subject, ranging from the model-theoretic method of diagrams and amalgamation to the construction of nonstandard models of analysis.

From a categorical perspective, compactness is closely related to the existence and well-behaviour of ultraproducts. In the ordinary setting, {\PolishL}os' theorem identifies the interpretation of a first-order formula $\phi$ in an ultraproduct of structures $A= \prod_{i\in \I}A_i/\U$ with the ultraproduct of their interpretations: 
\[
\phi(A) \cong \prod_{i\in \I}\phi(A_i)/\U.
\]
The compactness theorem then follows by a standard ultraproduct construction.

The interaction between logic and category theory has proved remarkably fruitful, leading to deep connections with accessible categories~\cite{MP89:libro,AR94:libro}, ultracategories~\cite{makkaiStoneDualityFirst1987,lurieUltracategories}, Grothendieck toposes~\cite{lawvere1970quantifiers,maclane2012sheaves}, the development of hyperdoctrines~\cite{lawvere1969adjointness,lawvere1968equality}, the theory of deductive systems~\cite{lambek1968deductive}, and type theory~\cite{seely1984locally}. Recently, enriched versions of regular and positive logic (also known as coherent logic) have been developed in a series of papers by Rosický and the second-named author \cite{RT23EUA,RT25ERegular,RT25positive}. These frameworks extend familiar fragments of first-order logic from sets to structures internal to a symmetric monoidal closed category $\V$, equipped with a suitable factorisation system $(\E, \M)$, which encodes how subobjects and existential quantifiers are interpreted. They recover classical regular and positive logic when $\V = \Set$, while also including examples arising from algebra~\cite{Pre11:libro}, order theory~\cite{bloom1976varieties,adamek2022categorical}, functional analysis~\cite{iovino2014applications}, and theoretical computer science~\cite{mardare2016quantitative}. Despite the development of the syntax and semantics of enriched logic, a general compactness theorem has so far been missing.

\textbf{Main results.} The main goal of the present paper is to provide a compactness theorem for enriched logic. More precisely, we establish enriched versions of {\PolishL}os' theorem, which we will refer to as {\PolishL}os-style theorems from now on, and compactness for both regular and positive logic. Our approach follows the classical strategy, but substantial new difficulties arise. 

In the enriched setting, ultraproducts must be constructed at the level of the base of enrichment, and the interpretation of formulas involves the chosen factorisation system, rather than subobjects alone. As a consequence, the validity of the compactness theorem depends on categorical properties of the base of enrichment and chosen factorisation system that are automatic in $\Set$, but that are an additional and genuine requirement in a general enriched framework.

To isolate the needed assumptions, we introduce the notion of {\em reduced product stable} enrichment bases (\cref{reduced-product-ready,reduced-product-stable}). This guarantees that reduced products interact suitably with finite limits, powers, and the chosen factorisation system. To cover the case of positive logic an additional condition is required, namely what we call the {\em ultraproduct covering property} (\cref{ultraproduct-covering-property}). This property controls the interaction between ultraproducts and finite coproducts and turns out to be satisfied in a large number of natural examples.

Under these hypotheses we prove {\PolishL}os-style theorems for enriched logic. The first theorem states that the interpretation of regular formulas commutes with reduced products, and that the interpretation of positive formulas commutes with ultraproducts. The second theorem gives the familiar consequence: if a sequent holds on a large subset of the factors, then it holds in the corresponding reduced product (\cref{los-sequents-converse} establishes a converse under additional assumptions).
\begin{theorem}[{\PolishL}os-style theorem 1]
\label{los-theorem}
Assume that $\V$ is reduced product stable. Let $\LL$ be a language on $\V$ and let $\varphi$ be an $\LL$-formula. Let $A := \prod_{i \in \I} A_i/\F$ be a reduced product of $\LL$-structures. Then
\[
\phi(A) \cong \prod_{i\in \I}\phi(A_i)/\F
\]
provided that:
\begin{enumerate}[label=(\roman*)]
\item\label{los-theorem:regular} $\varphi$ is regular, or;
\item\label{los-theorem:positive} $\varphi$ is positive, $\F$ is an ultrafilter, and $\V$ has the ultraproduct covering property.
\end{enumerate}
\end{theorem}
\begin{theorem}[{\PolishL}os-style theorem 2]
\label{los-sequents}
Assume that $\V$ is reduced product stable. Let $\LL$ be a language on $\V$ and let $\sigma$ be a sequent of $\LL$-formulas. Let $A := \prod_{i \in \I} A_i/\F$ be a reduced product of $\LL$-structures. Then
\[
\{i \in \I \mid A_i \models \sigma \} \in \F \quad \text{implies} \quad A \models \sigma,
\]
provided that:
\begin{enumerate}[label=(\roman*)]
\item\label{los-sequents:regular} $\sigma$ is a sequent of regular formulas, or;
\item\label{los-sequents:positive} $\sigma$ is a sequent of positive formulas, $\F$ is an ultrafilter, and $\V$ has the ultraproduct covering property.
\end{enumerate}
\end{theorem}
Another key feature of this paper is the treatment of formulas with parameters, which are extremely important in applications of the compactness theorem. In contrast with the classical setting, parameters cannot in general be represented by constant symbols, unless the base of enrichment is semicartesian (i.e., the monoidal unit is terminal, see \cref{parameters-not-enriched,sec:semicartesian-case}). This forces a separate treatment of parameters (see \cref{parameter}), so that our results are applicable beyond the semicartesian case. We extend our second {\PolishL}os-style theorem so that it can handle parameters (see \cref{los-parametrised}). Using this we then prove the compactness theorem for enriched regular and positive logic, allowing parameters.
\begin{theorem}[Compactness]
\label{compactness}
Assume that $\V$ is reduced product stable. Let $\LL$ be a language on $\V$, let $\TT$ be a theory, and let $\Sigma$ be a set of sequents with parameters in $\P$. Assume also that the monoidal unit $I$ is finitely presentable or that all sequents in $\Sigma$ are of the form $\top \vdash \psi(p_1, \ldots, p_n)$.

If every finite subset of $\TT \cup \Sigma$ has a model then $\TT \cup \Sigma$ has a model, provided that:
\begin{enumerate}[label=(\roman*)]
\item\label{compactness:regular} every sequent in $\TT \cup \Sigma$ is a sequent of regular formulas (possibly with parameters), or;
\item\label{compactness:positive} every sequent in $\TT \cup \Sigma$ is a sequent of positive formulas (possibly with parameters) and $\V$ has the ultraproduct covering property.
\end{enumerate}
\end{theorem}
As an application, we develop an enriched version of the model-theoretic method of diagrams. Given an enriched structure $A$, we show that homomorphisms with domain $A$ correspond precisely to models of the set $\Diag(A)$ of all (regular) atomic formulas with parameters from $A$ that are true in $A$ (this is called the \emph{diagram} of $A$, see \cref{diagram-of-structure}).
\begin{theorem}[Enriched model-theoretic method of diagrams]
\label{enriched-method-of-diagrams}
Assume that $\V_0$ is locally finitely presentable with its finitely presentable objects closed under the monoidal product and that $(\E, \M)$ is proper. Let $\LL$ be a language on $\V$ and let $A$ be an $\LL$-structure. Then the functor
\begin{align*}
A/\Str(\LL)_0 \quad &\to \quad \Mod(\Diag(A)), \\
h\colon A \to B \quad &\mapsto \quad (B, (h \circ a)_{p_a \in \P(A)}),
\end{align*}
which is the identity on arrows, is an isomorphism of categories.
\end{theorem}
We use this machinery to derive amalgamation properties, which ultimately result in the power of the central model-theoretic notion of a type as a set of formulas. As a corollary we show how spans of structures, or even models of a fixed theory, can be amalgamated (\cref{amalgamation}) if one of the arrows is an immersion (i.e., reflects truth of formulas).
\begin{theorem}
\label{types-amalgamate-structures}
Assume that $\V$ is reduced product stable, that $\V_0$ is locally finitely presentable, that $(\E, \M)$ is proper, and that $I$ is $\E$-projective.

Let $A$ be an $\LL$-structure and $B$ be a model of some $\LL$-theory $\TT$, and let $A \xleftarrow{a} X \xrightarrow{b} B$ be a span of arrows in $\V$. Suppose that $\tppp(a; A) \subseteq \tppp(b; B)$. Then there is a model $C$ of $\TT$ together with homomorphisms $f\colon A \to C$ and $g\colon B \to C$ such that $f \circ a = g \circ b$, provided that:
\begin{enumerate}[label=(\roman*)]
\item\label{types-amalgamate-structures:regular} $\TT$ is regular, or;
\item\label{types-amalgamate-structures:positive} $\TT$ is positive and $\V$ has the ultraproduct covering property.
\end{enumerate}
\end{theorem}

The assumptions required for our main results are shown to hold in a broad range of examples. Besides recovering the classical case over $\Set$, our framework applies to many presheaf toposes, categories of monoid acts, the category of abelian groups and related additive categories, categories of (small) enriched categories, and the categories of posets and pseudometric spaces. We also exhibit examples where some of the assumptions fail (sometimes we still have enough for our results concerning regular logic, see \cref{sec:sheaves-on-boolean-algebra}), showing that the additional categorical conditions introduced in the paper are genuinely necessary.

\textbf{Overview.} The paper is organised as follows. In \cref{sec:preliminaries} we recall the required background on factorisation systems and enriched logic. We then start by defining ultraproducts and reduced products for enriched structures in \cref{reduced-products-and-ultraproducts}. \Cref{sec:los} establishes enriched {\PolishL}os-style theorems. In \cref{compactness-and-formulas-with-parameters} we introduce parameters and prove the compactness theorem. \Cref{model-theoretic-method-of-diagrams,sec:amalgamation} develop the enriched model-theoretic method of diagrams and its applications to amalgamation. The standalone \cref{sec:ultraproduct-covering-property} studies the ultraproduct covering property and provides general criteria ensuring it holds. Finally, \cref{applications-and-examples} collects a range of examples and counterexamples illustrating the scope of our theory.

\section{Preliminaries}\label{sec:preliminaries}

\subsection{Factorisation systems}\label{section:fact-system}

We will assume the reader to be familiar with the notion of an (orthogonal) factorisation system, introduced in~\cite{freyd1972categories}. Given a category $\C$ this consists of a pair $(\E,\M)$ formed by classes of arrows in $\C$ satisfying the properties of~\cite[Definition~2.2]{freyd1972categories}. In this section we recall some basic properties of this theory.

\begin{definition}
\label{proper-factorisation-system}
A factorisation system $(\E, \M)$ on a category $\C$ is called \emph{proper} if $\E$ and $\M$ are contained in the classes of epimorphisms and monomorphisms, respectively.
\end{definition}
Since our enrichment base will be complete and cocomplete, the following fact (and its dual) give equivalent characterisations of $\M$ being contained in the monomorphisms (and dually, $\E$ being contained in the epimorphisms).
\begin{fact}[{\cite[Proposition 2.1.4]{freyd1972categories}}]
\label{factorisation-system-m-monos-equivalent-conditions}
The following are equivalent for a factorisation system $(\E, \M)$ in a category with finite coproducts:
\begin{itemize}
\item every arrow in $\M$ is a monomorphism;
\item $\E$ contains the strong epimorphisms;
\item $\E$ contains the split epimorphism;
\item $\E$ contains the regular epimorphisms;
\item the left cancellation property in $\E$ holds: $f g\in\E$ implies $f\in\E$.
\end{itemize}
\end{fact}

\begin{definition}
\label{e-generator}
Let $(\E, \M)$ be a factorisation system on a cocomplete category $\C$. A set of objects $\X$ in $\C$ is called an \emph{$\E$-generator} if for any object $A$ in $\C$ the arrow
\[
\sum_{f\colon X \to A, X \in \X} X \xrightarrow{\sum_{f\colon X \to A, X \in \X} f} A
\]
is in $\E$.
\end{definition}

The notions that are usually referred to as generator and strong generator (recalled below) arise in this way and are used quite often in category theory.

\begin{remark}
	If $\C$ is a locally presentable category both the (strong epi, mono) and (epi, strong mono) factorisation systems exist by~\cite[Proposition~1.61]{AR94:libro}. Then we say that $\X\subseteq \C$ is a:\begin{itemize}
		\item {\em generator} if it is an (epi)-generator,
		\item {\em strong generator} if it is a (strong epi)-generator.
	\end{itemize} 
	Beside the equivalent characterisation given by \cref{equivalent-e-generator} below we also have the following (see \cite[Propositions~4.5.8 and 4.5.20]{Bor94:libro}): let $J\colon \C\to [\X\op,\Set]$ be the restriction of the Yoneda embedding; then \begin{enumerate}
		\item $\X$ is a generator if and only if $J$ is faithful;
		\item $\X$ is a strong generator if and only if $J$ is conservative.
	\end{enumerate}
\end{remark}

We will freely use the following equivalent characterisation of $\E$-generators.
\begin{proposition}
\label{equivalent-e-generator}
Let $(\E, \M)$ be a factorisation system on a cocomplete category $\C$ with $\M$ contained in the monomorphisms. A set of objects $\X$ in $\C$ is an $\E$-generator if and only if for any $A\in\C$ there is a family $\{X_i\}_{i \in \I}$ of objects in $\X$ and an arrow $\sum_{i \in \I} X_i \to A$ in $\E$.
\end{proposition}
\begin{proof}
The left to right direction is trivial. The other direction follows from the canonical factorisation
\[
\sum_{i \in \I} X_i \to \sum_{f\colon X \to A, X \in \X} X \xrightarrow{\sum_{f\colon X\to A, X \in \X} f} A
\]
of the arrow $\sum_{i \in \I} X_i \to A$ in $\E$, and the fact that $\E$ has the left cancellation property by \cref{factorisation-system-m-monos-equivalent-conditions}.
\end{proof}
\begin{remark}
\label{finitely-presentables-e-generator}
If $(\E, \M)$ is a factorisation system on a locally finitely presentable category $\C$, with $\M$ contained in the monomorphisms, then $\C_\fp$ is an $\E$-generator. Indeed, by definition $\C_\fp$ is a strong generator of $\C$ and $\E$ contains the strong epimorphisms by \cref{factorisation-system-m-monos-equivalent-conditions}.
\end{remark}
%\begin{proof}
%Let $A$ be any object of $\C$. Then $A$ is a directed colimit of objects in $\C_\fp$. This directed colimit can be written as a coequaliser of coproducts of objects in $\C_\fp$. So $\C_\fp$ is a $\{\text{regular epimorphisms}\}$-generator. We conclude because $\E$ contains the regular epimorphisms by \cref{factorisation-system-m-monos-equivalent-conditions}.
%\end{proof}
\begin{lemma}
\label{lifts-along-f-imply-f-in-e}
Let $(\E, \M)$ be a factorisation system on a cocomplete category $\C$ with $\M$ contained in the monomorphisms. Assume also that $\C$ has an $\E$-generator $\X$. Suppose that $f\colon A \to B$ is an arrow such that any $g\colon X \to B$, with $X \in \X$, factors through $f$. Then $f \in \E$.
\end{lemma}
\begin{proof}
	Since $\X$ is an $\E$-generator we have an arrow $\sum_{j \in J} g_j \colon \sum_{j\in J} X_j\to B$, with $X_j\in\X$, which lies in $\E$. By assumption each $g_j$ factors through $f$ inducing an arrow $\tilde g_j\colon X_j\to A$. It follows that the triangle 
	\begin{center}
		\begin{tikzpicture}[baseline=(current  bounding  box.south), scale=2]
			
			\node (a0) at (0,0.7) {$\sum_{j\in J} X_j$};
			\node (b0) at (1,0.7) {$A$};
			\node (c0) at (1,0) {$B$};
			
			\path[font=\scriptsize]
			
			(a0) edge [->] node [above] {$\sum_{j\in J} \tilde g_j$} (b0)
			(a0) edge [->>] node [below] {$\sum_{j\in J} g_j\ \ \ \ \ \ \ $} (c0)
			(b0) edge [->] node [right] {$f$} (c0);
		\end{tikzpicture}	
	\end{center}
	commutes. Since $\M$ is contained in the monomorphisms, the left cancellation property holds, so that $f\in\E$.
\end{proof}

\subsection{Enriched logic}
\label{enriched-logic}
We assume the reader to be familiar with the basics of enriched category theory. For a reference on this topic, see \cite{Kel82:libro}.

We work in the setting developed in \cite{RT23EUA,RT25ERegular,RT25positive}; and therefore we fix a symmetric monoidal closed complete and cocomplete category $\V=(\V_0,\otimes,I)$ endowed with a factorisation system $(\E,\M)$, which is:\begin{itemize}
	\item enriched~\cite{lucyshyn2014enriched}: $\M$ is closed in $\V^\bbtwo$ under powers (or equivalently, $\E$ is closed under tensors); 
	\item such that every arrow in $\M$ is a monomorphism.
\end{itemize}

Later on we shall assume additional hypotheses on $\V$ and $(\E,\M)$, depending on the results that we need to prove; in particular, $\V_0$ will often be assumed to be locally finitely presentable and $(\E,\M)$ to be proper.

We denote by $\V_\fp$ the full subcategory spanned by the finitely presentable objects of $\V_0$. When $\V$ is {\em locally finitely presentable as a closed category}, meaning that $\V_\fp$ contains the unit and is closed under the monoidal product, there is a well-behaved theory of locally finitely presentable $\V$-categories (see \cite{Kel82:articolo}), and $\V_\fp$ coincides with the full subcategory of the {\em enriched} finitely presentable objects of $\V$. However, we will not consider enriched presentability in this paper.

Unlike the references we follow, that allow infinitary logic, in this paper we will treat only the {\em finitary} case, so that whenever we talk about languages, formulas, and theories, we actually mean finitary languages, finitary formulas, and finitary theories. Let us begin by recalling the notions of a language $\LL$, $\LL$-structures, and terms in $\LL$.

\begin{definition}[{\cite[Definition~3.3]{RT25ERegular}}]
	A (single-sorted) {\em language} $\LL$ (over $\V$) is the data of:\begin{enumerate}
		\item a set of function symbols $f\colon(X,Y)$ with arities $X,Y\in \V_\fp$;
		\item a set of relation symbols $R:X$, with arity an object $X\in\V_\fp$.
	\end{enumerate}
\end{definition}

\begin{definition}[{\cite[Definition~3.4]{RT25ERegular}}]
	Given a language $\LL$, an {\em $\LL$-structure} is the data of an object $A\in\V$ with:\begin{enumerate}
		\item an arrow $f_A\colon A^X\to A^Y$ in $\V$ for every function symbol $f\colon(X,Y)$ in $\LL$;
		\item an $\M$-subobject $r_A\colon R_A\rightarrowtail A^X$ for every relation symbol $R\colon X$ in $\LL$.
	\end{enumerate}
\end{definition}

\begin{definition}[{\cite[Definition~4.1]{RT23EUA}}]
	Let $\LL$ be a language. The class of \emph{$\LL$-terms} is defined recursively as follows:
	\begin{enumerate}
		\item every arrow $f\colon Y\to X$ of $\V_\fp$ is an $(X,Y)$-ary term; 
		\item every function symbol $f:(X,Y)$ of $\LL$ is an $(X,Y)$-ary term;
		\item if $t$ is a $(X,Y)$-ary term and $Z$ is in $\V_\fp$, then $t^Z$ is a $(Z\otimes X,Z\otimes Y)$-ary term;
		\item given $(t_1,\cdots,t_n)$, where each $t_j$ is an $(X_j,Y_j)$-ary term, and an $(\sum_{j=1}^n Y_j,W)$-ary term $s$; then $s(t_1,\cdots,t_n)$ is a $(\sum_{j=1}^n X_j, W)$-ary term.
	\end{enumerate}
\end{definition}	

Interpretation of terms in an $\LL$-structure is defined recursively as one would expect; see \cite[Definition~4.1]{RT23EUA} for details. Next we recall how the class of positive $\LL$-formulas is defined.

\begin{definition}[{\cite[Definition~3.4]{RT25positive}}]
	Given a language $\LL$, the atomic formulas of $\LL$ are defined as follows:
	\begin{enumerate}
		\item for any $(X,Y)$-ary terms $s,t$, 
		$$s(x)=t(x)$$ is an $X$-ary atomic formula;
		\item for any $X$-ary relation $R$, any $Y\in\V_\fp$, and any $(Z,X\!\otimes\! Y)$-ary term $t$, 
		$$R^Y(t(z))$$ is a $Z$-ary atomic formula;
		\item {\em truth}: for any arity $X\in\V_\fp$ we have the $X$-ary atomic formula $\top$;
		\item {\em false}: for any arity $X\in\V_\fp$ we have the $X$-ary atomic formula $\bot$.
	\end{enumerate}
	Positive formulas are built recursively from atomic formulas by taking:
	\begin{enumerate}
	    \setcounter{enumi}{4}
		\item {\em conjunctions}: if $\varphi_1$ and $\varphi_2$ are $X$-ary formulas then
		$$\phi_1\wedge\phi_2$$ is an $X$-ary formula;
		\item {\em existential quantification}: if $\psi(x,y)$ is an $(X+Y)$-ary formula then
		$$ \exists y \psi(x,y) $$
		is an $X$-ary formula;
		\item {\em disjunctions}: if $\varphi_1$ and $\varphi_2$ are $X$-ary formulas then
		$$
		\varphi_1\vee \phi_2
		$$
		is an $X$-ary formula.
	\end{enumerate}
	The class of regular formulas is defined by only allowing rules (1)-(3), (5), and (6).
\end{definition}
Enriched logic has only been developed for positive logic. Interpreting full first-order logic in a category requires substantially more structure, even for intuitionistic semantics. Should an enriched version of full first-order logic be developed in the future then it is likely the techniques in this paper would extend to such a logic. For now, we stick with enriched positive logic.

\begin{definition}[{\cite[Definition~3.5]{RT25positive}}]
	For any $\LL$-structure $A$ and any positive $X$-ary formula $\varphi(x)$ we define recursively its interpretation in $A$ as an $\M$-subobject
	$ \varphi(A)\rightarrowtail A^X$.
	{\setlength{\leftmargini}{1.8em}
		\begin{enumerate}
			\item If $\varphi(x)\equiv(s(x)=t(x))$ then the interpretation $\phi(A)\rightarrowtail A^X$ is the $(\E,\M)$-factorisation of	the equaliser of $s_A$ and $t_A$:
			\begin{center}
				\begin{tikzpicture}[baseline=(current  bounding  box.south), scale=2]
					
					\node (a0) at (0,0) {$A_{s,t}$};
					\node (c0) at (1.8,0) {$ A^X $};
					\node (d0) at (0.9,0.4) {$\phi(A)$};
					
					\path[font=\scriptsize]
					
					(a0) edge [->] node [below] {$\operatorname{eq}(s_A,t_A)$} (c0)
					(a0) edge [->>] node [above] {$\E\ \ \ $} (d0)
					(d0) edge [>->] node [above] {$\ \ \ \M$} (c0);
				\end{tikzpicture}	
			\end{center}
			\item If $\varphi(z)\equiv R^Y(t(z))$ then the interpretation $\phi(A)\rightarrowtail A^Z$ is the pullback
			\begin{center}
				\begin{tikzpicture}[baseline=(current  bounding  box.south), scale=2]
					
					\node (a0) at (0,0.8) {$\phi(A)$};
					\node (a0') at (0.2,0.6) {$\lrcorner$};
					\node (b0) at (1,0.8) {$A^Z$};
					\node (c0) at (0,0) {$R_A^Y$};
					\node (d0) at (1,0) {$A^{X\otimes Y}$};
					
					\path[font=\scriptsize]
					
					(a0) edge [>->] node [above] {} (b0)
					(a0) edge [->] node [left] {} (c0)
					(b0) edge [->] node [right] {$t_A$} (d0)
					(c0) edge [>->] node [below] {$r_A^Y$} (d0);
				\end{tikzpicture}	
			\end{center}
			
			\item If $\phi(x)\equiv\top$ then $\top(A):=A^X\rightarrowtail A^X$ is the identity.
			
			\item If $\phi(x)\equiv\bot$ then $\bot(A)\rightarrowtail A^X$ is the $(\E,\M)$-factorisation
			\begin{center}
				\begin{tikzpicture}[baseline=(current  bounding  box.south), scale=2]
					
					\node (a0) at (0,0) {$0$};
					\node (c0) at (2,0) {$ A^X $};
					\node (d0) at (1,0.4) {$\bot(A)$};
					
					\path[font=\scriptsize]
					
					(a0) edge [->] node [below] {$!$} (c0)
					(a0) edge [->>] node [above] {$\E\ \ \ $} (d0)
					(d0) edge [>->] node [above] {$\ \ \ \M$} (c0);
				\end{tikzpicture}	
			\end{center}
			\item If $ \varphi(x)\equiv\phi_1(x)\wedge\phi_2(x)$ then $\phi(A)\rightarrowtail A^X$ is the pullback
			\begin{center}
				\begin{tikzpicture}[baseline=(current  bounding  box.south), scale=2]
					
					\node (a0) at (0,0.8) {$\phi(A)$};
					\node (a0') at (0.2,0.6) {$\lrcorner$};
					\node (b0) at (1,0.8) {$\phi_1(A)$};
					\node (c0) at (0,0) {$\phi_2(A)$};
					\node (d0) at (1,0) {$A^{X}$};
					
					\path[font=\scriptsize]
					
					(a0) edge [>->] node [above] {} (b0)
					(a0) edge [>->] node [left] {} (c0)
					(b0) edge [>->] node [right] {} (d0)
					(c0) edge [>->] node [below] {} (d0);
				\end{tikzpicture}	
			\end{center}
			\item If $\varphi(x)\equiv \exists y \psi(x,y)$ then $\phi(A)\rightarrowtail A^X$ is the $(\E,\M)$-factorisation 
			\begin{center}
				\begin{tikzpicture}[baseline=(current  bounding  box.south), scale=2]
					
					\node (a0) at (-0.1,0) {$\psi(A)$};
					\node (c0) at (1,0) {$ A^{X+Y} $};
					\node (c1) at (2.1,0) {$ A^X$};
					\node (d0) at (1,0.4) {$\phi(A)$};
					
					\path[font=\scriptsize]
					
					(a0) edge [>->] node [below] {} (c0)
					(c0) edge [->] node [below] {$\pi_X$} (c1)
					(a0) edge [->>] node [above] {$\E\ \ \ $} (d0)
					(d0) edge [>->] node [above] {$\ \ \ \M$} (c1);
				\end{tikzpicture}
			\end{center}

			\item if $\varphi(x)\equiv \phi_1(x)\vee\phi_2(x)$ then $\phi(A)\rightarrowtail A^X$ is the $(\E,\M)$-factorisation 
			\begin{center}
				\begin{tikzpicture}[baseline=(current  bounding  box.south), scale=2]
					
					\node (a0) at (0,0) {$\phi_1(A)+\phi_2(A)$};
					\node (c0) at (2,0) {$ A^X $};
					\node (d0) at (1,0.4) {$\phi(A)$};
					
					\path[font=\scriptsize]
					
					(a0) edge [->] node [below] {} (c0)
					(a0) edge [->>] node [above] {$\E\ \ \ $} (d0)
					(d0) edge [>->] node [above] {$\ \ \ \M$} (c0);
				\end{tikzpicture}	
			\end{center}
	\end{enumerate}}
\end{definition}

Note that, if $\M$ contains the class of the regular monomorphisms (equivalently: $(\E, \M)$ is proper) then the interpretation of $s(x)=t(x)$ in $A$ coincides with the equaliser of $s_A$ and $t_A$. Similarly, if every arrow out of the initial object is in $\M$, then the interpretation of $\bot$ is the initial object (this is always the case in our examples).

We can talk about satisfaction of formulas and sequents of such.

\begin{definition}
	Given $X$-ary formulas $\varphi$ and $\psi$, we say that an $\LL$-structure $A$ {\em satisfies the sequent} $
	 \varphi(x) \vdash \psi(x),$ and write
	$$ A\models( \varphi\vdash \psi), $$
	if $\phi(A)\leq \psi(A)$ as $\M$-subobjects of $A^X$. %We say that $A$ {\em satisfies the sequent pointwise} if $ A\models \sigma(a) $ for any $a\colon Y \to A^X$.
\end{definition}

When in a sequent $\varphi(x) \vdash \psi(x)$ as above we have $\phi\equiv\top$ we simply write $A\models \psi$, instead of $A\models( \top\vdash \psi)$, and this happens if and only if $\psi(A)\cong A^X$.

\begin{definition}[{\cite[Definition~3.4]{RT25ERegular}}]
	A {\em homomorphism of $\LL$-structures} $h\colon A\to B$ is determined by an arrow $h\colon A\to B$ in $\V$ making the left square below commute
	\begin{center}
		\begin{tikzpicture}[baseline=(current  bounding  box.south), scale=2]
			
			\node (a0) at (0,0.8) {$A^X$};
			\node (b0) at (1,0.8) {$B^X$};
			\node (c0) at (0,0) {$A^Y$};
			\node (d0) at (1,0) {$B^Y$};
			
			\path[font=\scriptsize]
			
			(a0) edge [->] node [above] {$h^X$} (b0)
			(a0) edge [->] node [left] {$f_A$} (c0)
			(b0) edge [->] node [right] {$f_B$} (d0)
			(c0) edge [->] node [below] {$h^Y$} (d0);
		\end{tikzpicture}
		\qquad\qquad
		\begin{tikzpicture}[baseline=(current  bounding  box.south), scale=2]
				
				\node (a0) at (0,0.8) {$R_A$};
				\node (b0) at (1,0.8) {$R_B$};
				\node (c0) at (0,0) {$A^X$};
				\node (d0) at (1,0) {$B^X$};
				
				\path[font=\scriptsize]
				
				(a0) edge [dashed, ->] node [above] {$h_R$} (b0)
				(a0) edge [>->] node [left] {$r_A$} (c0)
				(b0) edge [>->] node [right] {$r_B$} (d0)
				(c0) edge [->] node [below] {$h^X$} (d0);
			\end{tikzpicture}	 	
	\end{center} 
	for any $f\colon(X,Y)$ in $\LL$, and such that there exist a (necessarily unique) arrow $h_R\colon R_A\to R_B$ completing the commutative diagram above on the right for any relation symbol $R:X$ in $\LL$.
\end{definition}

The $\V$-category $\Str(\LL)$ of $\LL$-structures is defined in~\cite[Definition~3.8]{RT25ERegular} and its underlying category $\Str(\LL)_0$ has $\LL$-structures as objects and homomorphisms of such as arrows.

As shown in \cite[Proposition~4.1]{RT25positive}, every positive formula $\phi$ induces a functor
$$ \phi(-)\colon \Str(\LL)_0\to \V_0 $$
assigning to each $A\in\Str(\LL)_0$ the object $\phi(A)$ and to each homomorphism of structures $h\colon A\to B$ the (necessarily unique) arrow $\phi(h)$ making the square below commute
\[\begin{tikzcd}[sep=scriptsize]
	{\phi(A)} & {\phi(B)} \\
	{A^X} & {B^X}
	\arrow["{\phi(h)}", from=1-1, to=1-2]
	\arrow[tail, from=1-1, to=2-1]
	\arrow[tail, from=1-2, to=2-2]
	\arrow["{h^X}"', from=2-1, to=2-2]
\end{tikzcd}\]

\begin{definition}[{\cite[Definition~3.12]{RT25positive}}]
	Given a set $\TT$ of sequents of formulas, we denote by $\Mod(\TT)$ the full subcategory of $\Str(\LL)$ spanned by those $\LL$-structures that satisfy the sequents in $\TT$. We shall call such $\TT$ an {\em $\LL$-theory} and an element $M\in\Mod(\TT)$ a {\em model of $\TT$}.
\end{definition}

We also have a notion of pointwise satisfaction that will turn out to be useful later.

\begin{definition}
	\label{evaluate-formula-at-generalised-element}
	Let $A$ be an $\LL$-structure, $\phi$ be an $X$-ary formula, and $a\colon Y \to A^X$ be an arrow in $\V$. We write
	$$ A\models \phi(a) $$
	to mean that $a$ factors through the inclusion $\phi(A)\rightarrowtail A^X$.
	
	Given a sequent $\sigma\equiv (\phi\vdash \psi)$ and $a\colon Y\to A^X$, we write
	$$ A\models \sigma(a) $$
	if $ A\models \phi(a) $ implies $ A\models \psi(a) $.
\end{definition}

In \cite{RT25positive} the pointwise satisfaction above was considered only for the case $Y=I$.

\begin{definition}
	\label{x-satisfaction}
	Let $\X$ be a set of objects of $\V$. Given an $\LL$-structure $A$ and a sequent $\sigma$, we say that {\em $A$ $\X$-satisfies $\sigma$} if $A\models \sigma(a)$ for all arrows $a\colon Y\to A^X$ in $\V$ with domain in $\X$.
\end{definition}

\begin{proposition}
	\label{sequent-satisfaction-test-on-generators}
	Let $\X$ be an $\E$-generator of $\V_0$, $A$ be an $\LL$-structure, and $\sigma$ a sequent. Then $A \models \sigma$ if and only if $A$ $\X$-satisfies $\sigma$.
\end{proposition}
As will be evident from the proof, for the left to right direction we get in fact that $A$ then $\Ob(\V)$-satisfies $\sigma$. So the assumption about $\X$ being an $\E$-generator is only used in the other direction.
\begin{proof}
	Let $\phi$ and $\psi$ be $X$-ary formulas such that $\sigma \equiv (\phi \vdash \psi)$. If $A \models \sigma$ then any arrow $Y \to A^X$ that factors through $\phi(A)$ also factors through $\psi(A)$, as witnessed by composing the factorisation with the arrow $\phi(A) \to \psi(A)$ that witnesses $A \models \sigma$.
	
	For the other direction we have by assumption that for any $a\colon Y \to \phi(A)$, with $Y \in \X$, there is $f_a\colon Y \to \psi(A)$ such that composing with the arrows $\phi(A) \rightarrowtail A^X$ and $\psi(A) \rightarrowtail A^X$ yields a commutative square. Taking the coproduct we thus have a commutative square given by the solid arrows in the diagram below.
	% https://q.uiver.app/#q=WzAsNCxbMSwwLCJcXHBzaV5BIl0sWzEsMiwiXFxwaGleQSJdLFsyLDEsIkFeWCJdLFswLDEsIlxcc3VtX3thIFxcY29sb24gWSBcXHRvIFxccGhpXkEsIFkgXFxpbiBcXFh9IFkiXSxbMCwyLCIiLDAseyJzdHlsZSI6eyJ0YWlsIjp7Im5hbWUiOiJtb25vIn19fV0sWzEsMiwiIiwyLHsic3R5bGUiOnsidGFpbCI6eyJuYW1lIjoibW9ubyJ9fX1dLFszLDAsIlxcc3VtX3thIFxcY29sb24gWSBcXHRvIFxccGhpXkEsIFkgXFxpbiBcXFh9IGZfYSJdLFsxLDAsIiIsMCx7InN0eWxlIjp7ImJvZHkiOnsibmFtZSI6ImRhc2hlZCJ9fX1dLFszLDEsIlxcc3VtX3thIFxcY29sb24gWSBcXHRvIFxccGhpXkEsIFkgXFxpbiBcXFh9IGEiLDIseyJzdHlsZSI6eyJoZWFkIjp7Im5hbWUiOiJlcGkifX19XV0=
	\[\begin{tikzcd}[sep=small]
		& {\psi(A)} & \\
		{\sum_{a \colon Y \to \phi(A), Y \in \X} Y} && {A^X} \\
		& {\phi(A)}
		\arrow[tail, from=1-2, to=2-3]
		\arrow["{\sum_{a \colon Y \to \phi(A), Y \in \X} f_a}", from=2-1, to=1-2]
		\arrow["{\sum_{a \colon Y \to \phi(A), Y \in \X} a}"', two heads, from=2-1, to=3-2]
		\arrow[dashed, from=3-2, to=1-2]
		\arrow[tail, from=3-2, to=2-3]
	\end{tikzcd}\]
	As $\X$ is an $\X$-generator, the bottom left arrow is in $\E$. So we get the dashed diagonal filler arrow, making the entire diagram commute. The triangle on the right then witnesses that $A \models \sigma$, as required.
\end{proof}

The following shows how existential quantification behaves with respect to pointwise satisfaction.

\begin{proposition}
	\label{existential-quantification-and-realisations}
	Let $\phi(x)$ be an $X$-ary formula of the form $\exists y \psi(x, y)$, where $\psi(x, y)$ is $(X+Y)$-ary. Let also $A$ be an $\LL$-structure and let $a\colon Z \to A^X$ be an arrow in $\V$. If there is an arrow $b\colon Z \to A^Y$ such that $A \models \psi((a, b))$, where $(a, b)\colon Z \to A^X \times A^Y \cong A^{X+Y}$ is the induced arrow, then $A \models \phi(a)$.
	
	The converse holds if $Z$ is $\E$-projective.
\end{proposition}
\begin{proof}
If $A \models \psi((a, b))$ then the dashed arrow in the diagram below exists, making everything commute.
% https://q.uiver.app/#q=WzAsNSxbMCwwLCJaIl0sWzEsMCwiXFxwc2koQSkiXSxbMiwwLCJcXHBoaShBKSJdLFsxLDEsIkFee1grWX0iXSxbMiwxLCJBXlgiXSxbMCwzLCIoYSwgYikiLDJdLFswLDEsIiIsMCx7InN0eWxlIjp7ImJvZHkiOnsibmFtZSI6ImRhc2hlZCJ9fX1dLFsxLDMsIiIsMCx7InN0eWxlIjp7InRhaWwiOnsibmFtZSI6Im1vbm8ifX19XSxbMSwyLCIiLDAseyJzdHlsZSI6eyJoZWFkIjp7Im5hbWUiOiJlcGkifX19XSxbMyw0LCJcXHBpX1giLDJdLFsyLDQsIiIsMSx7InN0eWxlIjp7InRhaWwiOnsibmFtZSI6Im1vbm8ifX19XV0=
\[\begin{tikzcd}[sep=small]
	Z & {\psi(A)} & {\phi(A)} \\
	& {A^{X+Y}} & {A^X}
	\arrow[dashed, from=1-1, to=1-2]
	\arrow["{(a, b)}"', from=1-1, to=2-2]
	\arrow[two heads, from=1-2, to=1-3]
	\arrow[tail, from=1-2, to=2-2]
	\arrow[tail, from=1-3, to=2-3]
	\arrow["{\pi_X}"', from=2-2, to=2-3]
\end{tikzcd}\]
We thus see that $a = \pi_X (a, b)$ factors through $\phi(A)$, and so $A \models \phi(a)$.
	
For the converse, suppose that $Z$ is $\E$-projective and that $A \models \phi(a)$. Then the dashed arrow in the diagram below exists, such that its composition with $\phi(A) \rightarrowtail A^X$ is $a$.
% https://q.uiver.app/#q=WzAsNSxbMCwwLCJaIl0sWzEsMCwiXFxwc2leQSJdLFsyLDAsIlxccGhpXkEiXSxbMSwxLCJBXntYK1l9Il0sWzIsMSwiQV5YIl0sWzAsMSwiIiwwLHsic3R5bGUiOnsiYm9keSI6eyJuYW1lIjoiZG90dGVkIn19fV0sWzEsMywiIiwwLHsic3R5bGUiOnsidGFpbCI6eyJuYW1lIjoibW9ubyJ9fX1dLFsxLDIsIiIsMCx7InN0eWxlIjp7ImhlYWQiOnsibmFtZSI6ImVwaSJ9fX1dLFszLDQsIlxccGlfWCIsMl0sWzIsNCwiIiwxLHsic3R5bGUiOnsidGFpbCI6eyJuYW1lIjoibW9ubyJ9fX1dLFswLDIsIiIsMix7ImN1cnZlIjotMywic3R5bGUiOnsiYm9keSI6eyJuYW1lIjoiZGFzaGVkIn19fV1d
\[\begin{tikzcd}[sep=small]
	Z & {\psi^A} & {\phi^A} \\
	& {A^{X+Y}} & {A^X}
	\arrow[dotted, from=1-1, to=1-2]
	\arrow[curve={height=-18pt}, dashed, from=1-1, to=1-3]
	\arrow[two heads, from=1-2, to=1-3]
	\arrow[tail, from=1-2, to=2-2]
	\arrow[tail, from=1-3, to=2-3]
	\arrow["{\pi_X}"', from=2-2, to=2-3]
\end{tikzcd}\]
Since $Z$ is $\E$-projective we then find the dotted arrow, making the top triangle commute. The composition $Z \to \psi(A) \to A^{X+Y} \to A^Y$ is then the desired $b$ such that $A \models \psi((a, b))$.
\end{proof}

We conclude this section with a brief comment on substitution in the enriched framework.

\begin{remark}\label{substitution}
	If $\varphi(x)$ is an $X$-ary formula and $t$ is a $(W,X)$-ary term, then we can recursively define the $W$-ary {\em substitution formula} $\varphi(t(w))$ as follows.\begin{enumerate}
		\item If $\varphi(x)\equiv(s_1(x)=s_2(x))$ then $\varphi(t(w)):= (s_1(t)(w)= s_2(t)(w))$.
		\item If $\varphi(x)\equiv R^Y(s(x))$ then $\varphi(t(w)):= R^Y(s(t)(w)).$
		\item If $\varphi(x)\equiv \top$ then $\phi(t(w)):=\top$.
		\item If $\varphi(x)\equiv \bot$ then $\phi(t(w)):=\bot$.
		\item If $ \varphi(x)\equiv\phi_1(x)\wedge\phi_2(x)$ then
		$\varphi(t(w)):=\phi_1(t(w))\wedge\phi_2(t(w)).$
		\item If $\varphi(x)\equiv \exists y \psi(x,y)$ then 
		$\varphi(t(w)):=\exists y \psi(t(w),y).$
		\item If $ \varphi(x)\equiv\phi_1(x)\vee\phi_2(x)$ then
		$\varphi(t(w)):=\phi_1(t(w))\vee\phi_2(t(w)).$
	\end{enumerate}
	Given an $A$, $\varphi$, and $t$ as above, it is straightforward to see that $\varphi(t(A))$ is given by the pullback below
	\begin{center}
		\begin{tikzpicture}[baseline=(current  bounding  box.south), scale=2]
			
			\node (a0) at (0,0.8) {$\varphi(t(A))$};
			\node (a0') at (0.2,0.6) {$\lrcorner$};
			\node (b0) at (1,0.8) {$A^W$};
			\node (c0) at (0,0) {$\phi(A)$};
			\node (d0) at (1,0) {$A^X$};
			
			\path[font=\scriptsize]
			
			(a0) edge [>->] node [above] {} (b0)
			(a0) edge [->] node [left] {} (c0)
			(b0) edge [->] node [right] {$t_A$} (d0)
			(c0) edge [>->] node [below] {} (d0);
		\end{tikzpicture}	
	\end{center}
	provided that one of the following holds:\begin{itemize}
		\item $\phi$ is a conjunction of atomic formulas (including $\bot$) and $\E$ is contained in the class of epimorphisms or is pullback stable;
		\item $\phi$ is regular and $\E$ is pullback stable;
		\item $\phi$ is positive, $\E$ is pullback stable, and $\V_0$ is finitary extensive.
	\end{itemize}
	In this context, given any $a\colon Y\to A^W$, it follows that  that 
	\begin{center}
		$ A\models \phi(t(a))$ \qquad if and only if \qquad$A\models \phi(t_A\circ a)$.
	\end{center}
\end{remark}

\section{Reduced products and ultraproducts}
\label{reduced-products-and-ultraproducts}
To even state a {\PolishL}os-style theorem, we need to make sense of ultraproducts of structures in the enriched sense. More generally, we will define reduced products of such structures, which are defined with respect to a filter rather than an ultrafilter and interact well with the regular fragment of logic.
\begin{definition}
\label{ultrafilter}
Let $\I$ be a set. A \emph{filter} on $\I$ is a nonempty family $\F \subseteq \P(\I)$ of subsets of $\I$ such that:
\begin{enumerate}[label=(\arabic*)]
\item $\emptyset \not \in \F$,
\item if $J \in \F$ and $J \subseteq J'$ then $J' \in \F$,
\item\label{filter-finite-intersections} if $J, J' \in \F$ then $J \cap J' \in \F$.
\end{enumerate}
We call a filter $\U$ on $\I$ an \emph{ultrafilter} on $\I$ if furthermore for all $J \subseteq \I$ either $J \in \U$ or $\I \setminus J \in \U$.
\end{definition}
Ultrafilters can equivalently be characterised as maximal filters. That is, those filters $\U$ such that if $\F \supseteq \U$ is a filter then $\F = \U$.
\begin{definition}
\label{reduced-product}
Let $\C$ be a complete and cocomplete category and let $\{A_i\}_{i \in \I}$ be a family of objects in $\C$. For any filter $\F$ on $\I$ we define the \emph{reduced product} to be the following directed colimit:
\[
\prod_{i \in \I} A_i / \F := \colim_{J \in \F} \prod_{i \in J} A_i,
\]
where $\F$ is ordered reversely by inclusion. The arrows in the diagram are the obvious projections $\prod_{i \in J} A_i \to \prod_{i \in J'} A_i$ for $J \supseteq J'$.

If $\F$ is an ultrafilter then we call $\prod_{i \in \I} A_i / \F$ the \emph{ultraproduct} instead.
\end{definition}
Since $\Str(\LL)_0$ is complete and cocomplete, we can always form reduced products of $\LL$-structures. However, to obtain a nice explicit description we will need our base of enrichment $\V$ to satisfy a few properties.
\begin{definition}
\label{reduced-product-ready}
Our fixed base of enrichment $\V$ with factorisation system $(\E, \M)$ is called \emph{reduced product ready} if the following hold.
	\begin{enumerate}[label=(\arabic*)]
		\item \label{reduced-product-ready:finite-limits-directed-colimits-commute} Finite limits commute in $\V$ with directed colimits.
		\item \label{reduced-product-ready:finite-powers-directed-colimits-commute} Finite powers commute in $\V$ with directed colimits.
		\item \label{reduced-product-ready:m-closed-under-directed-colimits} The class $\M$ is closed in $\V^\bbtwo$ under directed colimits.
	\end{enumerate}
\end{definition}
The following equivalent condition is sometimes easier to check than \cref{reduced-product-ready}\ref{reduced-product-ready:finite-powers-directed-colimits-commute}
\begin{fact}[{\cite[(5.2)]{Kel82:articolo}}]
\label{finite-powers-directed-colimits-commute-iff-finites-closed-under-tensor}
If $\V_0$ is locally finitely presentable then finite powers commute in $\V_0$ with directed colimits if and only if the finitely presentable objects in $\V_0$ are closed under the monoidal product.
\end{fact}
If $\V$ is reduced product ready then the reduced product of arrows in $\M$ is again in $\M$. More precisely, given a family $\{f_i\colon A_i \to B_i\}_{i \in \I}$ of arrows in $\M$ and a filter $\F$ on $\I$, the induced arrow
\[
\prod_{i \in \I} f_i / \F \colon \prod_{i \in \I} A_i / \F \rightarrowtail \prod_{i \in \I} B_i / \F
\]
is again in $\M$. This follows because such a reduced product is a directed colimit of products, and $\M$ is closed under limits (being the right part of a factorisation system) and under directed colimits (by assumption, \cref{reduced-product-ready}\ref{reduced-product-ready:m-closed-under-directed-colimits}). Similarly, we have for any finitely presentable $X$ that
\[
\left( \prod_{i \in \I} A_i / \F \right)^X \cong \prod_{i \in \I} A_i^X / \F,
\]
because powers always commute with limits and finite powers commute with directed colimits (by assumption, \cref{reduced-product-ready}\ref{reduced-product-ready:finite-powers-directed-colimits-commute}). Finally, finite limits commute with reduced products in a similar manner. We will often use these facts implicitly from now on.
\begin{definition}
\label{reduced-product-of-l-structures}
Assume that $\V$ is reduced product ready. Let $\LL$ be a language on $\V$, let $\{A_i\}_{i \in \I}$ be a family of $\LL$-structures, and let $\F$ be a filter on $\I$. We define the \emph{reduced product} $$A := \prod_{i \in \I} A_i / \F$$ to be the following $\LL$-structure. Its underlying object is $\prod_{i \in \I} A_i / \F$, as in \cref{reduced-product}. For each $(X, Y)$-ary function symbol $f$ in $\LL$ we let its interpretation $f_A \colon A^X \to A^Y$ be the composition
\[
A^X \cong \prod_{i \in \I} A_i^X / \F \xrightarrow{\ \prod_{i \in \I} f_{A_i} / \F\ } \prod_{i \in \I} A_i^Y / \F \cong A^Y.
\]
For each $X$-ary relation symbol $R$ in $\LL$ we let its interpretation $R_A$ be the $\M$-subobject represented by
\[
\prod_{i \in \I} R_{A_i} / \F \rightarrowtail \prod_{i \in \I} A_i^X / \F \cong A^X,
\]
which formally is the reduced product of representatives of the family of $\M$-subobjects $\{R_{A_i}\}_{i \in \I}$.
\end{definition}

\begin{proposition}[interpretations commute with reduced products]
\label{reduced-product-interpretations-as-expected}
Assume that $\V$ is reduced product ready. Let $\LL$ be a language and let $A := \prod_{i \in \I} A_i / \F$ be a reduced product of $\LL$-structures. Then:
\begin{enumerate}[label=(\roman*)]
\item \label{reduced-product-interpretations-as-expected:terms} for any $(X, Y)$-ary term $t$ in $\LL$ its interpretation $t_A$ is given by the composition
\[
A^X \cong \prod_{i \in \I} A_i^X / \F \xrightarrow{\prod_{i \in \I} t_{A_i} / \F} \prod_{i \in \I} A_i^Y / \F \cong A^Y,
\]
\item \label{reduced-product-interpretations-as-expected:relations} for any $X$-ary relation symbol $R$ and any $(Z, X \otimes Y)$-ary term $t$ in $\LL$ the interpretation of $R^Y(t)$ in $A$ is given by the composition
\[
\prod_{i \in \I} R^Y(t(A_i)) / \F \rightarrowtail \prod_{i \in \I} A_i^Z / \F \cong A^Z.
\]
\end{enumerate}
\end{proposition}
\begin{proof}
Item \ref{reduced-product-interpretations-as-expected:terms} is proved by induction on the complexity of the term $t$.

The case where $t$ is an arrow $f\colon Y \to X$ in $\V_\fp$ follows because finite powers commute with reduced products and $(-)^f \colon (-)^X \to (-)^Y$ is a natural transformation. The case where $t$ is a function symbol of arity $(X, Y)$ in $\LL$ follows by definition of the $\LL$-structure $A$.

For the case where $t$ is of the form $s^Z$ we compose isomorphisms as below.
\[\begin{tikzcd}[column sep=12em]
	{A^{X \otimes Z}} & {A^{Y \otimes Z}} \\
	{(A^X)^Z} & {(A^Y)^Z} \\
	{\left( \prod_{i \in \I} A_i^X/\F \right)^Z} & {\left( \prod_{i \in \I} A_i^Y/\F \right)^Z} \\
	{\prod_{i \in \I} (A_i^X)^Z /\F} & {\prod_{i \in \I} (A_i^Y)^Z /\F} \\
	{\prod_{i \in \I} A_i^{X \otimes Z} /\F} & {\prod_{i \in \I} A_i^{Y \otimes Z} /\F}
	\arrow["{s_A^Z}"{description}, from=1-1, to=1-2]
	\arrow[""{name=0, anchor=center, inner sep=0}, "\cong"{description}, from=1-1, to=2-1]
	\arrow[""{name=1, anchor=center, inner sep=0}, "\cong"{description}, from=1-2, to=2-2]
	\arrow["{(s_A)^Z}"{description}, from=2-1, to=2-2]
	\arrow[""{name=2, anchor=center, inner sep=0}, "\cong"{description}, from=2-1, to=3-1]
	\arrow[""{name=3, anchor=center, inner sep=0}, "\cong"{description}, from=2-2, to=3-2]
	\arrow["{\left( \prod_{i \in \I} s_{A_i} / \F \right)^Z}"{description}, from=3-1, to=3-2]
	\arrow[""{name=4, anchor=center, inner sep=0}, "\cong"{description}, from=3-1, to=4-1]
	\arrow[""{name=5, anchor=center, inner sep=0}, "\cong"{description}, from=3-2, to=4-2]
	\arrow["{\prod_{i \in \I} (s_{A_i})^Z / \F}"{description}, from=4-1, to=4-2]
	\arrow[""{name=6, anchor=center, inner sep=0}, "\cong"{description}, from=4-1, to=5-1]
	\arrow[""{name=7, anchor=center, inner sep=0}, "\cong"{description}, from=4-2, to=5-2]
	\arrow["{\prod_{i \in \I} s_{A_i}^Z / \F}"{description}, from=5-1, to=5-2]
	\arrow["{\text{definition of } s_A^Z}"{description}, draw=none, from=0, to=1]
	\arrow["{\text{induction hypothesis}}"{description}, draw=none, from=2, to=3]
	\arrow["{\text{finite powers commute with directed colimits}}"{description}, draw=none, from=4, to=5]
	\arrow["{\text{definition of } s_{A_i}^Z}"{description}, draw=none, from=6, to=7]
\end{tikzcd}\]
The outer rectangle is now the required isomorphism, where the vertical isomorphisms compose to the canonical isomorphism by commutativity of the square
% https://q.uiver.app/#q=WzAsNCxbMCwwLCJBXntYIFxcb3RpbWVzIFp9Il0sWzEsMCwiKEFeWCleWiJdLFswLDEsIlxccHJvZF97aSBcXGluIFxcSX0gQV9pXntYIFxcb3RpbWVzIFp9IC8gXFxGIl0sWzEsMSwiXFxwcm9kX3tpIFxcaW4gXFxJfSAoQV9pXlgpXlogLyBcXEYiXSxbMCwxLCJcXGNvbmciLDFdLFsyLDAsIlxcY29uZyIsMV0sWzIsMywiXFxjb25nIiwxXSxbMywxLCJcXGNvbmciLDFdXQ==
\[\begin{tikzcd}
	{A^{X \otimes Z}} & {(A^X)^Z} \\
	{\prod_{i \in \I} A_i^{X \otimes Z} / \F} & {\prod_{i \in \I} (A_i^X)^Z / \F}
	\arrow["\cong"{description}, from=1-1, to=1-2]
	\arrow["\cong"{description}, from=2-1, to=1-1]
	\arrow["\cong"{description}, from=2-1, to=2-2]
	\arrow["\cong"{description}, from=2-2, to=1-2]
\end{tikzcd}\]
which in turn follows from naturality of the isomorphism $(-)^{X \otimes Z} \cong ((-)^X)^Z$ (and similar for $Y$).

Finally, we consider the case where $t$ is a composition of terms. More precisely, there is a finite family of terms $\{s_j \colon (X_j, Y_j)\}_{j \in J}$ and a $(\sum_{j \in J} Y_j, Z)$-ary term $r$ such that $t$ is $r((s_j)_{j \in J})$. We again construct a commutative diagram such that the outer rectangle is the required isomorphism. To save some space, we simplify notation by setting $X := \sum_{j \in J} X_j$ and $Y := \sum_{j \in J} Y_j$.
\[\begin{tikzcd}[column sep=scriptsize]
	{A^X} && {A^Z} \\
	{\prod_{j \in J} A^{X_j}} & {\prod_{j \in J} A^{Y_j} \ \cong \ A^Y} & {A^Z} \\
	{\prod_{j \in J} \left( \prod_{i \in \I} A_i^{X_j} / \F \right)} & {\prod_{j \in J} \left( \prod_{i \in \I} A_i^{Y_j} / \F \right) \ \prod_{i \in \I} A_i^Y / \F} & {\prod_{i \in \I} A_i ^Z / \F} \\
	{\prod_{i \in \I} \left(\prod_{j \in J} A_i^{X_j} \right) / \F} & {\prod_{i \in \I} \left(\prod_{j \in J} A_i^{Y_j} \right) / \F \cong \prod_{i \in \I} A_i^Y / \F} & {\prod_{i \in \I} A_i ^Z / \F} \\
	{\prod_{i \in \I} A^X / \F} && {\prod_{i \in \I} A_i ^Z / \F}
	\arrow["{r((s_j)_{j \in J})_A}"{description}, from=1-1, to=1-3]
	\arrow[""{name=0, anchor=center, inner sep=0}, "\cong"{description}, from=1-1, to=2-1]
	\arrow[""{name=1, anchor=center, inner sep=0}, equals, nfold, from=1-3, to=2-3]
	\arrow["{(s_{j,A})_{j \in J}}"{description}, from=2-1, to=2-2]
	\arrow[""{name=2, anchor=center, inner sep=0}, "\cong"{description}, from=2-1, to=3-1]
	\arrow["{r_A}"{description}, from=2-2, to=2-3]
	\arrow[""{name=3, anchor=center, inner sep=0}, "\cong"{description}, shift left=12, from=2-2, to=3-2]
	\arrow[""{name=4, anchor=center, inner sep=0}, "\cong"{description}, shift right=5, from=2-2, to=3-2]
	\arrow[""{name=5, anchor=center, inner sep=0}, "\cong"{description}, from=2-3, to=3-3]
	\arrow["{\left( \prod_{i \in \I} s_{j,A_i} / \F \right)}"{pos=0.8}, shift left=2, from=3-1, to=3-2]
	\arrow[""{name=6, anchor=center, inner sep=0}, "\cong"{description}, from=3-1, to=4-1]
	\arrow["{\prod_{i \in \I} r_{A_i} / \F}", shift left=2, from=3-2, to=3-3]
	\arrow[""{name=7, anchor=center, inner sep=0}, "\cong"{description}, shift right=5, from=3-2, to=4-2]
	\arrow[shift left=12, equals, nfold, from=3-2, to=4-2]
	\arrow[equals, nfold, from=3-3, to=4-3]
	\arrow["{\prod_{i \in \I} (s_{j, A_i})_{j \in J} / \F}"'{pos=0.8}, shift right=2, from=4-1, to=4-2]
	\arrow[""{name=8, anchor=center, inner sep=0}, "\cong"{description}, from=4-1, to=5-1]
	\arrow["{\prod_{i \in \I} r_{A_i} / \F}"', shift right=2, from=4-2, to=4-3]
	\arrow[""{name=9, anchor=center, inner sep=0}, equals, nfold, from=4-3, to=5-3]
	\arrow["{\prod_{i \in \I} r((s_j)_{j \in J})_{A_i} / \F}"{description}, from=5-1, to=5-3]
	\arrow["{\text{definition of } r((s_j)_{j \in J})_A}"{description}, draw=none, from=0, to=1]
	\arrow["{\text{induction hypothesis}}"{description}, draw=none, from=2, to=4]
	\arrow["{\text{induction hypothesis}}"{description}, draw=none, from=3, to=5]
	\arrow["{\genfrac{}{}{0pt}{1}{\text{finite limits commute}}{\text{with reduced products}}}"{description}, draw=none, from=6, to=7]
	\arrow["{\text{definition of } \prod_{i \in \I} r((s_j)_{j \in J})_{A_i} / \F}"{description}, draw=none, from=8, to=9]
\end{tikzcd}\]
The commutativity of the unlabelled rectangle in the middle column follows by writing out the definitions of the isomorphisms involved.

For item \ref{reduced-product-interpretations-as-expected:relations} we consider the commutative cube below.
\begin{center}
	\begin{tikzpicture}[baseline=(current  bounding  box.south), scale=2, on top/.style={preaction={draw=white,-,line width=#1}}, on top/.default=6pt]
		
		\node (a1) at (0,1.1) {$R^Y(t(A))$};
		\node (b1) at (2.5,1.1) {$A^Z$};
		\node (a0) at (0,0) {$(R_A)^Y$};
		\node (b0) at (2.5,0) {$A^{X \otimes Y}$};
		
		\node (a1') at (1.5,1.7) {$\prod_{i \in \I} R^Y(t({A_i})) / \F$};
		\node (b1') at (4,1.7) {$\prod_{i \in \I} A_i^Z / \F$};
		\node (a0') at (1.5,0.7) {$\prod_{i \in \I} (R_{A_i})^Y / \F$};
		\node (b0') at (4,0.7) {$\prod_{i \in \I} A_i^{X \otimes Y}$};
		
		\path[font=\scriptsize]
		
		(a1') edge [>->] node [above] {} (b1')
		(a0') edge [>->] node [below] {} (b0')
		(a1') edge [->] node [left] {} (a0')
		(b1') edge [->] node [right] {$\prod_{i \in \I} t_{A_i} / \F$} (b0')
		
		(a1) edge [>->, on top] node [above] {} (b1)
		(a0) edge [>->] node [below] {} (b0)
		(a1) edge [->] node [left] {} (a0)
		(b1) edge [->, on top] node [below] {$\ \ \ \ \ t_A$} (b0)
		
		(a1) edge [dashed, ->] node [above] {$\cong\ $} (a1')
		(b1) edge [->] node [above] {$\cong\ $} (b1')
		(a0) edge [->] node [above] {$\cong\ $} (a0')
		(b0) edge [->] node [above] {$\cong\ $} (b0');
	\end{tikzpicture}	
\end{center}
The bottom left arrow is an isomorphism by definition of $A$ as an $\LL$-structure and the fact that $Y$ is finitely presentable, making the bottom face commute. The right face commutes by item \ref{reduced-product-interpretations-as-expected:terms}. The front face is a pullback by definition, and the back face is also a pullback since finite limits commute with reduced products. We thus find the dashed isomorphism, as required.
\end{proof}

\section{\texorpdfstring{{\PolishL}os-style}{Łos-style} theorems}
\label{sec:los}
In this section we prove {\PolishL}os-style theorems for two fragments of enriched logic: one for enriched regular logic and reduced products, and one for enriched positive logic and ultraproducts. To make sense of the reduced products (and ultraproducts), our base of enrichment needs to be at least reduced product ready. However, we need a little more.

\begin{definition}
\label{reduced-product-stable}
Our fixed base of enrichment $\V$ with factorisation system $(\E, \M)$ is called \emph{reduced product stable} if it is reduced product ready and the class $\E$ is closed in $\V^\bbtwo$ under small products.
\end{definition}
\begin{definition}
\label{ultraproduct-covering-property}
A complete and cocomplete category $\C$ with factorisation system $(\E, \M)$ is said to have the \emph{ultraproduct covering property} if for any set $\I$, any ultrafilter $\U$ on $\I$, and families $(A_i)_{i\in \I}$ and $(B_i)_{i\in \I}$ of objects in $\V$ the comparison arrows
\[
0\to \prod_{i\in \I}0/\U \qquad \text{and}\qquad   \prod_{i\in \I}A_i/\U + \prod_{i\in \I}B_i/\U \longrightarrow \prod_{i\in \I}(A_i+B_i)/\U
\]
are in $\E$.
\end{definition}
In \cref{applications-and-examples} we will see many examples of enrichment bases that are reduced product stable and satisfy the ultraproduct covering property, but we also give an example of an enrichment base that is reduced product stable but fails the ultraproduct covering property (\cref{sheaves-cantor-space-fails-ultraproduct-covering-property}).

The reason for requiring $\U$ in the above definition to be an ultrafilter is that this property fails for general filters, see the example below. Furthermore, in \cref{sec:ultraproduct-covering-property} we give natural assumptions that imply the ultraproduct covering property, and it is evident from the proofs there how $\U$ being an ultrafilter is used.

\begin{example}
\label{reduced-filter-covering-property-is-false}
Let $\V = \Set$ with the usual factorisation system. Let $A_i = B_i = 1$ for all $i \in \omega$ and let $\F$ be the filter of cofinite subsets of $\omega$. Then $\prod_{i \in \omega} A_i / \F = \prod_{i \in \omega} B_i / \F = 1$. The elements of $\prod_{i \in \omega} (A_i + B_i) / \F$ are represented by $\omega$-long strings of $1$s and $0$s, i.e., elements of $2^\omega$. Two such strings $(x_i)_{i \in \omega}$ and $(y_i)_{i \in \omega}$ represent the same element if and only if they agree on cofinitely many positions. So $\prod_{i \in \omega} (A_i + B_i) / \F$ has cardinality continuum, and can never admit a surjection from $\prod_{i \in \omega} A_i / \F + \prod_{i \in \omega} B_i / \F = 2$.
\end{example}

\begin{remark}\label{ultrapower-of-0}
	Since the initial object is stable under colimits and so are the arrows in $\E$, the arrow $0\to \prod_{i\in \I}0/\U$ lies in $\E$ whenever the arrow $0\to \prod_{i\in J}0$ lies in $\E$ for any nonempty set $J$. Note that this arrow is always a split monomorphism (since we can postcompose it with any product projection to obtain the identity on $0$); thus, if it is in $\E$ and every arrow of $\E$ is an epimorphism, then it is automatically an isomorphism.

It follows that the first condition of the ultraproduct covering property holds whenever the initial object is stable under small products. This is the case for instance when $0$ is actually a zero object (that is, coincides with the terminal object) or when it is strictly initial (that is, every arrow into $0$ is an isomorphism).
\end{remark}

\begin{repeated-theorem}[\cref{los-theorem} \textnormal{({\PolishL}os-style theorem 1)}]
Assume that $\V$ is reduced product stable. Let $\LL$ be a language on $\V$ and let $\varphi$ be an $\LL$-formula. Let $A := \prod_{i \in \I} A_i/\F$ be a reduced product of $\LL$-structures. Then
\[
\phi(A) \cong \prod_{i\in \I}\phi(A_i)/\F
\]
provided that:
\begin{enumerate}[label=(\roman*)]
\item $\varphi$ is regular, or;
\item $\varphi$ is positive, $\F$ is an ultrafilter, and $\V$ has the ultraproduct covering property.
\end{enumerate}
\end{repeated-theorem}
\begin{proof}
For a finitely presentable $X$ in $\V$, write
\[
\rho^X\colon A^X\xrightarrow{\ \cong\ }\prod_{i \in \I} A_i^X/\F
\]
for the canonical isomorphism. We will construct an isomorphism $\rho^\phi$, by induction on the complexity of $\phi$, such that the square below commutes.
% https://q.uiver.app/#q=WzAsNCxbMCwwLCJcXHBoaV5BIl0sWzEsMCwiQV5YIl0sWzAsMSwiXFxwcm9kX3tpIFxcaW4gSX0gXFxwaGlee0FfaX0gLyBcXEYiXSxbMSwxLCJcXHByb2Rfe2kgXFxpbiBJfSBBXlggLyBcXEYiXSxbMSwzLCJcXHJob15YIl0sWzAsMiwiXFxyaG9eXFxwaGkiLDJdLFsyLDMsIiIsMix7InN0eWxlIjp7InRhaWwiOnsibmFtZSI6Im1vbm8ifX19XSxbMCwxLCIiLDAseyJzdHlsZSI6eyJ0YWlsIjp7Im5hbWUiOiJtb25vIn19fV1d
\[\begin{tikzcd}[sep=scriptsize]
	{\phi(A)} & {A^X} \\
	{\prod_{i \in \I} \phi(A_i) / \F} & {\prod_{i \in \I} A^X / \F}
	\arrow[tail, from=1-1, to=1-2]
	\arrow["{\rho^\phi}"', from=1-1, to=2-1]
	\arrow["{\rho^X}", from=1-2, to=2-2]
	\arrow[tail, from=2-1, to=2-2]
\end{tikzcd}\]
Part of the base case, namely when $\phi \equiv R^Y(t)$ for an $X$-ary relation symbol $R$ and $(Z, X \otimes Y)$-ary term $t$, is already given by \cref{reduced-product-interpretations-as-expected}\ref{reduced-product-interpretations-as-expected:relations}.

Then there are $\top$ and $\bot$. The former is trivial, and the latter we will treat later in the proof, as it is only relevant for \ref{los-theorem:positive}.

The final part of the base case is when $\phi \equiv (s = t)$ for some $(X, Y)$-ary terms $s$ and $t$. Then, for any $i\in \I$, the interpretation $\phi(A_i)$ is defined as the $(\E,\M)$-factorisation
\[
A_i^{s,t} \twoheadrightarrow \phi(A_i) \rightarrowtail A_i^X,
\]
where $A_i^{s,t}$ is the equaliser of $(s_{A_i},t_{A_i})$. By \cref{reduced-product-interpretations-as-expected}\ref{reduced-product-interpretations-as-expected:terms} we have that $s_A \cong \prod_{i \in \I} s_{A_i} / \F$ and $t_A \cong \prod_{i \in \I} t_{A_i} / \F$. So, since $\prod_{i \in \I} A^{s,t}_i / \F$ is the equaliser of $\prod_{i \in \I} s_{A_i} / \F$ and $\prod_{i \in \I} t_{A_i} / \F$ (as finite limits commute with reduced products), we find an isomorphism $\rho^{s,t}\colon A^{s,t} \to \prod_{i \in \I} A^{s,t}_i / \F$. Taking the reduced product of the $(\E, \M)$-factorisations componentwise we thus obtain the solid part of the following diagram
	\begin{center}
		\begin{tikzpicture}[baseline=(current  bounding  box.south), scale=2]
			
			\node (a) at (1,0.8) {$A^{s,t}$};
			\node (b) at (2.5,0.8) {$\phi(A)$};
			\node (c) at (4,0.8) {$A^X$};
			\node (d) at (1,0) {$\prod_{i \in \I} A_i^{s,t}/\F$};
			\node (e) at (2.5,0) {$\prod_{i \in \I} \phi(A_i)/\F$};
			\node (f) at (4,0) {$\prod_{i \in \I} A_i^X/\F$};
			
			\path[font=\scriptsize]
			
			(a) edge [->>] node [above] {$q$} (b)
			(b) edge [>->] node [above] {$m$} (c)
			(d) edge [->] node [below] {$\bar q$} (e)
			(e) edge [>->] node [below] {$\bar m$} (f)
			
			(a) edge [->] node [left] {$\rho^{s,t}$} (d)
			(a) edge [->] node [right] {$\cong$} (d)
			(b) edge [dashed, ->] node [left] {$\rho^\phi$} (e)
			(c) edge [->] node [right] {$\rho^X$} (f)
			(c) edge [->] node [left] {$\cong$} (f);
		\end{tikzpicture}	
	\end{center}
Since $\M$ is closed under reduced products, the arrow $\bar m$ is still in $\M$. By orthogonality of the factorisation system, there exists a unique $\rho^\phi$ as above making the squares commute. Since $\V$ is reduced product stable, the reduced product of arrows in $\E$ is again in $\E$, because $\E$ is closed under colimits (being the left part of a factorisation system) and small products (by assumption). So $\bar q$ is in $\E$, and $\rho^\phi$ is an isomorphism by the uniqueness (up to isomorphism) of the $(\E,\M)$-factorisations.

For the case where $\phi \equiv \phi_1 \wedge \phi_2$ we consider the following commutative cube.
\begin{center}
	\begin{tikzpicture}[baseline=(current  bounding  box.south), scale=2, on top/.style={preaction={draw=white,-,line width=#1}}, on top/.default=6pt]
		
		\node (a1) at (0,1.1) {$\phi(A)$};
		\node (b1) at (2.5,1.1) {$\phi_2(A)$};
		\node (a0) at (0,0) {$\phi_1(A)$};
		\node (b0) at (2.5,0) {$A^X$};
		
		\node (a1') at (1.5,1.7) {$\prod_{i \in \I} \phi(A_i) / \F$};
		\node (b1') at (4,1.7) {$\prod_{i \in \I} \phi_2(A_i) / \F$};
		\node (a0') at (1.5,0.7) {$\prod_{i \in \I} \phi_1(A_i) / \F$};
		\node (b0') at (4,0.7) {$\prod_{i \in \I} A_i^X / \F$};
		
		\path[font=\scriptsize]
		
		(a1') edge [>->] node [above] {} (b1')
		(a0') edge [>->] node [below] {} (b0')
		(a1') edge [>->] node [left] {} (a0')
		(b1') edge [>->] node [right] {} (b0')
		
		(a1) edge [>->, on top] node [above] {} (b1)
		(a0) edge [>->] node [below] {} (b0)
		(a1) edge [>->] node [left] {} (a0)
		(b1) edge [>->, on top] node [left] {} (b0)
		
		(a1) edge [dashed, ->] node [above] {$\rho^\phi\ $} (a1')
		(b1) edge [->] node [above] {$\rho^{\phi_2}\ $} (b1')
		(b1) edge [->] node [below] {$\ \cong$} (b1')
		(a0) edge [->] node [above] {$\rho^{\phi_1}\ $} (a0')
		(a0) edge [->] node [below] {$\ \cong$} (a0')
		(b0) edge [->] node [above] {$\rho^X\ $} (b0')
		(b0) edge [->] node [below] {$\ \cong$} (b0');
	\end{tikzpicture}	
\end{center}
The existence of $\rho^{\phi_1}$ and $\rho^{\phi_2}$, as well as the commutativity of the bottom and right faces, follows from the induction hypothesis. The front face is a pullback by definition. The back face is, again by definition, a reduced product of pullbacks. So since finite limits commute with reduced products, the back face is a pullback and we find the required isomorphism $\rho^\phi$.

The final case that needs to be checked for item \ref{los-theorem:regular} is where $\phi\equiv \exists y \psi(x,y)$. We consider the diagram below whose horizontal arrows are induced by those defining $\phi(A_i)$ and $\phi(A)$ as $(\E,\M)$-factorisations. 
	\begin{center}
		\begin{tikzpicture}[baseline=(current  bounding  box.south), scale=2]
			
			\node (a) at (1,0.8) {$\psi(A)$};
			\node (b) at (2.5,0.8) {$\phi(A)$};
			\node (c) at (4,0.8) {$A^X$};
			\node (d) at (1,0) {$\prod_{i \in \I} \psi(A_i)/\F$};
			\node (e) at (2.5,0) {$\prod_{i \in \I} \phi(A_i)/\F$};
			\node (f) at (4,0) {$\prod_{i \in \I} A_i^X/\F$};
			
			\path[font=\scriptsize]
			
			(a) edge [->>] node [above] {} (b)
			(b) edge [>->] node [above] {} (c)
			(d) edge [->] node [below] {$r$} (e)
			(e) edge [>->] node [below] {$n$} (f)
			
			(a) edge [->] node [left] {$\rho^{\psi}$} (d)
			(a) edge [->] node [right] {$\cong$} (d)
			(b) edge [dashed, ->] node [left] {$\rho^\phi$} (e)
			(c) edge [->] node [right] {$\rho^X$} (f)
			(c) edge [->] node [left] {$\cong$} (f);
		\end{tikzpicture}	
	\end{center} 
Since $n\in\M$, by orthogonality, there is an arrow $\rho^\phi$ as above. As before, we also have $r \in \E$, and so $\rho^\phi$ is an isomorphism (see the case where $\phi \equiv (s = t)$).

Finally, with the additional hypotheses of \ref{los-theorem:positive}, we are left with two more cases, namely where $\phi\equiv\bot$ and  $\phi\equiv \phi_1\vee\phi_2$. 

When $\phi\equiv\bot$, of arity $X$, we have again an induced diagram as follows.
\begin{center}
	\begin{tikzpicture}[baseline=(current  bounding  box.south), scale=2]
		
		\node (a) at (1,0.8) {$0$};
		\node (b) at (2.4,0.8) {$\bot(A)$};
		\node (c) at (3.9,0.8) {$A^X$};
		\node (d) at (1,0) {$\prod_{i \in \I} 0/\F$};
		\node (e) at (2.4,0) {$\prod_{i \in \I} \bot(A_i)/\F$};
		\node (f) at (3.9,0) {$\prod_{i \in \I} A_i^X/\F$};
		
		\path[font=\scriptsize]
		
		(a) edge [->>] node [above] {} (b)
		(b) edge [>->] node [above] {} (c)
		(d) edge [->>] node [below] {$r$} (e)
		(e) edge [>->] node [below] {$n$} (f)
		
		(a) edge [->>] node [left] {} (d)
		(b) edge [dashed, ->] node [left] {$\rho^\bot$} (e)
		(c) edge [->] node [right] {$\rho^X$} (f)
		(c) edge [->] node [left] {$\cong$} (f);
	\end{tikzpicture}	
\end{center} 
Again, $r$ is in $\E$ and $n$ in $\M$. By the ultraproduct covering property the left vertical arrow is also in $\E$, so $\rho^\bot$ must be an isomorphism.

To conclude consider the case where $\phi\equiv \phi_1\vee\phi_2$; we have an induced diagram 
	\begin{center}
		\begin{tikzpicture}[baseline=(current  bounding  box.south), scale=2]
			
			\node (0) at (-1.4,0.6) {$\prod\limits_{i \in \I} \phi_1(A_i)/\F+\prod\limits_{i \in \I} \phi_2(A_i)/\F$};
			\node (a) at (0,1.2) {$\phi_1(A)+\phi_2(A)$};
			\node (b) at (1.9,1.2) {$\phi(A)$};
			\node (c) at (3.4,1.2) {$A^X$};
			\node (d) at (0,0) {$\prod\limits_{i \in \I}( \phi_1(A_i)+\phi_2(A_i))/\F$};
			\node (e) at (1.9,0) {$\prod\limits_{i \in \I} \phi(A_i)/\F$};
			\node (f) at (3.4,0) {$\prod\limits_{i \in \I} A_i^X/\F$};
			
			\path[font=\scriptsize]
			
			(a) edge [->>] node [above] {} (b)
			(b) edge [>->] node [above] {} (c)
			(d) edge [->] node [below] {$r$} (e)
			(e) edge [>->] node [below] {$n$} (f)
			
			(a) edge [->] node [left] {} (d)
			(a) edge [->] node [right] {} (d)
			(b) edge [dashed, ->] node [left] {$\rho^\phi$} (e)
			(c) edge [->] node [right] {$\rho^X$} (f)
			(c) edge [->] node [left] {$\cong$} (f)
			
			(a) edge [->] node [above] {$\rho^{\phi_1}+\rho^{\phi_2}\ \hspace{30pt} $} (0)
			(0) edge [->] node [left] {$c\ \ $} (d);
		\end{tikzpicture}	
	\end{center} 
	where $\rho^{\phi_1}+\rho^{\phi_2}$ is invertible by inductive hypothesis, $c$ is the unique comparison arrow, and $\rho^\phi$ is again induced by orthogonality. As before, we have that $r \in \E$. By the ultraproduct covering property we also have $c \in \E$. So $\rho^\phi$ is an isomorphism by uniqueness of $(\E,\M)$-factorisations.
\end{proof}

\begin{lemma}
\label{reduced-product-factor-on-big-subset}
    Let $\{A_i\}_{i \in \I}$ and $\{B_i\}_{i \in \I}$ be two families of objects in a complete and cocomplete category $\C$, indexed by the same set $\I$ and let $\F$ be a filter on $\I$. Suppose that for some $J \in \F$ there is a family of arrows $\{f_i \colon A_i \to B_i\}_{i \in J}$. Then there is a unique induced arrow $f\colon \prod_{i \in \I} A_i / \F \to \prod_{i \in \I} B_i / \F$ such that the diagram below commutes
	% https://q.uiver.app/#q=WzAsNCxbMCwxLCJcXHByb2Rfe2kgXFxpbiBJfSBBX2kgLyBcXFUiXSxbMSwxLCJcXHByb2Rfe2kgXFxpbiBJfSBCX2kgLyBcXFUiXSxbMCwwLCJcXHByb2Rfe2kgXFxpbiBKfSBBX2kiXSxbMSwwLCJcXHByb2Rfe2kgXFxpbiBKfSBCX2kiXSxbMCwxLCJmIiwyXSxbMiwwXSxbMiwzLCJcXHByb2Rfe2kgXFxpbiBKfSBmX2kiXSxbMywxXV0=
	\[\begin{tikzcd}
		{\prod_{i \in J} A_i} & {\prod_{i \in J} B_i} \\
		{\prod_{i \in \I} A_i / \F} & {\prod_{i \in \I} B_i / \F}
		\arrow["{\prod_{i \in J} f_i}", from=1-1, to=1-2]
		\arrow[from=1-1, to=2-1]
		\arrow[from=1-2, to=2-2]
		\arrow["f"', from=2-1, to=2-2]
	\end{tikzcd}\]
    where the vertical arrows are the coprojections.
\end{lemma}
\begin{proof}
The compositions
\[
\left( \prod_{i \in J'} A_i \xrightarrow{\prod_{i \in J'} f_i} \prod_{i \in J'} B_i \to \prod_{i \in \I} B_i / \F \right)_{J' \in \F, J' \subseteq J},
\]
form a cocone for the diagram of $\prod_{i \in \I} A_i / \F$ restricted to those $J' \in \F$ with $J' \subseteq J$. This restriction is cofinal in the original diagram of $\prod_{i \in \I} A_i / \F$ and thus induces a unique arrow $\prod_{i \in \I} A_i / \F \to \prod_{i \in \I} B_i / \F$, which is our required $f$.
\end{proof}

\begin{repeated-theorem}[\cref{los-sequents} \textnormal{({\PolishL}os-style theorem 2)}]
Assume that $\V$ is reduced product stable. Let $\LL$ be a language on $\V$ and let $\sigma$ be a sequent of $\LL$-formulas. Let $A := \prod_{i \in \I} A_i/\F$ be a reduced product of $\LL$-structures. Then
\[
\{i \in \I \mid A_i \models \sigma \} \in \F \quad \text{implies} \quad A \models \sigma,
\]
provided that:
\begin{enumerate}[label=(\roman*)]
\item $\sigma$ is a sequent of regular formulas, or;
\item $\sigma$ is a sequent of positive formulas, $\F$ is an ultrafilter, and $\V$ has the ultraproduct covering property.
\end{enumerate}
\end{repeated-theorem}
\begin{proof}
Let $\phi$ and $\psi$ be the $X$-ary formulas such that $\sigma$ is the sequent $\phi \vdash \psi$. Write $J := \{i \in \I \mid A_i \models \sigma\}$ and assume $J \in \F$. For any $i \in J$ there is an arrow $f_i\colon \phi(A_i) \to \psi(A_i)$ making the diagram below commute.
% https://q.uiver.app/#q=WzAsMyxbMCwwLCJcXHBzaV57QV9pfSJdLFswLDEsIlxccGhpXntBX2l9Il0sWzEsMCwiQV9pXlgiXSxbMCwyLCIiLDAseyJzdHlsZSI6eyJ0YWlsIjp7Im5hbWUiOiJtb25vIn19fV0sWzEsMiwiIiwyLHsic3R5bGUiOnsidGFpbCI6eyJuYW1lIjoibW9ubyJ9fX1dLFsxLDAsImZfaSJdXQ==
\[\begin{tikzcd}[sep=small]
	{\psi(A_i)} & {A_i^X} \\
	{\phi(A_i)}
	\arrow[tail, from=1-1, to=1-2]
	\arrow["{f_i}", from=2-1, to=1-1]
	\arrow[tail, from=2-1, to=1-2]
\end{tikzcd}\]
By \cref{reduced-product-factor-on-big-subset} we thus find the arrow $f$ in the diagram below, making the square on the left commute.
% https://q.uiver.app/#q=WzAsNixbMCwxLCJcXHByb2Rfe2kgXFxpbiBKfSBcXHBzaV57QV9pfSJdLFswLDMsIlxccHJvZF97aSBcXGluIEp9IFxccGhpXntBX2l9Il0sWzMsMSwiXFxwcm9kX3tpIFxcaW4gSn0gQV9pXlgiXSxbMSwwLCJcXHByb2Rfe2kgXFxpbiBJfSBcXHBzaV57QV9pfSAvIFxcRiBcXGNvbmcgXFxwc2leQSJdLFsxLDIsIlxccHJvZF97aSBcXGluIEl9IFxccGhpXntBX2l9IC8gXFxGIFxcY29uZyBcXHBoaV5BIl0sWzQsMCwiXFxwcm9kX3tpIFxcaW4gSX0gQV9pIC8gXFxGIFxcY29uZyBBXlgiXSxbMCwyLCIiLDAseyJzdHlsZSI6eyJ0YWlsIjp7Im5hbWUiOiJtb25vIn19fV0sWzEsMiwiIiwyLHsiY3VydmUiOjMsInN0eWxlIjp7InRhaWwiOnsibmFtZSI6Im1vbm8ifX19XSxbMSwwLCJcXHByb2Rfe2kgXFxpbiBKfSBmX2kiXSxbNCwzLCJmIiwwLHsibGFiZWxfcG9zaXRpb24iOjcwfV0sWzQsNSwiIiwyLHsiY3VydmUiOi0zLCJzdHlsZSI6eyJ0YWlsIjp7Im5hbWUiOiJtb25vIn19fV0sWzMsNSwiIiwwLHsic3R5bGUiOnsidGFpbCI6eyJuYW1lIjoibW9ubyJ9fX1dLFswLDNdLFsxLDRdLFsyLDVdXQ==
%\[\begin{tikzcd}[sep=small]
%	& {\prod_{i \in \I} \psi(A_i) / \F \cong \psi(A)} &&& {\prod_{i \in \I} A_i / \F \cong A^X} \\
%	{\prod_{i \in J} \psi(A_i)} &&& {\prod_{i \in J} A_i^X} \\
%	& {\prod_{i \in \I} \phi(A_i) / \F \cong \phi(A)} \\
%	{\prod_{i \in J} \phi(A_i)}
%	\arrow[tail, from=1-2, to=1-5]
%	\arrow[from=2-1, to=1-2]
%	\arrow[tail, from=2-1, to=2-4]
%	\arrow[from=2-4, to=1-5]
%	\arrow["f"{pos=0.7}, from=3-2, to=1-2]
%	\arrow[curve={height=-18pt}, tail, from=3-2, to=1-5]
%	\arrow["{\prod_{i \in J} f_i}", from=4-1, to=2-1]
%	\arrow[curve={height=18pt}, tail, from=4-1, to=2-4]
%	\arrow[from=4-1, to=3-2]
%\end{tikzcd}\]
\begin{center}
	\begin{tikzpicture}[baseline=(current  bounding  box.south), scale=2, on top/.style={preaction={draw=white,-,line width=#1}}, on top/.default=6pt]
		
		\node (a1) at (0,1.1) {$\prod_{i \in J} \phi(A_i)$};
		\node (b1) at (3.5,1.1) {$\prod_{i \in J} A_i^X $};
		\node (a0) at (0,0) {$\prod_{i \in J} \psi(A_i)$};
		
		\node (a1') at (1.5,1.7) {$\prod_{i \in \I} \phi(A_i) / \F \cong \phi(A)$};
		\node (b1') at (5,1.7) {$\prod_{i \in \I} A_i / \F \cong A^X$};
		\node (a0') at (1.5,0.7) {$\prod_{i \in \I} \psi(A_i) / \F \cong \psi(A)$};
		
		\path[font=\scriptsize]
		
		(a1') edge [>->] node [above] {} (b1')
		(a0') edge [bend right=20, >->] node [below] {} (b1')
		(a1') edge [->] node [above] {$f\ \ \ \ $} (a0')
		
		(a1) edge [>->, on top] node [above] {} (b1)
		(a0) edge [bend right=20, >->, on top] node [below] {} (b1)
		(a1) edge [->] node [left] {$\prod_{i \in J} f_i$} (a0)
		
		(a1) edge [->] node [above] {} (a1')
		(b1) edge [->] node [above] {} (b1')
		(a0) edge [->] node [above] {} (a0');
	\end{tikzpicture}	
\end{center}
Furthermore, the triangle involving $f$ commutes by the unicity part of \cref{reduced-product-factor-on-big-subset}. The isomorphisms in that same triangle follow from \cref{los-theorem}, which concludes our proof.
\end{proof}

For a converse we use pointwise satisfaction of formulas.

\begin{theorem}[{\PolishL}os-style theorem 2, with converse]
\label{los-sequents-converse}
Assume that $\V$ is reduced product stable and suppose that there exists a finitely presentable $\E$-generator $Y$ of $\V_0$. Let $\LL$ be a language on $\V$ and let $\sigma$ be a sequent of $\LL$-formulas. Let $A := \prod_{i \in \I} A_i/\F$ be a reduced product of $\LL$-structures. Then
\[
\{i \in \I \mid A_i \models \sigma \} \in \F \quad \text{if and only if} \quad A \models \sigma,
\]
provided that:
\begin{enumerate}[label=(\roman*)]
\item\label{los-converse:regular-1} $\sigma \equiv (\phi \vdash \psi)$ is a sequent of $X$-ary regular formulas and there is an arrow $a \colon Y \to A^X$ such that $A \models \phi(a)$, or;
\item\label{los-converse:regular-2} $\sigma$ is a sequent of regular formulas and $\F$ is an ultrafilter, or;
\item\label{los-converse:positive} $\sigma$ is a sequent of positive formulas, $\F$ is an ultrafilter, and $\V$ has the ultraproduct covering property.
\end{enumerate}
\end{theorem}
\begin{proof}
The left to right direction is \cref{los-sequents}, we prove the converse. We start by defining
\[
J := \{i \in \I \mid A_i \not \models \sigma\}.
\]
For every $i \in J$ there is $b_i\colon Y \to A_i^X$ such that $A_i \not \models \sigma(b_i)$, by \cref{sequent-satisfaction-test-on-generators}. We first prove \ref{los-converse:regular-1}.

By assumption, there is an arrow $a\colon Y \to A^X$ such that $A \models \phi(a)$. By \cref{los-theorem}\ref{los-theorem:regular} this arrow factors as
\[
Y \longrightarrow \prod_{i \in \I} \phi(A_i) / \F \cong \phi(A) \rightarrowtail A^X.
\]
Since $Y$ is finitely presentable, the first factor of the above arrow factors as
\[
Y \xrightarrow{\ a'\ } \prod_{i \in J'} \phi(A_i) \longrightarrow \prod_{i \in \I} \phi(A_i) / \F,
\]
for some $J' \in \F$. For $i \in J \cup J'$ we define an arrow $c_i\colon Y \to A^X_i$ as follows: if $i \in J$ then $c_i := b_i$, otherwise, if $i \in J' \setminus J$, then $c_i$ is the composite 
$$Y \xrightarrow{\ a'\ } \prod_{i \in J'} \phi(A_i) \longrightarrow \phi(A_i) \rightarrowtail A_i^X.$$ 
For each $i \in J \cup J'$ we have $A_i \models \phi(c_i)$. Indeed, for $i \in J' \setminus J$ this is by construction and for $i \in J$ we have $A_i \not \models \sigma(c_i)$, so we must have $A_i \models \phi(c_i)$ and $A_i \not \models \psi(c_i)$. Write $c\colon Y \to A^X$ for the composite
\[
Y \xrightarrow{(c_i)_{i \in J \cup J'}} \prod_{i \in J \cup J'} A_i^X \longrightarrow \prod_{i \in \I} A_i^X / \F \cong A^X.
\]
Then by construction (and implicitly using \cref{los-theorem}\ref{los-theorem:regular} again) we have $A \models \phi(c)$. So since $A \models \sigma$, we must also have $A \models \psi(c)$. So, like before, $c$ factors as
\[
Y \longrightarrow \prod_{i \in J^*} \psi(A_i) \longrightarrow \prod_{i \in J^*} A_i^X \longrightarrow \prod_{i \in \I} A_i^X / \F \cong A^X,
\]
for some $J^* \in \F$. We may as well assume $J^* \subseteq J \cup J'$. In fact, this factorisation is essentially unique, and so we may assume that the factor $Y \to \prod_{i \in J^*} \psi(A_i) \to \prod_{i \in J^*}A_i^X$ is given by $(c_i)_{i \in J^*}$. So for all $i \in J^*$ we have $A_i \models \psi(c_i)$, which means that $J^* \cap J = \emptyset$. In other words, $J^* \subseteq \I \setminus J$, and hence $\{ i \in \I \mid A_i \models \sigma \} = \I \setminus J \in \F$, as required.

We now prove \ref{los-converse:regular-2} by considering two cases. Either there is an arrow $a\colon Y \to A^X$ as in \ref{los-converse:regular-1}, and then we are done by \ref{los-converse:regular-1}, or there is no such arrow. In that second case we cannot have $J \in \F$, as $A_i \models \phi(b_i)$ for every $i \in J$ and so we would have an arrow
\[
b\colon Y \xrightarrow{(b_i)_{i \in J}} \prod_{i \in J} A_i^X \longrightarrow \prod_{i \in \I} A_i^X / \F \cong A^X,
\]
such that $A \models \phi(b)$. So since $\F$ is an ultrafilter, we have $\{ i \in \I \mid A_i \models \sigma \} = \I \setminus J \in \F$, as required.

Finally, the proof of \ref{los-converse:positive} is the same, except any uses of \cref{los-theorem}\ref{los-theorem:regular} are replaced by \cref{los-theorem}\ref{los-theorem:positive}.
\end{proof}
\begin{example}
\label{los-converse-assumption-necessity}
The assumption about the existence of an arrow $Y \to A^X$ in \cref{los-sequents-converse}\ref{los-converse:regular-1} is necessary. Without it, the statement is false, already in the case $\V = \Set$. In that case, $Y = 1$ is our finitely presentable generator. Consider a language $\LL$ with one unary relation symbol $P$. We let $\I = \omega$ and define $A_i := \{0, 1\}$ with $A_i \models \neg P(0) \wedge P(1)$ if $i$ is even and $A_i := \{0\}$ with $A_i \models \neg P(0)$ if $i$ is odd. Let $\F$ be the filter on $\I$ of cofinite sets. Then $A \models P(x) \vdash \bot$, because every cofinite set contains at least one odd number. At the same time $\{i \in \I \mid A_i \models P(x) \vdash \bot\}$ is the set of odd numbers, which is not an element of $\F$ because any cofinite set will contain at least one even number.
\end{example}

\section{Compactness and formulas with parameters}
\label{compactness-and-formulas-with-parameters}
We would be ready to prove the compactness theorem now, but one of the main strengths of the compactness theorem is to apply it to what model theorists refer to as ``parameters''. Their function is to describe part of a structure that we wish to construct and show that finite fragments of it can be constructed, so that the compactness theorem then implies that our desired structure exists. In the case $\V = \Set$ this is done using constant symbols: we introduce constant symbols for the elements we wish to exist in the structure and we use logical formulas to describe how we want these elements to be related to one another. For example, this happens in the model-theoretic method of diagrams (see \cref{traditional-method-of-diagrams}). However, for a general base of enrichment $\V$ we cannot just work with constant symbols. For example, in the case $\V = \Ab$ a constant symbol -- i.e., a $(0, X)$-ary function symbol -- can only be interpreted in one way in any given structure $A$, namely as the unique arrow $0 = 1 = A^0 \to A^X$. We thus introduce parameters as a separate mechanism.
\begin{definition}
\label{parameter}
A \emph{parameter} is a symbol $p\colon X$, where $X$ is an object of $\V$, called the \emph{domain of $p$}. Given a set $\P$ of parameters and an object $A$ of $\V$, an {\em interpretation of $\P$ in $A$} is a tuple $(p^A)_{p \in \P}$, where for each $p\colon X$ in $\P$
\[
	p^A\colon X \to A
\]
is an arrow in $\V$.

A \emph{morphism of $\P$-interpretations} $h\colon A \to B$ is an arrow in $\V$ making the triangle
	\begin{center}
		\begin{tikzpicture}[baseline=(current  bounding  box.south), scale=2]
			
			\node (a0) at (0,0.8) {$X$};
			\node (b0) at (1,0.8) {$A$};
			\node (c0) at (1,0) {$B$};
			
			\path[font=\scriptsize]
			
			(a0) edge [->] node [above] {$p^A$} (b0)
			(a0) edge [->] node [below] {$p^B\ \ \ \ $} (c0)
			(b0) edge [->] node [right] {$h$} (c0);
		\end{tikzpicture}	
	\end{center}
	commute for all $p\colon X$ in $\P$.
\end{definition}
The main point about parameters is to use them in $\LL$-structures and $\LL$-formulas.
\begin{definition}
\label{structures-with-parameters}
Let $\LL$ be a language on $\V$ and let $\P$ be a set of parameters. An \emph{$(\LL, \P)$-structure} $(A, (p^A)_{p \in \P})$ is an $\LL$-structure $A$ together with an interpretation $(p^A)_{p \in \P}$ of $\P$ in $A$. A \emph{morphism of $(\LL, \P)$-structures} is an arrow in $\V$ that is both a homomorphism of $\LL$-structures and a morphism of $\P$-interpretations. We write $\Str(\LL,\P)$ for the (ordinary) category of $(\LL, \P)$-structures.
\end{definition}
Sometimes we may drop the tuple of parameters from the notation and just say talk about the $(\LL, \P)$-structure $A$.

As discussed before, parameters are a separate mechanism from the enriched syntax. The example below shows the necessity for this separation. The situation is different when $\V$ is semicartesian (i.e., the monoidal unit $I$ is the terminal object), in which case we can capture parameters using constant symbols, like in the $\V = \Set$ case. This is made precise in \cref{sec:semicartesian-case}.
\begin{example}
\label{parameters-not-enriched}
In general $\Str(\LL,\P)$ is not the underlying category of some $\V$-category. Indeed, if $\V=\Ab$, $\LL=\emptyset$, and $\P=\{p\colon \Z\}$, then $\Str(\LL,\P)=\Ab_*$ coincides with the category of pointed abelian groups (that is, abelian groups with a specified element) which is not $\Ab$-enriched (since, for example, certain homsets are empty).
\end{example}
We now introduce the notion of parametrised formulas and sequents, and the validity of such formulas and sequents in structures. 
\begin{definition}
\label{parametrised-formulas}
An $\LL$-formula with parameters in $\P$ is a tuple $(\phi, p_1, \ldots, p_n)$, which we will denote by $\phi(p_1, \ldots, p_n)$, where $\phi$ is an $(X_1 + \ldots + X_n)$-ary $\LL$-formula and $p_1\colon X_1, \ldots, p_n\colon X_n$ are parameters in $\P$.

Given an $(\LL, \P)$-structure $(A, (p^A)_{p \in \P})$ and $\phi(p_1, \ldots, p_n)$ as above, there is an arrow $p := (p_1^A + \ldots + p_n^A)^t\colon I \to A^{X_1 + \ldots + X_n}$ induced by the transpose of the coproduct of the interpretations of the parameters. We say that \emph{$(p_1, \ldots, p_n)$ realises $\phi$ in $A$} or that \emph{$A$ satisfies $\phi(p_1, \ldots, p_n)$}, and write
\[
A \models \phi(p_1, \ldots, p_n),
\]
if $A \models \phi(p)$ as an $\LL$-structure, in the sense of \cref{evaluate-formula-at-generalised-element}.
\end{definition}
\begin{definition}
An $\LL$-sequent with parameters in $\P$ is a tuple $(\sigma, p_1, \ldots, p_n)$, denoted by $\sigma(p_1, \ldots, p_n)$, where $\sigma$ is an $(X_1 + \ldots + X_n)$-ary sequent $\phi \vdash \psi$, and $p_1\colon X_1, \ldots, p_n\colon X_n$ are parameters in $\P$. We may also write $\phi(p_1, \ldots, p_n) \vdash \psi(p_1, \ldots, p_n)$ for such a parametrised sequent.

Given an $(\LL, \P)$-structure $(A, (p^A)_{p \in \P})$ and $\sigma(p_1, \ldots, p_n)$ as above, we say that \emph{$A$ satisfies $\sigma(p_1, \ldots, p_n)$} and write
\[
A \models \sigma(p_1, \ldots, p_n),
\]
if $A \models \phi(p_1, \ldots, p_n)$ implies $A \models \psi(p_1, \ldots, p_n)$.

Given a set $\Sigma$ of $\LL$-sequents with in $\P$, we denote by $\Mod(\Sigma)$ the full subcategory of $\Str(\LL,\P)$ spanned by those $(\LL, \P)$-structures that satisfy every sequent in $\Sigma$. An object of $\Mod(\Sigma)$ is called a \emph{model of $\Sigma$}.
\end{definition}
As usual, we will view a formula $\phi(p_1, \ldots, p_n)$ with parameters as a sequent $\top \vdash \phi(p_1, \ldots, p_n)$, so that it makes sense to speak of a model of a set of formulas with parameters.

We note that the domains of parameters in a sequent (or formula) with parameters are finitely presentable, because they match the arity of the sequent, which is formulated in finitary logic.

We are now ready for our final ingredient in proving a compactness theorem for parametrised sequents: a parametrised version of our {\PolishL}os-style theorem, which in turn requires reduced products of $(\LL, \P)$-structures. 
\begin{definition}
Assume that $\V$ is reduced product ready. Let $\LL$ be a language on $\V$ and let $\P$ be a set of parameters. Given a family $\{A_i\}_{i \in \I}$ of $(\LL, \P)$-structures and a filter $\F$ on $\I$ we define the \emph{reduced product} $A := \prod_{i \in \I} A_i / \F$ to be the following $(\LL, \P)$-structure. The underlying $\LL$-structure is just the reduced product of $\LL$-structures. For $p\colon X$ in $\P$ we define its interpretation $p^A\colon X \to A$ to be the composition
\[
X \xrightarrow{(p^{A_i})_{i \in \I}} \prod_{i \in \I} A_i \to \prod_{i \in \I} A_i / \F,
\]
where the second arrow is the coprojection.
\end{definition}
\begin{theorem}[{\PolishL}os-style theorem, with parameters]
\label{los-parametrised}
Assume that $\V$ is reduced product stable. Let $\LL$ be a language on $\V$ and let $\sigma(p_1, \ldots, p_n)$ be a sequent of $\LL$-formulas with parameters in $\P$. Assume also that the monoidal unit $I$ is finitely presentable or that $\sigma$ is of the form $\top \vdash \psi(p_1, \ldots, p_n)$. Let $A := \prod_{i \in \I} A_i/\F$ be a reduced product of $(\LL, \P)$-structures. Then
\[
\{i \in \I \mid A_i \models \sigma(p_1, \ldots, p_n) \} \in \F \quad \text{implies} \quad A \models \sigma(p_1, \ldots, p_n),
\]
provided that:
\begin{enumerate}[label=(\roman*)]
\item\label{los-parametrised:regular} $\sigma$ is a sequent of regular formulas, or;
\item\label{los-parametrised:positive} $\sigma$ is a sequent of positive formulas, $\F$ is an ultrafilter, and $\V$ has the ultraproduct covering property.
\end{enumerate}
\end{theorem}
\begin{proof}
Let $\phi$ and $\psi$ be $(X_1 + \ldots + X_n)$-ary formulas such that $\sigma$ is the sequent $\phi \vdash \psi$, and suppose that $A \models \phi(p_1, \ldots, p_n)$. Write $p := (p_1^A + \ldots + p_n^A)^t \colon I \to A^{X_1 + \ldots + X_n}$ for the arrow induced by the transpose of the coproduct of the interpretations of the parameters. So $p$ factors through $\phi(A) \cong \prod_{i \in \I} \phi(A_i) / \F$, where the isomorphism follows from \cref{los-theorem}. We claim that $p$ factors as
\[
I \to \prod_{i \in J} \phi(A_i) \to \prod_{i \in J} A_i^{X_1 + \ldots + X_n} \to \prod_{i \in \I} A_i^{X_1 + \ldots + X_n} / \F \cong A^{X_1 + \ldots + X_n},
\]
for some $J \in \F$, such that the first two arrows compose to $((p_1^{A_i} + \ldots + p_n^{A_i})^t)_{i \in J}\colon I \to \prod_{i \in J} A_i^{X_1 + \ldots + X_n}$. If $\phi \equiv \top$ then we can take $J = \I$. If $\phi$ is not $\top$ then we assumed that $I$ is finitely presentable, so we find $J \in \F$ such that $p$ factors through $\prod_{i \in J} \phi(A_i)$. The condition on the composition of the first two arrows then follows (after possibly replacing $J$ with a smaller set in $\F$) from essential uniqueness of the factorisation of $p$ through the diagram of $\prod_{i \in \I} A_i^{X_1 + \ldots + X_n} / \F$.

Define
\[
J' := J \cap \{i \in \I \mid A_i \models \sigma(p_1, \ldots, p_1) \},
\]
so $J' \in \F$. Then for all $i \in J'$ we have that $A_i \models \psi(p_1, \ldots, p_n)$. So $(p_1^{A_i} + \ldots + p_n^{A_i})^t \colon I \to A_i^{X_1 + \ldots + X_n}$ factors through $\psi(A_i)$. We thus find the dashed arrow in the diagram below, making the entire diagram commute (where the isomorphisms at the top come from \cref{los-theorem}).
% In the tikzcd diagram below, the curve of the p arrow and the bottom was manually set to -50pt and 60pt respectively, as quiver does not allow for such curved arrows.
% https://q.uiver.app/#q=WzAsNyxbMCwxLCJJIl0sWzEsMiwiXFxwcm9kX3tpIFxcaW4gSid9IFxccHNpXntBX2l9Il0sWzIsMiwiXFxwcm9kX3tpIFxcaW4gSid9IEFfaV57WF8xICsgXFxsZG90cyArIFhfbn0iXSxbMiwxLCJcXHByb2Rfe2kgXFxpbiBJfSBBX2lee1hfMSArIFxcbGRvdHMgKyBYX259IC8gXFxGIl0sWzEsMSwiXFxwcm9kX3tpIFxcaW4gSX0gXFxwc2lee0FfaX0gLyBcXEYiXSxbMiwwLCJBXntYXzEgKyBcXGxkb3RzICsgWF9ufSJdLFsxLDAsIlxccHNpXkEiXSxbMCwxLCIiLDEseyJzdHlsZSI6eyJib2R5Ijp7Im5hbWUiOiJkYXNoZWQifX19XSxbMSwyXSxbMCwyLCIoKHBfMV57QV9pfSArIFxcbGRvdHMgKyBwX25ee0FfaX0pXnQpX3tpIFxcaW4gSid9IiwxLHsiY3VydmUiOjV9XSxbMiwzXSxbMSw0XSxbNCw2LCJcXGNvbmciLDFdLFs2LDUsIiIsMix7InN0eWxlIjp7InRhaWwiOnsibmFtZSI6Im1vbm8ifX19XSxbMyw1LCJcXGNvbmciLDFdLFs0LDNdLFswLDUsInAiLDEseyJjdXJ2ZSI6LTV9XV0=
\[\begin{tikzcd}
	& {\psi(A)} & {A^{X_1 + \ldots + X_n}} \\
	I & {\prod_{i \in \I} \psi(A_i) / \F} & {\prod_{i \in \I} A_i^{X_1 + \ldots + X_n} / \F} \\
	& {\prod_{i \in J'} \psi(A_i)} & {\prod_{i \in J'} A_i^{X_1 + \ldots + X_n}}
	\arrow[tail, from=1-2, to=1-3]
	\arrow["p"{description}, curve={height=-50pt}, from=2-1, to=1-3]
	\arrow[dashed, from=2-1, to=3-2]
	\arrow["{((p_1^{A_i} + \ldots + p_n^{A_i})^t)_{i \in J'}}"{description}, curve={height=60pt}, from=2-1, to=3-3]
	\arrow["\cong"{description}, from=2-2, to=1-2]
	\arrow[from=2-2, to=2-3]
	\arrow["\cong"{description}, from=2-3, to=1-3]
	\arrow[from=3-2, to=2-2]
	\arrow[from=3-2, to=3-3]
	\arrow[from=3-3, to=2-3]
\end{tikzcd}\]
So $p$ factors through $\psi(A)$ and hence $A \models \psi(p_1, \ldots, p_n)$, as required.
\end{proof}

\begin{repeated-theorem}[\cref{compactness} \textnormal{(Compactness)}]
Assume that $\V$ is reduced product stable. Let $\LL$ be a language on $\V$, let $\TT$ be a theory, and let $\Sigma$ be a set of sequents with parameters in $\P$. Assume also that the monoidal unit $I$ is finitely presentable or that all sequents in $\Sigma$ are of the form $\top \vdash \psi(p_1, \ldots, p_n)$.

If every finite subset of $\TT \cup \Sigma$ has a model then $\TT \cup \Sigma$ has a model, provided that:
\begin{enumerate}[label=(\roman*)]
\item every sequent in $\TT \cup \Sigma$ is a sequent of regular formulas (possibly with parameters), or;
\item every sequent in $\TT \cup \Sigma$ is a sequent of positive formulas (possibly with parameters) and $\V$ has the ultraproduct covering property.
\end{enumerate}
\end{repeated-theorem}
\begin{proof}
Recall that $[\TT \cup \Sigma]^{< \omega}$ denotes the set of finite subsets of $\TT \cup \Sigma$. For each $S \in [\TT \cup \Sigma]^{< \omega}$ we fix a model $A_S$ of $S$, which exists by assumption. For a sequent (possibly with parameters) $\sigma \in \TT \cup \Sigma$ we write
\[
[\sigma] := \{S \in [\TT \cup \Sigma]^{< \omega} \mid \sigma \in S\}.
\]
For any $\sigma_1, \sigma_2 \in \TT \cup \Sigma$ we have $\{\sigma_1, \sigma_2\} \in [\sigma_1] \cap [\sigma_2]$, so the family of sets $\{[\sigma]\}_{\sigma \in \TT \cup \Sigma}$ can be extended to an ultrafilter $\U$ on $[\TT \cup \Sigma]^{< \omega}$. We claim that the ultraproduct of $(\LL, \P)$-structures
\[
A := \prod_{S \in [\TT \cup \Sigma]^{< \omega}} A_S / \U
\]
is the required model of $\TT \cup \Sigma$.

To verify this, we let $\sigma \in \TT \cup \Sigma$ be a sequent (possibly with parameters). By construction $A_S \models \sigma$ for each $S \in [\sigma]$, so $[\sigma] \subseteq \{S \in [\TT \cup \Sigma]^{< \omega} \mid A_S \models \sigma \}$. Again, by construction $[\sigma] \in \U$, so $\{S \in [\TT \cup \Sigma]^{< \omega} \mid A_S \models \sigma \} \in \U$. We thus conclude by \cref{los-sequents} (or \cref{los-parametrised} if $\sigma$ is a sequent with parameters) that $A \models \sigma$, as required.
\end{proof}

\section{Model-theoretic method of diagrams}
\label{model-theoretic-method-of-diagrams}
The model-theoretic method of diagrams is a technique for axiomatising homomorphisms $A \to B$, for a fixed $\LL$-structure $A$. Combined with the compactness theorem, this is a powerful and extremely important model-theoretic tool. It allows us to amalgamate models (\cref{amalgamation}), and is the cornerstone of the correspondence between (syntactic) types and automorphism orbits of the monster model, see also \cref{monsters}.

We begin by recalling the traditional method of diagrams, and point out why this fails generally in categories other than $\Set$. This provides us with intuition for the rest of this section.
\begin{remark}[Traditional model-theoretic method of diagrams]
\label{traditional-method-of-diagrams}
Fix some traditional (i.e., $\V = \Set$) language $\LL$ and an $\LL$-structure $A$. We extend $\LL$ to a new language $\LL(A)$ by adding a new constant symbol $c_a$ for each $a \in A$. The \emph{diagram of $A$} is the following set of $\LL(A)$-sentences:
\begin{align*}
\Diag(A) := \{\phi(c_{a_1}, \ldots, c_{a_n}) \mid \,\,&\phi(x_1, \ldots, x_n) \text{ is an atomic $\LL$-formula and}\\
&A \models \phi(a_1, \ldots, a_n) \}.
\end{align*}
Given any $\LL$-homomorphism $h\colon A \to B$, we can view $B$ naturally as an $\LL(A)$-structure by interpreting $c_a$ as $h(a)$, for all $a \in A$, and we will have $B \models \Diag(A)$. Conversely, given a model $B$ of $\Diag(A)$, we can define a function $h\colon A \to B$ by setting $h(a) = (c_a)_B$, and this will be an $\LL$-homomorphism. There is thus a correspondence between $\LL$-homomorphisms with domain $A$ and models of $\Diag(A)$.
\end{remark}
In \cref{traditional-method-of-diagrams} we made essential use of the fact that we can construct a function $A \to B$ on an element-by-element basis. This can be done because any object in $\Set$ is a coproduct of the terminal object $1$. This is no longer true in an arbitrary base of enrichment $\V$. However, if $\V_0$ is locally finitely presentable then any object is still a colimit of finitely presentable objects. We can, and will, thus replace the use of constant symbols by parameters, as introduced in \cref{compactness-and-formulas-with-parameters}.

We start by breaking up an arbitrary parameter into its finite parts, similar to how traditionally an infinite tuple of variables or constant symbols breaks up into finite subtuples (and ultimately into singletons).
\begin{definition}
\label{finitary-parameters}
Let $A$ be any object in $\V$. We define a set of parameters as follows
$$ \P(A):=\{ p_a\colon X \mid a\colon X \to A \text{ is an arrow in $\V$ and $X$ is finitely presentable} \}. $$
\end{definition}
We only use $\P(A)$ in contexts where $\V_0$ is locally presentable, so it is indeed (essentially) a set.
\begin{definition}
\label{diagram-of-structure}
Let $\LL$ be a language on $\V$ and let $A$ be an $\LL$-structure. We naturally interpret $\P(A)$ in $A$ by defining $(p_a)^A:=a$. We then define the \emph{diagram of $A$} to be the following set of formulas with parameters in $\P(A)$:
\begin{align*}
\Diag(A):=\{ \phi(p_1, \ldots, p_n) \mid \,\,&p_1\colon X_1, \ldots, p_n\colon X_n \text{ are in } \P(A),\\
&\phi \text{ is $(X_1 + \ldots + X_n)$-ary regular atomic, and}\\
&A \models \phi(p_1, \ldots, p_n) \}.
\end{align*}
\end{definition}
By definition $A$ is a model of $\Diag(A)$.

We only allow regular atomic formulas in $\Diag(A)$ (i.e., we exclude $\bot$) for two reasons. Firstly, this is enough, as we will see. Secondly, we can then apply the compactness theorem to $\Diag(A)$ even in settings where we do not have the ultraproduct covering property (cf.\ \cref{sec:amalgamation,sec:sheaves-on-boolean-algebra}). This may seem unnecessary, because one would expect that $\bot$ is never satisfied, but this can happen. For example, if $\V = \Ab$, $\LL = \emptyset$, and $A = \Z$, then $A \models \bot(0)$, where $0$ is the zero map.

\begin{notation}
\label{equations-parameters}
Let $s\colon (X, Z)$ and $t\colon (Y, Z)$ be $\LL$-terms. We recall from \cite[Notation 4.7]{RT25ERegular} that $s(x) = t(y)$ is an abbreviation for the $(X+Y)$-ary atomic formula defined by
\[
s(i_X(x, y)) = t(i_Y(x, y)),
\]
where $i_X$ and $i_Y$ are the inclusions of $X$ and $Y$ into $X+Y$, respectively. We extend this notation to equations with parameters. That is, given parameters $p \colon X$ and $q\colon Y$, we write $s(p) = t(q)$ for the parametrised formula $(s(x)=t(y), p, q)$.
\end{notation}
\begin{proposition}
\label{equation-holds-iff-square-commutes}
Let $\LL$ be a language on $\V$, let $\P = \{p\colon X, q\colon Y\}$ be a set of parameters, and let $(A, p^A, q^A)$ be an $(\LL, \P)$-structure. For any $\LL$-terms $s\colon (X, Z)$ and $t\colon (Y, Z)$ we have that $A \models s(p) = t(q)$ if the square below commutes.
% https://q.uiver.app/#q=WzAsNCxbMCwwLCJJIl0sWzEsMCwiQV5YIl0sWzAsMSwiQV5ZIl0sWzEsMSwiQV5aIl0sWzAsMSwiKHBeQSledCJdLFswLDIsIihxXkEpXnQiLDJdLFsyLDMsInRfQSIsMl0sWzEsMywic19BIl1d
\[\begin{tikzcd}[sep=scriptsize]
	I & {A^X} \\
	{A^Y} & {A^Z}
	\arrow["{(p^A)^t}", from=1-1, to=1-2]
	\arrow["{(q^A)^t}"', from=1-1, to=2-1]
	\arrow["{s_A}", from=1-2, to=2-2]
	\arrow["{t_A}"', from=2-1, to=2-2]
\end{tikzcd}\]
The converse holds if $(\E, \M)$ is proper.
\end{proposition}
\begin{proof}
The interpretation of the term $i_X$ in $A$ is the projection $A^{i_X} \colon A^{X+Y} \to A^X$, and similar for $i_Y$ and $A^{i_Y}$. So the interpretation of the formula $s(x) = t(y)$ in $A$ is the $(\E, \M)$-factorisation of the equaliser $E$ of $s_A A^{i_X}$ and $t_A A^{i_Y}$, as pictured in the diagram below.
% https://q.uiver.app/#q=WzAsNCxbMiwwLCJBXntYK1l9Il0sWzMsMCwiQV5aIl0sWzEsMCwiKHMoeCkgPSB0KHkpKShBKSJdLFswLDAsIkUiXSxbMCwxLCJzX0EgQV57aV9YfSIsMCx7Im9mZnNldCI6LTF9XSxbMCwxLCJ0X0EgQV57aV9ZfSIsMix7Im9mZnNldCI6MX1dLFszLDIsImUiLDAseyJzdHlsZSI6eyJoZWFkIjp7Im5hbWUiOiJlcGkifX19XSxbMiwwLCJtIiwwLHsic3R5bGUiOnsidGFpbCI6eyJuYW1lIjoibW9ubyJ9fX1dXQ==
\[\begin{tikzcd}
	E & {(s(x) = t(y))(A)} & {A^{X+Y}} & {A^Z}
	\arrow["e", two heads, from=1-1, to=1-2]
	\arrow["m", tail, from=1-2, to=1-3]
	\arrow["{s_A A^{i_X}}", shift left, from=1-3, to=1-4]
	\arrow["{t_A A^{i_Y}}"', shift right, from=1-3, to=1-4]
\end{tikzcd}\]
Write $r := (p^A + q^A)^t \colon I \to A^{X + Y}$. If the square in the statement commutes then $s_A A^{i_X} r = t_A A^{i_Y} r$. So there is an induced arrow $I \to E$, which shows that $r$ factors through $(s(x) = t(y))(A)$. That is, $A \models s(p) = t(q)$, as required.

For the converse we assume that $(\E, \M)$ is proper. So $\E$ is contained in the epimorphisms, hence $\M$ contains all regular monomorphisms by \cref{factorisation-system-m-monos-equivalent-conditions}. Then $e$ is an isomorphism and $(s(x) = t(y))(A)$ is already the equaliser. If $A \models s(p) = t(q)$ then $r$ factors through $(s(x) = t(y))(A)$, say, as $m f$ for some $f\colon I \to (s(x) = t(y))(A)$. We compute
\[
s_A (p^A)^t = s_A A^{i_X} r = s_A A^{i_X} m f = t_A A^{i_Y} m f = t_A A^{i_Y} r = t_A (q^A)^t,
\]
as required.
\end{proof}
\begin{corollary}
\label{equation-holds-iff-triangle-commutes}
Let $\LL$ be a language on $\V$, let $\P = \{p\colon X, q\colon Y\}$ be a set of parameters, and let $(A, p^A, q^A)$ be an $(\LL, \P)$-structure. For any arrow $s\colon Y \to X$ in $\V$ we have that $A \models s(p) = q$ if $p^A \circ s = q^A$.

The converse holds if $(\E, \M)$ is proper.
\end{corollary}
\begin{proof}
By definition, $s$ is an $(X, Y)$-ary term. So we apply \cref{equation-holds-iff-square-commutes} with $t = \id_Y$. Commutativity of the square in \cref{equation-holds-iff-square-commutes} then becomes $(q^A)^t = s_A \circ (p^A)^t$. In this case, $s_A = A^s\colon A^X \to A^Y$, and so the transpose of $s_A \circ (p^A)^t$ is $p^A \circ s$. So, by taking transposes, commutativity of the square is equivalent to $p^A \circ s = q^A$.
\end{proof}
We now establish the correspondence between homomorphisms of $\LL$-structures with domain $A$ and models of $\Diag(A)$, which we will phrase in terms of an isomorphism of categories.
\begin{lemma}
\label{homomorphism-is-model-of-diagram}
Let $\LL$ be a language on $\V$ and let $A$ be an $\LL$-structure. There is a faithful functor
\begin{align*}
A/\Str(\LL)_0 \quad &\to \quad \Mod(\Diag(A)), \\
h\colon A \to B \quad &\mapsto \quad (B, (h \circ a)_{p_a \in \P(A)}),
\end{align*}
which is the identity on arrows. Furthermore, if $\V_\fp$ is a generator of $\V_0$ then this functor is also full and injective on objects.
\end{lemma}
\begin{proof}
Let $h\colon A \to B$ be an object in $A/\Str(\LL)_0$. First we show that $(B, (h \circ a)_{a \in \P(A)})$ is a model of $\Diag(A)$. Let $\phi(p_{a_1}, \ldots, p_{a_n})\in\Diag(A)$. Then the fact that $A \models \phi(p_{a_1}, \ldots, p_{a_n})$ means that the dashed arrow below exists.
\begin{center}
	\begin{tikzpicture}[baseline=(current  bounding  box.south), scale=2]
		
		\node (a0) at (0,0.8) {$I$};
		\node (b0) at (1.1,0.8) {$\phi(A)$};
		\node (e0) at (2.2,0.8) {$A^{X_1 + \ldots + X_n}$};
		\node (d0) at (1.1,0) {$\phi(B)$};
		\node (f0) at (2.2,0) {$B^{X_1 + \ldots + X_n}$};
		
		\path[font=\scriptsize]
		
		(a0) edge [dashed, ->] node [above] {} (b0)
		(b0) edge [->] node [left] {$\phi(h)$} (d0)
		(b0) edge [>->] node [below] {} (e0)
		(d0) edge [>->] node [above] {} (f0)
		(e0) edge [->] node [right] {$h^X$} (f0)
		
		(a0) edge [bend left=25, ->] node [above] {$(a_1 + \ldots + a_n)^t$} (e0);
	\end{tikzpicture}	
\end{center}  
Since $h^X (a_1 + \ldots + a_n)^t = (h \circ (a_1 + \ldots + a_n))^t$, and this factors through $\phi(B)$, it follows that also $B \models \phi(p_{a_1}, \ldots, p_{a_n})$.

Let now $h\colon A \to B$ and $h'\colon A \to B'$ be objects in $A/\Str(\LL)_0$, and let $g\colon B \to B'$ be an arrow in $A/\Str(\LL)_0$. That is, $g \circ h = h'$. Then $g$ is by definition already an arrow in $\Str(\LL)_0$, that is, it an $\LL$-homomorphism. Furthermore, for any $p_a \in \P(A)$ we have
\[
g \circ (h \circ a) = h' \circ a,
\]
so $g$ is a morphism of $(\LL, \P(A))$-structures. We conclude that the functor is indeed well-defined, and since it is the identity on arrows, it is faithful.

Assume now that $\V_\fp$ is a generator of $\V_0$. We first show that the functor is full. Let $h$ and $h'$ be as above, and let $g\colon (B, (h \circ a)_{p_a \in \P(A)}) \to (B', (h' \circ a)_{p_a \in \P(A)})$ be any arrow in $\Mod(\Diag(A))$. Then $g$ is a morphism of $(\LL, \P(A))$-structures. So for any $a\colon X \to A$ in $\V$ with $X$ in $\V_\fp$ we have
\[
g \circ h\circ a = h'\circ a.
\]
Using that $\V_\fp$ is a generator, it follows that $g \circ h=h'$. As $g$ is in particular an $\LL$-homomorphism, we conclude that it is an arrow $h \to h'$ in $A/\Str(\LL)_0$.

Finally we show that the functor is injective on objects. Once again, we let $h$ and $h'$ be as above, and now we assume that $(B, (h \circ a)_{p_a \in \P(A)}) = (B', (h' \circ a)_{p_a \in \P(A)})$. Then $B=B'$ as $\LL$-structures and $h \circ a = h'\circ a$ for any $a\colon X \to A$ in $\V$ with $X$ in $\V_\fp$. Since $\V_\fp$ is a generator of $\V_0$, it follows that $h=h'$.
\end{proof}

\begin{repeated-theorem}[\cref{enriched-method-of-diagrams} \textnormal{(Enriched model-theoretic method of diagrams)}]
Assume that $\V_0$ is locally finitely presentable with its finitely presentable objects closed under the monoidal product and that $(\E, \M)$ is proper. Let $\LL$ be a language on $\V$ and let $A$ be an $\LL$-structure. Then the functor
\begin{align*}
A/\Str(\LL)_0 \quad &\to \quad \Mod(\Diag(A)), \\
h\colon A \to B \quad &\mapsto \quad (B, (h \circ a)_{p_a \in \P(A)}),
\end{align*}
which is the identity on arrows, is an isomorphism of categories.
\end{repeated-theorem}
By \cref{finite-powers-directed-colimits-commute-iff-finites-closed-under-tensor} the above theorem applies in particular to those $\V$ that are reduced product ready and where $\V_0$ is locally finitely presentable (and $(\E, \M)$ is proper).
\begin{proof}
In a locally finitely presentable category $\V_0$, the finitely presentable objects $\V_\fp$ always form a generator. So by \cref{homomorphism-is-model-of-diagram} it only remains to show that the functor is surjective on objects, which amounts to constructing its inverse. Let $(B, (p_a^B)_{p_a \in \P(A)})$ be any model of $\Diag(A)$. We first construct an $\LL$-homomorphism $h\colon A \to B$. Since $\V_0$ is locally finitely presentable we can write
\[
A \cong \colim (\V_\fp / A \xrightarrow{\pi} \V_0)
\]
taken along the canonical diagram on the slice $\V_\fp/A$. Now, for any $t\colon Y \to X$ in $\V_\fp$ we have by \cref{equation-holds-iff-triangle-commutes} that
\[
A \models t(p_a) = p_{a \circ t}.
\]
Since $(B, (p_a^B)_{p_a \in \P(A)})$ is a model of $\Diag(A)$, we also know that $B \models t(p_a) = p_{a\circ t}$, which by \cref{equation-holds-iff-triangle-commutes} again means exactly that the triangle
	\begin{center}
		\begin{tikzpicture}[baseline=(current  bounding  box.south), scale=2]
			
			\node (a0) at (0,0.7) {$X$};
			\node (b0) at (1,0.7) {$Y$};
			\node (c0) at (1,0) {$B$};
			
			\path[font=\scriptsize]
			
			(a0) edge [->] node [above] {$t$} (b0)
			(a0) edge [->] node [below] {$p_{a\circ t}^B\ \ \ \ $} (c0)
			(b0) edge [->] node [right] {$p_a^B$} (c0);
		\end{tikzpicture}	
	\end{center}
commutes. In other words, the family $(p_a^B)_{a \in \V_\fp/A}$ defines a cocone over the diagram $\pi\colon \V_\fp / A \to \V_0$. Thus there exists a unique arrow $h\colon A \to B$ in $\V$ such that $h \circ a = p_a^B$ for all $a \in \V_\fp / A$. We are left to show that $h$ preserves the interpretation of function and relation symbols. It then follows immediately that the functor sends $h$ to $(B, (p^B_a)_{p_a \in \P(A)})$.
	
Before proceeding, we fix some notation. Let $b\colon Z \to A^X$ be any arrow in $\V$ with $X$ and $Z$ finitely presentable. Then the transpose $b^t\colon X \otimes Z \to A$ has finitely presentable domain, and so there is a parameter $p_{b^t}\colon X \otimes Z$ in $\P(A)$. The equality $h \circ b^t = p^B_{b^t}$ then translates to commutativity of the triangle below
	\begin{center}
		\begin{tikzpicture}[baseline=(current  bounding  box.south), scale=2]
			
			\node (a0) at (0,0.7) {$Z$};
			\node (b0) at (1,0.7) {$A^X$};
			\node (c0) at (1,0) {$B^X$};
			
			\path[font=\scriptsize]
			
			(a0) edge [->] node [above] {$b$} (b0)
			(a0) edge [->] node [below] {$q_b\ \ \ \ $} (c0)
			(b0) edge [->] node [right] {$h^X$} (c0);
		\end{tikzpicture}	
	\end{center}
where we define $q_b := (p_{b^t}^B)^t$.

Let now $f\colon (X, Y)$ be an arbitrary function symbol in $\LL$ and let $b\colon Z \to A^X$ be any arrow in $\V$ with $Z$ finitely presentable. Consider the diagram below, where we wish to prove that the central square commutes.
	\begin{center}
		\begin{tikzpicture}[baseline=(current  bounding  box.south), scale=2, on top/.style={preaction={draw=white,-,line width=#1}}, on top/.default=6pt]
			
			\node (a0) at (0,1.2) {$Z$};
			\node (b0) at (1.2,1) {$A^X$};
			\node (e0) at (2.4,1) {$B^X$};
			\node (d0) at (1.2,0) {$A^Y$};
			\node (f0) at (2.4,0) {$B^Y$};
			
			\path[font=\scriptsize]
			
			(a0) edge [->] node [above] {$b$} (b0)
			(a0) edge [bend left=20, ->] node [above] {$q_b$} (e0)
			(a0) edge [->] node [left] {$f_A \circ b$} (d0)
			(a0) edge [->] node [right] {$\ \ \ \ q_{f_A \circ p}$} (f0)
			
			(b0) edge [->, on top] node [left] {$f_A$} (d0)
			(b0) edge [->] node [above] {$h^X$} (e0)
			(d0) edge [->] node [below] {$h^Y$} (f0)
			(e0) edge [->] node [right] {$f_B$} (f0);
		\end{tikzpicture}	
	\end{center}
The triangles with vertex $Z$ and bases $h^X$ and $h^Y$ commute by the arguments above. The triangle with base $f_A$ commutes by definition and so by \cref{equation-holds-iff-square-commutes}:
\[
A \models f^Z(p_{b^t})= p_{(f_A \circ b)^t}.
\]
It follows that $B$ satisfies the same formula with parameters, which means that the triangle with base $f_B$ also commutes (\cref{equation-holds-iff-square-commutes}). In conclusion, the equality
\[
f_B\circ h^X\circ b= h^Y\circ f_A\circ b
\]
holds. Since this is the case for any arrow $b$ with domain in the generator $\V_\fp$, we obtain that $h$ preserves the interpretation of $f$.
	
Finally, let $R\colon X$ be an arbitrary relation symbol in $\LL$, and let $r_A\colon R_A \rightarrowtail A^X$ and $r_B\colon R_B \rightarrowtail B^X$ be its interpretations in $A$ and $B$, respectively. Let $\{b_j\colon Z_j \to R_A\}_{j \in J}$ be the set of all arrows in $\V$ with codomain $R_A$ and finitely presentable domain. For each $j \in J$ we write $a_j := r_A b_j\colon Z_j \to A^X $ with transpose $a_j^t\colon X \otimes Z_j \to A$. Then by definition $A \models R^{Z_j}(p_{a_j^t})$; thus also $B \models R^{Z_j}(p_{a_j^t})$ meaning that $(p_{a_j^t}^B)^t = h^X \circ a_j$ factors as $r_B \circ c_j$ for some $c_j\colon Z_j \to R_B$. We thus obtain the solid part of the commutative diagram below.
% https://q.uiver.app/#q=WzAsNSxbMCwxLCJSX0EiXSxbMSwxLCJBXlgiXSxbMiwxLCJCXlgiXSxbMiwwLCJSX0IiXSxbMCwwLCJcXHN1bV97aiBcXGluIEp9IFpfaiJdLFswLDEsInJfQSIsMix7InN0eWxlIjp7InRhaWwiOnsibmFtZSI6Im1vbm8ifX19XSxbMSwyLCJoXlgiLDJdLFszLDIsInJfQiIsMCx7InN0eWxlIjp7InRhaWwiOnsibmFtZSI6Im1vbm8ifX19XSxbNCwzLCJjIDo9IFxcc3VtX3tqIFxcaW4gSn0gY19qIl0sWzQsMCwiYiA6PSBcXHN1bV97aiBcXGluIEp9IGJfaiIsMl0sWzAsMywiIiwxLHsic3R5bGUiOnsiYm9keSI6eyJuYW1lIjoiZGFzaGVkIn19fV1d
\[\begin{tikzcd}
	{\sum_{j \in J} Z_j} && {R_B} \\
	{R_A} & {A^X} & {B^X}
	\arrow["{c := \sum_{j \in J} c_j}", from=1-1, to=1-3]
	\arrow["{b := \sum_{j \in J} b_j}"', from=1-1, to=2-1]
	\arrow["{r_B}", tail, from=1-3, to=2-3]
	\arrow[dashed, from=2-1, to=1-3]
	\arrow["{r_A}"', tail, from=2-1, to=2-2]
	\arrow["{h^X}"', from=2-2, to=2-3]
\end{tikzcd}\]
The arrow $b$ is then in $\E$ because $\V_\fp$ is an $\E$-generator by \cref{finitely-presentables-e-generator}. By definition $r_B$ is in $\M$, so by orthogonality the dashed arrow $R_A \to R_B$ exists, making everything commute. It follows that $h$ preserves the interpretation of $R$ and is an $\LL$-homomorphism, as required.
\end{proof}
If we take $\LL$ to be the empty language then a model of $\Diag(A)$ is essentially just a cocone for the canonical diagram of $A$. The model-theoretic method of diagrams then just says that, in a locally finitely presentable category $\V_0$, arrows $h\colon A \to B$ correspond to cocones for the canonical diagram of $A$. Although this is nothing new, it is useful that this can be expressed syntactically, so we introduce some notation and terminology for it.
\begin{definition}
\label{bare-diagram}
The \emph{bare diagram} $\BareDiag(A)$ of an object $A$ is $\Diag(A)$, where $A$ is considered as a structure in the empty language.
\end{definition}
\begin{corollary}
\label{method-of-bare-diagrams}
For any object $A$ of $\V$ there is a faithful functor
\begin{align*}
A/\V_0 \quad &\to \quad \Mod(\BareDiag(A)), \\
h\colon A \to B \quad &\mapsto \quad (B, (h \circ a)_{p_a \in \P(A)}),
\end{align*}
which is the identity on arrows.

Furthermore, if $\V_0$ is locally finitely presentable and $(\E, \M)$ is proper then this functor is an isomorphism of categories.
\end{corollary}
Note that the above does not require closure of $\V_\fp$ under the monoidal product, because that was only used in \cref{enriched-method-of-diagrams} to prove that $h$ is an $\LL$-homomorphism.

\subsection{The semicartesian case}
\label{sec:semicartesian-case}

In this section we focus on the case where $\V$ is semicartesian; that is, when the unit $I$ is the terminal object of $\V$. This is necessary to obtain an extension of Theorem~\ref{enriched-method-of-diagrams} where the isomorphism is actually an enriched isomorphism. We further assume that the unit is finitely presentable so that it can be used as an arity of our logic.

The reason why the unit needs to be the terminal object is twofold. On one hand, it allows us to see parameters as function symbols (with input arity $0$) of an enriched language, so that structures naturally assemble into a $\V$-category. On the other hand, enriched slice categories are better behaved in the semicartesian case, as the underlying category of an enriched slice is the ordinary slice on the underlying category

Let us focus first on parameters and interpretations. In this case, to specify a parameter $p\colon X$ is the same as specifying a function symbol $c_p\colon (0,X)$, that is, a constant symbol. Indeed, the interpretation of such a function symbol in a structure $A$ is given by an arrow
\[
(c_p)_A \colon 1 \cong A^0\to A^X
\]
which, by monoidal closedness, corresponds to an arrow $p^A\colon X \cong 1 \otimes X \to A$ in $\V$.
\begin{definition}
\label{adding-constants-for-parameters}
Given a set $\P$ of parameters (with finitely presentable domains) and a language $\LL$, we form a new language $\LL_\P$ by extending $\LL$ with a new constant symbol $c_p\colon (0, X)$ for each $p \in \P$.
\end{definition}
The correspondence between the interpretations of a parameter $p$ and a constant $c_p$ tells us how to view an $(\LL, \P)$-structure as an $\LL_\P$-structure, and vice versa. In particular, it follows that $\Str(\LL,\P)$ and $\Str(\LL_\P)_0$ are isomorphic.

A parametrised formula $\phi(p_1, \ldots, p_n)$ corresponds to the $\LL_\P$-formula $\phi(c_{p_1}, \ldots, c_{p_n})$. Indeed, the function symbols $c_{p_1}\colon (0, X_1), \ldots, c_{p_n}\colon (0, X_n)$ induce a $(0, X_1 + \ldots + X_n)$-ary term $(c_{p_1}, \ldots, c_{p_n})$, since the $n$-fold sum of $0$ is still $0$. Then this term can be substituted for the free variables in $\phi$ as in \cref{substitution}. The semantics for these two formulas agree, as is made precise below.
\begin{proposition}
\label{constants-vs-parameter-semantics}
Let $A$ be an $(\LL, \P)$-structure, simultaneously viewed as an $\LL_\P$-structure. Then
\[
A \models \phi(p_1, \ldots, p_n) \quad \Longleftrightarrow \quad A \models \phi(c_{p_1}, \ldots, c_{p_n}),
\]
provided that we satisfy the necessary assumptions for $\phi$ from \cref{substitution}.
\end{proposition}
Since the diagram of a structure only contains atomic formulas, the necessary (for us) assumptions from \cref{substitution} are satisfied by assuming that $(\E, \M)$ is proper, which we need anyway.
\begin{proof}
Consider the pullback below given by \cref{substitution}.
\begin{center}
	\begin{tikzpicture}[baseline=(current  bounding  box.south), scale=2]
		
		\node (0) at (-1.8,0.8) {$ $};
		\node (a0) at (0,0.8) {$\phi(c_{p_1}, \ldots, c_{p_n})(A)$};
		\node (a0') at (0.2,0.6) {$\lrcorner$};
		\node (b0) at (1.5,0.8) {$A^0 \cong 1$};
		\node (c0) at (0,0) {$\phi(A)$};
		\node (d0) at (1.5,0) {$A^{X_1 + \ldots + X_n}$};
		
		\path[font=\scriptsize]
		
		(a0) edge [>->] node [above] {$m$} (b0)
		(a0) edge [->] node [left] {} (c0)
		(b0) edge [->] node [right] {$(c_{p_1}, \ldots, c_{p_n})_A= (p_1^A +\ldots + p_n^A)^t$} (d0)
		(c0) edge [>->] node [below] {} (d0);
	\end{tikzpicture}	
\end{center}
Then $A \models \phi(p_1, \ldots, p_n)$ if and only if the arrow $(p_1^A + \ldots + p_n^A)^t$ factors through the inclusion of $\phi(A)$, if and only if there exists $r\colon 1\to \phi(c_{p_1}, \ldots, c_{p_n})(A)$ such that $m r=1$, if and only if $m$ is an isomorphism (since it is both in $\M$ and a split epimorphism), if and only if $A \models \phi(c_{p_1}, \ldots, c_{p_n})$. 
\end{proof}
It follows that we can view $\Mod(\Diag(A))$ as a $\V$-category of models of an $\LL_{\P(A)}$-theory, for which we introduce some notation.
\begin{definition}
\label{adding-constants-for-structure}
Let $A$ be an $\LL$-structure. We write $\LL(A)$ for the language that we called $\LL_{\P(A)}$ above. Through the translation discussed above, we view $\Diag(A)$ as an $\LL(A)$-theory.
\end{definition}
As mentioned in the introduction to this section, slice categories also make sense in the enriched sense. Indeed $A/\Str(\LL)$ can be viewed as a comma object in the $2$-category $\V$-$\mathbf{CAT}$ (which is complete). Explicitly, the objects of $A/\Str(\LL)$ are arrows $h\colon A\to B$ in $\Str(\LL)$, as in the ordinary case, and the hom-objects are defined by the pullback below in $\V$.
\begin{center}
	\begin{tikzpicture}[baseline=(current  bounding  box.south), scale=2]
		
		\node (a0) at (0,0.8) {$A/\Str(\LL)(h,h')$};
		\node (a0') at (0.2,0.6) {$\lrcorner$};
		\node (b0) at (1.6,0.8) {$\Str(\LL)(B,B')$};
		\node (c0) at (0,0) {$1$};
		\node (d0) at (1.6,0) {$\Str(\LL)(A,B')$};
		
		\path[font=\scriptsize]
		
		(a0) edge [->] node [above] {} (b0)
		(a0) edge [->] node [left] {} (c0)
		(b0) edge [->] node [right] {$-\circ h$} (d0)
		(c0) edge [->] node [below] {$h'$} (d0);
	\end{tikzpicture}
\end{center}
Since the same pullback in $\Set$ applied to $\Str(\LL)_0$ defines the homsets of $A/\Str(\LL)_0$, it follows immediately that $(A/\Str(\LL))_0\cong A/\Str(\LL)_0$.

\begin{remark}
	Note that the construction above still makes sense when $\V$ is not semicartesian, one just has to replace $1$ with $I$ in the pullback above. However, in this case it is not true any more that $(A/\Str(\LL))_0$ is isomorphic to $A/\Str(\LL)_0$, since the left vertical above now carries non-trivial information.  
\end{remark}

Since comma objects are examples of flexible 2-limits and $\Str(\LL)$ is locally presentable (as a $\V$-category), by \cite[Corollary 5.8]{LT22:virtual} also $A/\Str(\LL)$ is locally presentable and the forgetful $\V$-functor $U\colon A/\Str(\LL)\to \Str(\LL)$ is continuous and directed-colimit preserving. If $\V$ is locally finitely presentable as a closed category, then $\Str(\LL)$ and $A/\Str(\LL)$ are actually locally {\em finitely} presentable.

We are now ready to prove that in this case the ordinary functor from \cref{homomorphism-is-model-of-diagram} extends to one of $\V$-categories. Recall that a $\V$-functor is called faithful if the action on hom-objects is a monomorphism in $\V$.
\begin{lemma}
\label{homomorphism-is-model-of-diagram-enriched}
Assume that $\V$ is semicartesian, the monoidal unit is finitely presentable, and $(\E,\M)$ is proper. The functor from \cref{homomorphism-is-model-of-diagram} extends uniquely to a $\V$-functor
\begin{align*}
A/\Str(\LL) \quad &\to \quad \Mod(\Diag(A)), \\
h\colon A \to B \quad &\mapsto \quad (B, (h \circ a)_{p_a \in \P(A)}),
\end{align*}
which is faithful and preserves powers. Moreover, if $\V_\fp$ is a generator of $\V_0$ this is also injective on objects and fully faithful as a $\V$-functor.
\end{lemma}
\begin{proof}
To begin with, notice that $\Mod(\Diag(A))$ is closed in $\Str(\LL(A))$ under powers (since $(\E,\M)$ is proper the validity of regular atomic formulas is stable under powers), and the forgetful $\V$-functor  $\Str(\LL(A))\to\Str(\LL)$ preserves powers since they are computed as in $\V$ in both $\V$-categories (see \cite[Theorem~3.9]{RT25ERegular}). 
%This is conservative, since an homomorphism of $\LL(A)$-structures is an isomorphism if and only if the underlying morphism of $\LL$-structures is, and its underlying functor makes the triangle 
\begin{comment}
	\begin{tikzpicture}[baseline=(current  bounding  box.south), scale=2]
		
		\node (a0) at (0,0.8) {$A/\Str(\LL)_0$};
		\node (b0) at (1.6,0.8) {$\Mod(\Diag(A))_0$};
		\node (c0) at (0.8,0) {$\Str(\LL)_0$};
		
		\path[font=\scriptsize]
		
		(a0) edge [->] node [above] {$\tilde F$} (b0)
		(a0) edge [->] node [left] {$U_0\ $} (c0)
		(b0) edge [->] node [right] {$U'_0$} (c0);
	\end{tikzpicture}	
\end{comment}

 Consider now the ordinary functor 
 $$ \tilde F\colon A/\Str(\LL)_0 \longrightarrow\Mod(\Diag(A))_0$$
 defined by \cref{homomorphism-is-model-of-diagram}. By definition of $\tilde F$ and since powers in both $\V$-categories are constructed as in $\Str(\LL)$ if follows that $\tilde F$ preserves the {\em action of taking powers}; meaning that
$$ \tilde F(h^X)\cong \tilde F(h)^X $$
naturally in $h\in A/\Str(\LL)_0$ and $X\in\V$ (we cannot ask that $\tilde F$ actually preserves powers, as it is not a $\V$-functor). 

By the dual of results in the Appendix of \cite{JK01:articolo} (see also \cite[9.2]{LT20:articolo}), it follows that $\tilde F=F_0$ for a unique power-preserving $\V$-functor $F\colon A/\Str(\LL) \to \Mod(\Diag(A))$. The idea behind this argument is that, since $\tilde F$ preserves the action of taking powers, then the composite 
%$$ \V_0(X,A/\Str(\LL)(h,k) )\cong A/\Str(\LL)_0(h,k^X)\xrightarrow{\tilde F_{h,k^X}}\Mod(\Diag(A))_0(\tilde F(h),\tilde F(k^X))\cong \Mod(\Diag(A))_0(\tilde F(h),\tilde F(k)^X)\cong \V_0(X,\Mod(\Diag(A))(\tilde F(h),\tilde F(k))) $$
\begin{center}
	\begin{tikzpicture}[baseline=(current  bounding  box.south), scale=2]
		
		\node (a0) at (0,0.8) {$\V_0(X,A/\Str(\LL)(h,k) )$};
		\node (b0) at (2,0.8) {$A/\Str(\LL)_0(h,k^X)$};
		\node (c0) at (5,0.8) {$\Mod(\Diag(A))_0(\tilde F(h),\tilde F(k^X))$};
		\node (d0) at (0.8,0.1) {$\Mod(\Diag(A))_0(\tilde F(h),\tilde F(k)^X)$};
		\node (e0) at (4.2,0.1) {$\V_0(X,\Mod(\Diag(A))(\tilde F(h),\tilde F(k)))$};
		
		\path[font=\scriptsize]
		
		(a0) edge [->] node [above] {$\cong$} (b0)
		(b0) edge [->] node [above] {$\tilde F_{h,k^X}$} (c0)
		(c0) edge [->] node [above] {$\cong $} (d0)
		(d0) edge [->] node [below] {$\cong$} (e0);
	\end{tikzpicture}
\end{center}	
defines a natural transformation $\V_0(-,A/\Str(\LL)(h,k))\Rightarrow \V_0(-,\Mod(\Diag(A))(\tilde F(h),\tilde F(k)))$, which, by the Yoneda embedding, is induced by a unique morphism
$$ F_{h,k}\colon A/\Str(\LL)(h,k)\to \Mod(\Diag(A))_0(\tilde F(h),\tilde F(k))$$
which defines the action of $F$ on the hom-objects.
	
By \cref{homomorphism-is-model-of-diagram}, $F_0$ is faithful; we will use the fact $F$ preserves powers to infer that $F$ is then faithful as a $\V$-functor. Indeed, given any two objects $h,k \in A/\Str(\LL)$ the action on hom-objects $F_{h,k}$ is a monomorphism if and only if $\V_0(X, F_{h,k})$ is injective for all objects $X$ in $\V$. Since $F$ preserves powers by $X$ it follows that
\[
\V_0(X, F_{h,k}) \cong (F_0)_{h,k^X}
\]
as objects of $\Set^\bbtwo$. As $F_0$ is faithful we have that $(F_0)_{h,k^X}$ is a monomorphism and hence $F_{h,k}$ is a monomorphism too.

Assume now that $\V_\fp$ is a generator of $\V$, then by \cref{homomorphism-is-model-of-diagram} the functor $F_0$ is fully faithful and injective on objects. Thus, arguing as above, $\V_0(X, F_{h,k})$ is bijective for any $X\in\V$ and any $h,k$ in the slice. It follows that $F_{h,k}$ is an isomorphism in $\V$ and hence $F$ is fully faithful as a $\V$-functor.
\end{proof}
\begin{theorem}[Enriched model-theoretic method of diagrams, semicartesian case]
\label{semicartesian-enriched-method-of-diagrams}
Assume that $\V$ is semicartesian with $\V_0$ locally finitely presentable as a closed category, and that $(\E, \M)$ is proper. Let $\LL$ be a language on $\V$ and let $A$ be an $\LL$-structure. Then the $\V$-functor
\begin{align*}
A/\Str(\LL) \quad &\to \quad \Mod(\Diag(A)), \\
h\colon A \to B \quad &\mapsto \quad (B, (h \circ a)_{p_a \in \P(A)}),
\end{align*}
from \cref{homomorphism-is-model-of-diagram-enriched} is an isomorphism of $\V$-categories.
\end{theorem}
\begin{proof}
By \cref{homomorphism-is-model-of-diagram-enriched} the functor is injective on objects and fully faithful. Moreover, by \cref{enriched-method-of-diagrams} it is actually bijective on objects (since its underlying functor is); thus it is an isomorphism of $\V$-categories.
\end{proof}

\section{Amalgamation}
\label{sec:amalgamation}
The main application of the model-theoretic method of diagrams is to amalgamate structures (see \cref{amalgamation}). The base of amalgamation does not even need to be a structure itself (see \cref{types-amalgamate-structures}). The latter is what ultimately results in the power of the central model-theoretic notion of types (see \cref{monsters}).
\begin{definition}
\label{pp-formula}
An $\LL$-formula $\phi$ is called \emph{positive primitive}, or a \emph{pp-formula}, if it is of the form $$\phi(x) \equiv \exists y \psi(x,y)$$ where $\psi$ is a conjunction of regular atomic formulas.
\end{definition}
\begin{definition}
\label{pp-type}
Let $A$ be an $\LL$-structure and let $a\colon X \to A$ be any arrow in $\V$. We naturally view $A$ as an $(\LL, \P(X))$-structure as follows. By definition, a parameter $p_x\colon X'$ in $\P(X)$ corresponds to an arrow $x\colon X' \to X$, and we set $p^A_x := a \circ x$. We define the \emph{positive primitive type of $a$ in $A$}, or \emph{pp-type of $a$ in $A$}, to be the following set of formulas with parameters
\begin{align*}
\tppp(a; A) := \{ \phi(p_1, \ldots, p_n) \mid \,\,&p_1\colon X_1, \ldots, p_n\colon X_n \text{ are in } \P(X),\\
&\phi \text{ is an $(X_1 + \ldots + X_n)$-ary pp-formula, and}\\
&A \models \phi(p_1, \ldots, p_n) \}.
\end{align*}
\end{definition}
\begin{remark}
\label{diagram-contained-in-type}
The definitions of pp-type and diagram (\cref{diagram-of-structure}) look somewhat similar. There are two key differences. Firstly, the parameters in a type come from $\P(X)$ for any object $X$ in $\V$, so $X$ does not have to be (the underlying object of) an $\LL$-structure. Secondly, the diagram only contains atomic formulas with parameters. For an arbitrary $\LL$-structure $A$ we thus have that
\[
\Diag(A) \subseteq \tppp(\id_A; A),
\]
and in fact $\Diag(A)$ consists of precisely the atomic formulas in $\tppp(\id_A; A)$.
\end{remark}
In combination with the compactness theorem, types are very powerful. The key property is that they allow us to amalgamate structures, even if the `base' of the amalgamation is not necessarily a structure itself.
\begin{repeated-theorem}[\cref{types-amalgamate-structures}]
Assume that $\V$ is reduced product stable, that $\V_0$ is locally finitely presentable, that $(\E, \M)$ is proper, and that $I$ is $\E$-projective.

Let $A$ be an $\LL$-structure and $B$ be a model of some $\LL$-theory $\TT$, and let $A \xleftarrow{a} X \xrightarrow{b} B$ be a span of arrows in $\V$. Suppose that $\tppp(a; A) \subseteq \tppp(b; B)$. Then there is a model $C$ of $\TT$ together with homomorphisms $f\colon A \to C$ and $g\colon B \to C$ such that $f \circ a = g \circ b$, provided that:
\begin{enumerate}[label=(\roman*)]
\item $\TT$ is regular, or;
\item $\TT$ is positive and $\V$ has the ultraproduct covering property.
\end{enumerate}
\end{repeated-theorem}
\begin{proof}
Consider the set of parameters $\P := \P(A) \cup \P(B) \cup \P(X)$, together with the following set of parametrised formulas
\[
\Gamma := \Diag(A) \cup \Diag(B) \cup \BareDiag(X) \cup \Sigma_a \cup \Sigma_b.
\]
Here $\Sigma_a := \{ p_{a \circ x} = p_x \mid p_x \in \P(X) \}$, which is a set of formulas with parameters in $\P(A) \cup \P(X)$, and $\Sigma_b$ is defined similarly.

We claim that $\TT \cup \Gamma$ has a model. By compactness (\cref{compactness}) it is enough to show that every finite subset of $\TT \cup \Gamma$ has a model. We prove this by showing that for any finite $\Delta_0 \subseteq \Diag(A)$ and $\Sigma_0 \subseteq \Sigma_a$ the set of parametrised formulas
\[
\Gamma_0 := \Delta_0 \cup \Diag(B) \cup \BareDiag(X) \cup \Sigma_0 \cup \Sigma_b
\]
has a model that is also a model of $\TT$.

Let $p_{x_1}\colon X_1, \ldots, p_{x_n}\colon X_n$ and $p_{a_1}\colon Y_1, \ldots, p_{a_k}\colon Y_k$ and be the parameters from $\P(X)$ and $\P(A)$, respectively, that appear in $\Sigma_0$ and $\Delta_0$, without duplicates. Let $\psi(\bar x,\bar y)$ be the conjunction of all the underlying $\LL$-formulas in $\Delta_0$ and $\Sigma_0$, with $\bar x=(x_1, \ldots, x_n)$ and $\bar y=( y_1, \ldots, y_k)$. Here, if a certain formula involves only a subset of $(X_1, \ldots, X_n, Y_1, \ldots, Y_k)$, we act by substitution along the term induced by the coproduct inclusion of the given arities, this does not change the meaning of satisfaction by \cref{substitution}. So $\psi(\bar x,\bar y)$ is $(X_1 + \ldots + X_n + Y_1 + \ldots + Y_k)$-ary and is a conjunction of regular atomic formulas. We let $\phi(\bar x)$ be the pp-formula $\exists \bar y \psi(\bar x,\bar y)$.
 
Naturally viewing $A$ as an $(\LL, \P(A) \cup \P(X))$-structure, we have $A \models \psi(\bar p_{\bar x}, \bar p_{\bar a})$, where $\bar p_{\bar x}= (p_{x_1}, \ldots, p_{x_n})$ and $\bar p_{\bar a}= (p_{a_1}, \ldots, p_{a_k})$. By \cref{existential-quantification-and-realisations} we thus have $A \models \phi(\bar p_{\bar x})$. So $\phi(\bar p_{\bar x}) \in \tppp(a; A) \subseteq \tppp(b; B)$, and hence $B \models \phi(\bar p_{\bar x})$. By \cref{existential-quantification-and-realisations} again we thus find $a'\colon I \to B^{Y_1 + \ldots + Y_k}$ such that
\[
B \models \psi(((p^B_{x_1} + \ldots + p^B_{x_n})^t, a')).
\]

We will now interpret $p_{a_1}, \ldots, p_{a_k}$ in $B$. For each $1 \leq i \leq k$ we let $p^B_{a_i}$ be the transpose of the $i$th component of $a'$. That is, the transpose of the composite
\[
I \xrightarrow{a'} B^{Y_1 + \ldots + Y_k} \cong B^{Y_1} \times \ldots \times B^{Y_k} \xrightarrow{\pi_i} B^{Y_i},
\]
where $\pi_i$ is the $i$th projection. By construction we thus have $B \models \psi(\bar p_{\bar x}, \bar p_{\bar a})$, making $B$ into a model of $\TT \cup \Gamma_0$, proving the claim.

There is thus a model $C$ of $\TT \cup \Gamma$. Since $C$ is a model of $\Diag(A)$ and $\Diag(B)$ there are $\LL$-homomorphisms $f\colon A \to C$ and $g\colon B \to C$ by the enriched model-theoretic method of diagrams (\cref{enriched-method-of-diagrams}). There is also an arrow $h\colon X \to C$ in $\V$ by \cref{method-of-bare-diagrams}, because $C$ is a model of $\BareDiag(X)$. By that same result, $h$ is the unique arrow such that $p^C_x = h \circ x$ for all $p_x \in \P(X)$. We also have
\[
p^C_x =  p^C_{a \circ x} = f \circ a \circ x,
\]
for all $p_x \in \P(X)$, where the first equality follows because $C$ is a model of $\Sigma_a$ and the second equality follows because $f$ is such that $p^C_{a'} = f \circ a'$ for all $p_{a'} \in \P(A)$. By uniqueness of $h$ we thus conclude that $h = f \circ a$. Analogously, we have $h = g \circ b$, and so $f \circ a = g \circ b$.
\end{proof}
\begin{remark}
\label{monsters}
Typically, model theorists like to work in a monster model $\MM$, a very big and very homogeneous model in which every (small) model embeds. Amalgamation of small models then translates to an automorphism of the monster model, making it so that two tuples $a$ and $b$ in $\MM$ (in our setting: arrows $a,b\colon X \to \MM$) are in the same automorphism orbit (in our setting: satisfy $f \circ a = b$ for some automorphism $f$ of $\MM$) if and only if they satisfy the same formulas (in our setting: $\tppp(a; \MM) = \tppp(b; \MM)$). The construction of such a monster model is standard, once an amalgamation property like \cref{types-amalgamate-structures} has been established. However, doing so is beyond the scope of this paper, so we leave it for future work.
\end{remark}

The following notion was considered in \cite[Section~6]{RT25ERegular}, under the name of {\em elementary morphism}, where a connection was established with an enriched notion of pure morphisms in $\Str(\LL)$.

\begin{definition}
\label{immersion}
A homomorphism $h\colon A \to B$ of $\LL$-structures is called an \emph{$\LL$-immersion} if for every pp-formula $\phi(x)$ the induced square
% https://q.uiver.app/#q=WzAsNCxbMCwwLCJcXHBoaShBKSJdLFsxLDAsIlxccGhpKEIpIl0sWzAsMSwiQV5YIl0sWzEsMSwiQl5YIl0sWzIsMywiaF5YIiwyXSxbMCwxLCJoX1xccGhpIl0sWzAsMiwiIiwxLHsic3R5bGUiOnsidGFpbCI6eyJuYW1lIjoibW9ubyJ9fX1dLFsxLDMsIiIsMSx7InN0eWxlIjp7InRhaWwiOnsibmFtZSI6Im1vbm8ifX19XV0=
\[\begin{tikzcd}[sep=scriptsize]
	{\phi(A)} & {\phi(B)} \\
	{A^X} & {B^X}
	\arrow["{\phi(h)}", from=1-1, to=1-2]
	\arrow[tail, from=1-1, to=2-1]
	\arrow[tail, from=1-2, to=2-2]
	\arrow["{h^X}"', from=2-1, to=2-2]
\end{tikzcd}\]
is a pullback.
\end{definition}

\begin{proposition}
\label{semantic-homomorphism}
Let $h\colon A \to B$ be an $\LL$-homomorphism. Then for every $a\colon Y \to A^X$ and any $X$-ary pp-formula $\phi(x)$ we have that
\[
A \models \phi(a) \quad \Longrightarrow \quad B \models \phi(h^X \circ a).
\]
If $h$ is an $\LL$-immersion then we have for all such $a$ and $\phi$ that
\[
A \models \phi(a) \quad \Longleftrightarrow \quad B \models \phi(h^X \circ a).
\]
\end{proposition}
\begin{proof}
Consider the diagram below
% https://q.uiver.app/#q=WzAsNSxbMCwwLCJZIl0sWzEsMCwiXFxwaGkoQSkiXSxbMiwwLCJcXHBoaShCKSJdLFsxLDEsIkFeWCJdLFsyLDEsIkJeWCJdLFszLDQsImheWCIsMl0sWzEsMiwiXFxwaGkoaCkiXSxbMSwzLCIiLDIseyJzdHlsZSI6eyJ0YWlsIjp7Im5hbWUiOiJtb25vIn19fV0sWzIsNCwiIiwwLHsic3R5bGUiOnsidGFpbCI6eyJuYW1lIjoibW9ubyJ9fX1dLFswLDMsImEiLDAseyJjdXJ2ZSI6M31dLFswLDIsImIiLDAseyJjdXJ2ZSI6LTMsInN0eWxlIjp7ImJvZHkiOnsibmFtZSI6ImRhc2hlZCJ9fX1dLFswLDEsImMiLDEseyJzdHlsZSI6eyJib2R5Ijp7Im5hbWUiOiJkYXNoZWQifX19XV0=
\[\begin{tikzcd}[row sep=scriptsize]
	Y & {\phi(A)} & {\phi(B)} \\
	& {A^X} & {B^X}
	\arrow["c"{description}, dashed, from=1-1, to=1-2]
	\arrow["b", curve={height=-18pt}, dashed, from=1-1, to=1-3]
	\arrow["a", curve={height=18pt}, from=1-1, to=2-2]
	\arrow["{\phi(h)}", from=1-2, to=1-3]
	\arrow[tail, from=1-2, to=2-2]
	\arrow[tail, from=1-3, to=2-3]
	\arrow["{h^X}"', from=2-2, to=2-3]
\end{tikzcd}\]
If $A \models \phi(a)$ then by definition the arrow $c$ exists, making the left triangle commute. So $h^X \circ a$ factors through $\phi(B)$ via $\phi(h) \circ c$, and hence $B \models \phi(h^X \circ a)$.

If $h$ is an $\LL$-immersion and $B \models \phi(h^X \circ a)$ then by definition the arrow $b$ in the above diagram exists, making the outer square commute. Since the smaller square is a pullback square, we thus find the arrow $c$ making everything commute. So $a$ factors through $\phi(A)$ via $c$, and we conclude that $A \models \phi(a)$.
\end{proof}

\begin{corollary}
\label{amalgamation}
Assume that $\V$ is reduced product stable, that $\V_0$ is locally finitely presentable, that $(\E, \M)$ is proper, and that $I$ is $\E$-projective.

Let $B \xleftarrow{f} A \xrightarrow{g} C$ be a span of $\LL$-structures, where $f$ is an immersion, $g$ is a homomorphism, and $C$ is a model of some $\LL$-theory $\TT$. Then there exists a cospan $B \xrightarrow{f'} D \xleftarrow{g'} C$ of homomorphisms, such that the arising square commutes (i.e., $f' \circ f = g' \circ g$) and $D$ is a model of $\TT$, provided that:
\begin{enumerate}[label=(\roman*)]
\item\label{amalgamation:regular} $\TT$ is regular, or;
\item\label{amalgamation:positive} $\TT$ is positive and $\V$ has the ultraproduct covering property.
\end{enumerate}
\end{corollary}
\begin{proof}
We claim that $\tppp(f; B) \subseteq \tppp(g; C)$, and so the result follows from \cref{types-amalgamate-structures}. Let $\phi(p_{a_1}, \ldots, p_{a_n}) \in \tppp(f; B)$ be a formula with parameters $p_{a_1}\colon X_1, \ldots, p_{a_n}\colon X_n$ in $\P(A)$. Writing out definitions, this means that
\[
B \models \phi(f^{X_1 + \ldots + X_n} \circ (a_1 + \ldots + a_n)^t).
\]
So by \cref{semantic-homomorphism} we have
\[
A \models \phi((a_1 + \ldots + a_n)^t).
\]
Again, by \cref{semantic-homomorphism} we have
\[
C \models \phi(g^{X_1 + \ldots + X_n} \circ (a_1 + \ldots + a_n)^t).
\]
The latter is precisely saying that $\phi(p_{a_1}, \ldots, p_{a_n}) \in \tppp(g; C)$, as required.
\end{proof}

\section{The ultraproduct covering property}
\label{sec:ultraproduct-covering-property}
In this section we give sufficient conditions that imply the ultraproduct covering property (\cref{ultraproduct-covering-property}). Note that the contents of this section are about ordinary category theory, no enrichment is considered here.
\begin{definition}
\label{finitely-connected}
An object $X$ in a category $\C$ is called \emph{finitely connected} if $\C(X, -)$ preserves finite coproducts.
\end{definition}
\begin{theorem}\label{ultraproduct-covering-property-isos}
Suppose that a category $\C$ has a strong generator $\X$ formed by finitely presentable finitely connected objects. Then the arrows appearing in \cref{ultraproduct-covering-property} are isomorphisms. It follows that $\C$ satisfies the ultraproduct covering property for any choice of factorisation system.
\end{theorem}
\begin{proof}
Consider $\X$ as a full subcategory of $\C$; then the restricted Yoneda embedding
\[
J\colon \C \longrightarrow [\X\op,\Set]
\]
preserves all limits as well as directed colimits and finite coproducts (since $\C(X,-)$ does so for each $X \in \X$) and is conservative (since $\X$ is a strong generator). It follows that $J$ carries the comparison arrows in \cref{ultraproduct-covering-property} to the corresponding ones in $[\X\op, \Set]$, which we know to be isomorphisms because they are so in $\Set$ and limits and colimits are computed pointwise in $[\X\op, \Set]$. The result follows by conservativity of $J$.
\end{proof}

It is enough for $\X$ to only be an $\E$-generator.
\begin{theorem}
\label{sufficient-condition-for-ultraproduct-covering-property}
Let $\C$ be a category with a factorisation system $(\E, \M)$. Suppose that $\C$ has an $\E$-generator $\X$ that consists of finitely connected finitely presentable objects. Then $\C$ satisfies the ultraproduct covering property.
\end{theorem}
\begin{proof}
Let us first consider the condition	on the initial object. Since every object $X\in\X$ is finitely connected, there is no arrow $X\to0$. It follows by \cref{lifts-along-f-imply-f-in-e} that the unique arrow $0\to \prod_{j \in J}0$, for any nonempty set $J$, lies in $\E$. Thus so does $0\to \prod_{j \in J}0/\U$, for any ultrafilter $\U$ on $J$, by \cref{ultrapower-of-0}.	
	
Let $\{A_i\}_{i \in \I}$ and $\{B_i\}_{i \in \I}$ be familes of objects, $\U$ an ultrafilter on $\I$, and $f\colon X \to \prod_{i \in \I} (A_i + B_i) / \U$ any arrow with $X \in \X$. By \cref{lifts-along-f-imply-f-in-e} it is enough to show that $f$ factors through the comparison arrow
\[
\prod_{i \in \I} A_i / \U + \prod_{i \in \I} B_i / \U \to \prod_{i \in \I} (A_i + B_i) / \U.
\]
As $X$ is finitely presentable, $f$ factors as $X \xrightarrow{f'} \prod_{i \in J} (A_i + B_i) \to \prod_{i \in \I} (A_i + B_i) / \U$, for some $J \in \U$. For $i \in J$, write $f'_i\colon X \to A_i + B_i$ for the composition of $f'$ with the $i$th projection. Since $X$ is finitely connected, $f'_i$ factors as $X \xrightarrow{g_i} C_i \to A_i + B_i$, where $C_i$ is either $A_i$ or $B_i$. Define
\begin{align*}
J_A := \{i \in J \mid C_i = A_i \}, \quad\qquad
J_B := \{i \in J \mid C_i = B_i \}.
\end{align*}
Then either $J_A \in \U$ or $J_B \in \U$, because $\U$ is an ultrafilter. Without loss of generality, assume $J_A \in \U$. Then we find the following commutative diagram
% https://q.uiver.app/#q=WzAsNixbMCwwLCJYIl0sWzEsMCwiXFxwcm9kX3tpIFxcaW4gSl9BfSBBX2kiXSxbMiwxLCJcXHByb2Rfe2kgXFxpbiBKX0F9IEFfaSArQl9pIl0sWzIsMCwiXFxwcm9kX3tpIFxcaW4gSl9BfSBBX2kgKyBcXHByb2Rfe2kgXFxpbiBKX0F9IEJfaSJdLFszLDAsIlxccHJvZF97aSBcXGluIEl9QV9pIC8gXFxVICsgXFxwcm9kX3tpIFxcaW4gSX0gQl9pIC8gXFxVIl0sWzMsMSwiXFxwcm9kX3tpIFxcaW4gSX0gKEFfaSArIEJfaSkgLyBcXFUiXSxbMCwxLCIoZ19pKV97aSBcXGluIEpfQX0iXSxbMCwyLCIoZidfaSlfe2kgXFxpbiBKX0F9IiwyLHsiY3VydmUiOjN9XSxbMSwzXSxbMywyXSxbMiw1XSxbMyw0XSxbNCw1XV0=
\[\begin{tikzcd}
	X & {\prod_{i \in J_A} A_i} & {\prod_{i \in J_A} A_i + \prod_{i \in J_A} B_i} & {\prod_{i \in \I}A_i / \U + \prod_{i \in \I} B_i / \U} \\
	&& {\prod_{i \in J_A} A_i +B_i} & {\prod_{i \in \I} (A_i + B_i) / \U}
	\arrow["{(g_i)_{i \in J_A}}", from=1-1, to=1-2]
	\arrow["{(f'_i)_{i \in J_A}}"', curve={height=18pt}, from=1-1, to=2-3]
	\arrow[from=1-2, to=1-3]
	\arrow[from=1-3, to=1-4]
	\arrow[from=1-3, to=2-3]
	\arrow[from=1-4, to=2-4]
	\arrow[from=2-3, to=2-4]
\end{tikzcd}\]
The composition of the bottom arrows is $f$, so the composition of the top arrows gives the required factorisation.
\end{proof}

The conditions in \cref{ultraproduct-covering-property-isos,sufficient-condition-for-ultraproduct-covering-property} are sufficient, but not necessary. For example, in \cref{sec:example-abelian-groups} we will see that $\Ab$ satisfies the ultraproduct covering property (the comparison arrow is an isomorphism). However, there are no finitely connected objects in $\Ab$. Indeed, initial objects are never finitely connected and for any non-trivial group $G$ the diagonal $\Delta\colon G \to G \times G \cong G + G$ does not factor through either summand.

\section{Applications and examples}
\label{applications-and-examples}
Rather than having some running examples that we slowly build up throughout the paper, we present them here all at once, so that everything can be found in one place. We first give an overview of the various assumptions on $\V$ and $(\E, \M)$ that are used throughout the paper.
\begin{itemize}
\item The most basic assumptions are being \emph{reduced product ready} (\cref{reduced-product-ready}) and \emph{reduced product stable} (\cref{reduced-product-stable}), which are needed for all our main results.
\item All our main results have statements for regular logic and positive logic. The latter needs $\V$ to have the \emph{ultraproduct covering property} (\cref{ultraproduct-covering-property}).
\item The assumption that \emph{$\V_0$ has a finitely presentable $\E$-generator} is needed for a converse to version 2 of our {\PolishL}os-style theorem (\cref{los-sequents-converse}).
\item To deal with sequents with parameters, both in our {\PolishL}os-style theorem (\cref{los-parametrised}) and the enriched compactness theorem (\cref{compactness}), we need that the \emph{monoidal unit $I$ is finitely presentable}. Without that assumption, these theorems still work for formulas with parameters (i.e., sequents of the form $\top \vdash \phi(p_1, \ldots, p_n)$).
\item For the enriched model-theoretic method of diagrams (\cref{enriched-method-of-diagrams}) and for the amalgamation results in \cref{sec:amalgamation}, we need that \emph{$\V_0$ is locally finitely presentable, $(\E, \M)$ is proper, and $I$ is $\E$-projective}. That last assumption is not needed in the model-theoretic method of diagrams, but it does need the finitely presentable objects to be closed under the monoidal product, something which is already included the other assumptions for the amalgamation results.
\begin{itemize}
\item If \emph{$\V$ is semicartesian and the monoidal unit finitely presentable} then the enriched model-theoretic method of diagrams, which is formulated as an isomorphism of ordinary categories, upgrades to an isomorphism of enriched categories (\cref{semicartesian-enriched-method-of-diagrams}).
\end{itemize}
\end{itemize}
If the base is locally finitely presentable as a closed category, which is often the case, then many of the assumptions are automatic. In terms of the above list, the second and third bullet points become automatic, and the first one is partially satisfied.
\begin{proposition}
\label{locally-finitely-presentable-as-closed-category}
If $\V$ is locally finitely presentable as a closed category then the first two properties of being reduced product ready (\cref{reduced-product-ready}, \ref{reduced-product-ready:finite-limits-directed-colimits-commute} and \ref{reduced-product-ready:finite-powers-directed-colimits-commute}) hold. Furthermore, if the right class $\M$ of the factorisation system $(\E, \M)$ is the class of monomorphisms or the class of regular monomorphisms then $\V$ is reduced product ready.
\end{proposition}
\begin{proof}
In any locally finitely presentable category, finite limits commute with directed colimits \cite[Proposition 1.59]{AR94:libro}, and \cref{reduced-product-ready}\ref{reduced-product-ready:finite-powers-directed-colimits-commute} holds by \cref{finite-powers-directed-colimits-commute-iff-finites-closed-under-tensor}. If $\M$ is the class of (regular) monomorphisms then its members can be expressed as finite limits, and so \cref{reduced-product-ready}\ref{reduced-product-ready:m-closed-under-directed-colimits} holds.
\end{proof}

\subsection{The category of sets}
\label{sec:example-sets}
By taking $\V = \Set$ with its cartesian closed structure and the usual (epi, mono) factorisation system, we recover ordinary (unenriched) logic. Of course, this satisfies all the assumptions that we have considered in this paper for our enrichment base. The results in this paper then recover the classical results.

\subsection{Presheaf toposes}
\label{sec:example-toposes}
We describe a class of presheaf toposes that can be used as bases of enrichment, so that all our theorems apply.
\begin{theorem}
\label{example-presheaf-toposes}
Let $\C$ be a small category such that one of the following holds:
\begin{itemize}
\item $\C$ has binary products,
\item each connected component of $\C$ is equivalent to a finite category.
\end{itemize}
Then the presheaf topos $[\C\op, \Set]$ with the cartesian closed structure and the $(\E, \M) := (\text{epi}, \text{mono})$ factorisation system has the following properties:
\begin{enumerate}[label=(\arabic*)]
\item\label{example-presheaf-toposes:lfp} it is locally finitely presentable,
\item\label{example-presheaf-toposes:reduced-product-stable} it is reduced product stable,
\item\label{example-presheaf-toposes:ultraproduct-covering-property} it satisfies the ultraproduct covering property.
\end{enumerate}
\end{theorem}
\begin{proof}
Presheaf categories are locally finitely presentable \cite[Example 1.12]{AR94:libro}: the representable presheaves form a strong generator of finitely presentable objects.

For \ref{example-presheaf-toposes:reduced-product-stable} we note that any topos automatically satisfies \ref{reduced-product-ready:finite-limits-directed-colimits-commute} and \ref{reduced-product-ready:m-closed-under-directed-colimits} of \cref{reduced-product-ready}. So to conclude that $[\C\op, \Set]$ is reduced product ready it is enough to show that finitely presentable objects are closed under binary products, by \cref{finite-powers-directed-colimits-commute-iff-finites-closed-under-tensor}. This follows by \cref{product-of-fp-presheaves-if-c-has-finite-components,product-of-fp-presheaves-if-c-has-products} for our respective cases.

Epimorphisms in presheaf categories are precisely those morphisms that are pointwise surjective. Since products are computed pointwise, we conclude that epimorphisms are closed under products, so $[\C\op, \Set]$ is reduced product stable.

Property \ref{example-presheaf-toposes:ultraproduct-covering-property}, the ultraproduct covering property, holds because it does so in $\Set$ and limits and colimits in $[\C\op, \Set]$ are computed pointwise.
\end{proof}
\begin{corollary}
\label{example-presheaf-toposes-lfp-as-a-closed-category}
Let $\C$ be a small category such that one of the following holds:
\begin{itemize}
\item $\C$ has all finite products,
\item $\C$ is equivalent to a finite category with a terminal object.
\end{itemize}
Then the presheaf topos $[\C\op, \Set]$ with the cartesian closed structure and the $(\E, \M) := (\text{epi}, \text{mono})$ factorisation system, in addition to satisfying the properties listed in \cref{example-presheaf-toposes}, is locally finitely presentable as a closed category and the terminal object (and hence, the monoidal unit) is projective.
\end{corollary}
\begin{proof}
Both  properties follow from the terminal object being representable.
\end{proof}
The only place in the above proof where the assumptions in the bullet points are used is to show that finitely presentable objects are closed under binary products. We first give an example of how this can fail in general presheaf categories.
\begin{example}
\label{set-n-not-lfp-as-cartesian-closed-category}
Consider the category $\Set^\NN$ of sets with an endomorphism. So this is a presheaf category, where $\NN$ is viewed as a monoid (i.e., a one-object category) under addition. There is one representable presheaf, which is simply $\NN$ where the endomorphism takes $n$ to $n+1$. Representable presheaves are finitely presentable, but the product $\NN \times \NN$ is not finitely presentable. Indeed, the endomorphism on the product sends $(n, k)$ to $(n+1, k+1)$. So, for $i < \omega$, we can define $X_i := \{ (n,k) \in \NN \times \NN \mid |n-k| \leq i \}$ and we view $X_i$ as a subpresheaf of $\NN \times \NN$. With the inclusions $X_i \subseteq X_j$ for $i \leq j$ we can write $\NN \times \NN$ as the directed colimit $\NN \times \NN = \colim_{i < \omega} X_i$. The identity on $\NN \times \NN$ does not factor through this diagram, so $\NN \times \NN$ is not finitely presentable.
\end{example}

We are grateful to Ad\'{a}mek and Sousa for helpful discussions that improved the proofs of the following two lemmas. The (partial) converse of the first is due to them, using similar techniques as in their recent work~\cite{AS26}, and we thank them for allowing us to include it.
\begin{lemma}
\label{product-of-fp-presheaves-if-c-has-finite-components}
If each connected component of $\C$ is equivalent to a finite category then the finitely presentable objects in $[\C\op, \Set]$ are closed under binary products. If $\C$ is a groupoid the converse is true. 
\end{lemma}
\begin{proof}
Without loss of generality we can assume that each connected component of $\C$ {\em is} a finite category. Let $F$ be a presheaf on $\C$. If $F(C)$ is empty (resp.\ finite) for all $C$ in some connected component of $\C$ then we will say that $F$ is empty (resp.\ finite) on that connected component. It is then straightforward to check that $F$ is finitely presentable if and only if $F$ is empty on all but finitely many connected components, and whenever $F$ is nonempty on a connected component then it is finite on that component. This property is clearly closed under binary products.

For the converse assume that $\C$ is a groupoid; the goal translates into proving that every object of $\C$ has finite automorphism group. We prove the contrapositive, so let $C \in \C$ have infinite automorphism group $G$. We will only consider presheaves where all components are empty, except for the one at $C$, so that we can identify such presheaves with right $G$-sets. We view $G$ as a presheaf, namely the presheaf represented by $C$, which is thus finitely presentable. We claim that $G \times G$ is not finitely presentable. For an arbitrary $X \subseteq G$ we define the presheaf $A_X := \{ (x, y) \in G \times G \mid x y^{-1} \in X \}$. Then $G \times G = \colim_{X \subseteq G, X \text{ is finite}} A_X$, but $\id_{G \times G}$ does not factor through this diagram, proving our claim.
\end{proof}
\begin{lemma}
\label{product-of-fp-presheaves-if-c-has-products}
Let $\C$ be a small category with binary products. Then the finitely presentable objects in $[\C\op, \Set]$ are closed under binary products.
\end{lemma}
\begin{proof}
The finitely presentable presheaves are precisely the finite colimits of representable presheaves by \cite[Theorem 5.37 and Proposition 5.41]{Kel82:libro}. Colimits in $\Set$ distribute over binary products \cite[Example 3.20]{adamekAlgebraicTheoriesCategorical2010}. Since limits and colimits of presheaves are computed pointwise, the same is true in $[\C\op, \Set]$. So a binary product of finitely presentable objects is a finite colimit of binary products of representables. We conclude because the representables are closed under binary products, by assumption and the fact that the Yoneda embedding preserves limits.
\end{proof}

\subsection{Topos of sheaves on Boolean algebras}
\label{sec:sheaves-on-boolean-algebra}
Here we will consider toposes of sheaves on (countable) Boolean algebras, which turn out to be examples that satisfy all our assumptions, except for the ultraproduct covering property (see \cref{sheaves-cantor-space-fails-ultraproduct-covering-property}). We thank Kanalas for helpful discussions related to the material in this section, and in particular \cref{sheaves-cantor-space-fails-ultraproduct-covering-property} is based on an idea of his (taking place in $\Sh(\RR)$ instead \cite{kanalasRealizingTypesToposvalued2026}).
\begin{theorem}
\label{example-sheaves-on-boolean-algebra}
Let $\Sh(B, J_\coh)$ be the topos of sheaves on a countable Boolean algebra $B$ with respect to the coherent topology, equipped with the cartesian closed structure and the usual (epi, mono) factorisation system. Then:
\begin{enumerate}[label=(\arabic*)]
\item\label{example-sheaves-on-boolean-algebra:lfp-as-closed-category} it is locally finitely presentable as a closed category,
\item\label{example-sheaves-on-boolean-algebra:reduced-product-stable} it is reduced product stable,
\item\label{example-sheaves-on-boolean-algebra:terminal-object-projective} the terminal object (and hence, the monoidal unit) is projective.
\end{enumerate}
\end{theorem}
\begin{proof}
Item \ref{example-sheaves-on-boolean-algebra:lfp-as-closed-category} follows from \cref{coherent-topos-is-lfp-as-closed-category}, since $\Sh(B, J_\coh)$ is a coherent topos. We can equivalently view $\Sh(B, J_\coh)$ as the topos of sheaves on the Stone space of $B$. So it is a localic topos, and following \cite[Example D4.5.2(d)]{Joh02:libro}, supports split in $\Sh(B, J_\coh)$, which means that every subterminal object is projective. So we have \ref{example-sheaves-on-boolean-algebra:terminal-object-projective}.

Finally, we prove \ref{example-sheaves-on-boolean-algebra:reduced-product-stable}. By \cref{locally-finitely-presentable-as-closed-category} it remains to show that epimorphisms are closed under products. So let $\{f_i\colon A_i \to B_i\}_{i \in \I}$ be a family of epimorphisms. As $\Sh(B, J_\coh)$ is localic, the subterminal objects form a generator \cite[Theorem C1.4.7]{Joh02:libro}. By \cref{lifts-along-f-imply-f-in-e} it is thus enough to show that for any subterminal $T$, any arrow $g\colon T \to \prod_{i \in \I} B_i$ factors through $\prod_{i \in \I} f_i$. This quickly follows, because for each $i \in \I$, projectivity of $T$ (see before) yields a factorisation $\pi_i g = f_i h_i$, for some $h_i\colon T \to A_i$, where $\pi_i\colon \prod_{i \in \I} B_i \to B_i$ is the $i$th projection. So $(h_i)_{i \in \I}\colon T \to \prod_{i \in \I} A_i$ is the required factor.
\end{proof}
\begin{proposition}
\label{coherent-topos-is-lfp-as-closed-category}
Any coherent topos is locally finitely presentable as a closed category with respect to the cartesian closed structure.
\end{proposition}
\begin{proof}
Any coherent topos can be presented as $\Sh(\C, J)$, where $\C$ is a finitely complete category and $J$ is generated by finite covering familes \cite[Theorem D3.3.1]{Joh02:libro}. Following \cite[Remark D3.3.12]{Joh02:libro}, such a topos is locally finitely presentable. We will redo part of that argument to prove our stronger claim.

Write $\iota\colon \Sh(\C, J) \to [\C\op, \Set]$ for the inclusion functor, and $a\colon [\C\op, \Set] \to \Sh(\C, J)$ for the sheafification functor. As mentioned in \cite[Remark D3.3.12]{Joh02:libro}, it is straightforward to check that $\iota$ preserves directed colimits. It follows that if $F$ is finitely presentable in $[\C\op, \Set]$ then $a(F)$ is finitely presentable in $\Sh(\C, J)$. Indeed, for a directed diagram $(G_i)_{i \in \I}$ in $\Sh(\C, J)$:
\begin{align*}
\Hom(a(F), \colim_{i \in \I} G_i) &\cong \Hom(F, \iota(\colim_{i \in \I} G_i)) \\
&\cong \Hom(F, \colim_{i \in \I} \iota(G_i)) \\
&\cong \colim_{i \in \I} \Hom(F, \iota(G_i)) \\
&\cong \colim_{i \in \I} \Hom(a(F), G_i),
\end{align*}
as required.

Since $\C$ has a terminal object, the terminal object in $[\C\op, \Set]$ is representable and is thus finitely presentable. It follows that the terminal object in $\Sh(\C, J)$ is finitely presentable.

Now let $F$ and $G$ be finitely presentable objects in $\Sh(\C, J)$. It remains to show that $F \times G$ is finitely presentable. Since $[\C\op, \Set]$ is locally finitely presentable, there is a directed diagram $(A_i)_{i \in \I}$ of finitely presentable objects in $[\C\op, \Set]$ such that $\iota(F) = \colim_{i \in \I} A_i$. Hence
\[
F \cong a \iota(F) = a(\colim_{i \in \I} A_i) \cong \colim_{i \in \I} a(A_i).
\]
Since $F$ is finitely presentable, the identity must factor through this diagram, and so $F$ is a retract of $a(A)$ for some $A$ that is finitely presentable $[\C\op, \Set]$. Similarly, $G$ is a retract of $a(B)$ for some $B$ that is finitely presentable in $[\C\op, \Set]$. So $F \times G$ is a retract of $a(A) \times a(B) \cong a(A \times B)$. We thus have that $F \times G$ is a retract of a finitely presentable object, because $A \times B$ is finitely presentable by \cref{product-of-fp-presheaves-if-c-has-products}, and so $F \times G$ is finitely presentable.
\end{proof}
We finish this section with an example of a topos as in \cref{example-sheaves-on-boolean-algebra} that fails the ultraproduct covering property.
\begin{example}
\label{sheaves-cantor-space-fails-ultraproduct-covering-property}
Recall that the Cantor space $2^\omega$ it the space with a clopen basis given by
\[
U(\eta) = \{ \sigma \in 2^\omega \mid \eta \subseteq \sigma \},
\]
where $\eta \in 2^{<\omega}$ ranges over the finite sequences. We also recall that it is the Stone space of the countable atomless Boolean algebra. So \cref{example-sheaves-on-boolean-algebra} applies to $\Sh(2^\omega)$. We show that the ultraproduct covering property fails for $\Sh(2^\omega)$.

Write $\zeta$ and $\zeta_n$ for the constant $0$ sequences of length $\omega$ and $n$ respectively. For each $n < \omega$, let $A_n$ and $B_n$ be the subterminals corresponding to the opens $U_n := U(\zeta_n)$ and $V := 2^\omega \setminus \{\zeta\}$, respectively. Let $\U$ be a non-principal ultrafilter on $\omega$. We will show that the arrow
\[
\prod_{n < \omega} A_n / \U + \prod_{n < \omega} B_n / \U \to \prod_{n < \omega} A_n + B_n / \U
\]
is not an epimorphism by showing that it is not surjective on the stalk at $\zeta$. Firstly, we argue that $\prod_{n < \omega} A_n / \U = 0$ by showing that $\prod_{n \in J} A_n = 0$ for all $J \in \U$. Indeed, for any $J \in \U$ and any nonempty open $W \subseteq 2^\omega$, there is  $n \in J$ such that $W \not\subseteq U_n$, because $J$ is infinite. Thus $A_n(W) = \emptyset$ and the claim follows.

Since $B_n(W) = \emptyset$ for each neighbourhood $W$ of $\zeta$, we have $\left(\prod_{n \in J} B_n\right)_\zeta = \emptyset$. Stalks commute with colimits, so the stalk of $\prod_{n < \omega} A_n / \U + \prod_{n < \omega} B_n = \prod_{n < \omega} B_n / \U$ at $\zeta$ is empty. To conclude we show that the stalk of $\prod_{n < \omega} A_n + B_n / \U$ at $\zeta$ is nonempty, for which it is enough to show that $(A_n + B_n)(2^\omega)$ is nonempty for every $n < \omega$. Viewing $A_n + B_n$ as the \'{e}tal\'{e} space $p\colon U_n + V \to 2^\omega$, this amounts to producing a section $s\colon 2^\omega \to U_n + V$. We have $2^\omega = U_n + (2^\omega \setminus U_n)$, because $U_n$ is clopen. As $2^\omega \setminus U_n \subseteq V$ we can let $s$ be the induced map $2^\omega = U_n + (2^\omega \setminus U_n) \to U_n + V$.

In fact, a {\PolishL}os-style theorem for positive logic is not possible in $\Sh(2^\omega)$. Consider a language $\LL$ with two propositional symbols $P$ and $Q$. For each $n < \omega$ we let $M_n$ be the model that interprets $P$ and $Q$ as $A_n$ and $B_n$ respectively, and we let $M := \prod_{n < \omega} M_n / \U$. Consider the formula $\phi \equiv P \vee Q$. Then $\phi(M_n) = 1$ for all $n < \omega$, while $\phi(M)$ is the subterminal corresponding to $V$.
\end{example}

\subsection{The category of monoid acts}
The main motivation for this section is that we would like apply our work to the category $\Set^\NN$ of sets with an endomorphism, to capture the model theory of difference algebra (see also \cite{tomasicTopostheoreticViewDifference2020}). However, as we have seen in \cref{set-n-not-lfp-as-cartesian-closed-category}, the cartesian closed structure on this category does not work well for us. We swap our topos-theoretic perspective for a more algebraic one, and work in the generality where $\NN$ is replaced by any commutative monoid.
\label{sec:example-s-acts}
\begin{definition}
\label{category-of-s-acts}
Let $S$ be a monoid. An \emph{$S$-act} is a set $X$ with a left $S$-action. We will write the action of $s \in S$ on $x \in X$ as $s \cdot x$. We write $\Act{S}$ for the category of $S$-acts, where morphisms are functions that commute with the $S$-actions.
\end{definition}
Of course, $\Act{S}$ is the same as the presheaf category $[S, \Set]$, with $S$ viewed as a one-object category. However, as discussed before, we will take a more algebraic perspective and the intended monoidal structure will be different from the cartesian closed structure we put on toposes (see \cref{sec:example-toposes,sec:sheaves-on-boolean-algebra}). We thus use different notation to emphasise the different perspective.
\begin{definition}
\label{s-act-tensor}
Let $S$ be a commutative monoid. We define a tensor operation $\otimes$ on $S$-acts as follows. Let $A$ and $B$ be $S$-acts. Define a binary relation $\sim$ on $A \times B$ by
\[
(s \cdot a, b) \sim (a, s \cdot b),
\]
and let $\approx$ be the equivalence relation generated by $\sim$. We define the tensor of $A$ and $B$ to be the quotient
\[
A \otimes B := (A \times B) / {\approx},
\]
and we write $a \otimes b$ for the equivalence class of $(a, b)$. The $S$-action on $A \otimes B$ is given by $s \cdot (a \otimes b) := s \cdot a \otimes b = a \otimes s \cdot b$.

This assignment is bifunctorial: for $f\colon A \to A'$ and $g\colon B \to B'$ we have a morphism
\begin{align*}
f \otimes g\colon A \otimes B &\to A' \otimes B', \\
a \otimes b &\mapsto f(a) \otimes g(b).
\end{align*}
\end{definition}
The commutativity of $S$ is needed so that $A$ can also be viewed as a right $S$-act, and so the tensor actually makes sense.
\begin{definition}
\label{s-act-internal-hom}
Let $S$ be a monoid. We define a functor $[-, -]\colon \Act{S}\op \times \Act{S} \to \Act{S}$ as follows. For $S$-acts $A$ and $B$ we set
\[
[A, B] = \Act{S}(A, B),
\]
and the $S$-action is pointwise. That is, for $f\colon A \to B$ we let $(s \cdot f)\colon A \to B$ be given by $(s \cdot f)(a) = s \cdot f(a)$. The functor's action on morphisms is given by precomposition and postcomposition, respectively.
\end{definition}
\begin{proposition}
\label{s-acts-symmetric-monoidal-closed}
Let $S$ be a commutative monoid. The tuple $(\Act{S}, \otimes, S, [-, -])$, is a symmetric monoidal closed category. Here $S$ acts on itself in the obvious way.
\end{proposition}
\begin{proof}
This is a straightforward, but lengthy, verification. It also follows from more general results, as the construction here is a specific instance of Day convolution \cite{dayConstructionBiclosedCategories1970}, for which the above conclusion is known. Alternatively, $\Act{S}$ can be seen as the Eilenberg-Moore category of algebras for a commutative monad, so by \cite[Corollary 2.5.7]{sealTensorsMonadsActions2013} the tuple in the statement is indeed symmetric monoidal closed.
\end{proof}
\begin{theorem}
\label{example-s-acts}
Let $S$ be a commutative monoid. Then $(\Act{S}, \otimes, S, [-, -])$ with the usual $(\E, \M) := (\text{epi}, \text{mono})$ factorisation system has the following properties:
\begin{enumerate}[label=(\arabic*)]
\item\label{example-s-acts:lfp-as-closed-category} it is locally finitely presentable as a closed category,
\item\label{example-s-acts:reduced-product-stable} it is reduced product stable,
\item\label{example-s-acts:ultraproduct-covering-property} it satisfies the ultraproduct covering property,
\item\label{example-s-acts:terminal-projective} the monoidal unit is projective.
\end{enumerate}
\end{theorem}
\begin{proof}
The proof is mostly the same as \cref{example-presheaf-toposes}, because $\Act{S}$ is a presheaf category. We note the differences, which come from the different monoidal structure.

The monoidal unit is a representable presheaf (in fact, the only representable presheaf) and is thus finitely presentable and projective.

We are left to verify that finite powers commute with directed colimits, which was the point of this different monoidal structure. So let $A$ be finitely presentable and let $(C_i)_{i \in \I}$ be a directed system. Then indeed we have
\[
[A, \colim_{i \in \I} C_i] \cong \colim_{i \in \I} [A, C_i],
\]
as sets by definition of finite presentability, and it quickly follows that this is an isomorphism of $S$-acts.
\end{proof}

\subsection{The category of abelian groups and other \texorpdfstring{$\CMon$}{CMon}-categories}
\label{sec:example-abelian-groups}

For this section we let $\V$ be a symmetric monoidal closed category that is locally finitely presentable as a closed category and naturally enriched over $\CMon$, the category of commutative monoids. The final requirement is equivalent to asking that $\V$ has {\em finite direct sums} (that is, finite coproducts coincide with finite products). The prototypical example is $\Ab$, the category of abelian groups, but we shall see other examples later.

On such a $\V$ we can always choose the (strong epi, mono) factorisation system, which we now fix. As usual, since $\V$ is locally finitely presentable as a closed category, it is reduced product ready (see \cref{locally-finitely-presentable-as-closed-category}). Moreover, since finite coproducts are absolute, they commute with (all limits and in particular) products, so that: 
$$
0\cong \prod_{i\in J}0 \qquad \text{and}\qquad   \prod_{i\in J}A_i + \prod_{i\in J}B_i \cong\prod_{i\in J}(A_i+B_i)
$$
for any set $J$ and $J$-indexed families of objects. Now, given any ultrafilter $\U$ on a set $\I$, using the canonical isomorphism above and the fact that finite coproducts commute with directed colimits, it follows that
$$
0\cong \prod_{i\in \I}0/\U \qquad \text{and}\qquad   \prod_{i\in \I}A_i/\U + \prod_{i\in \I}B_i/\U \cong\prod_{i\in \I}(A_i+B_i)/\U,
$$
showing that $\V$ satisfies the ultraproduct covering property. 

It remains to give conditions on $\V$ implying stability of strong epimorphisms under products; this would give that $\V$ is also reduced product stable.  This is certainly the case when $\V$ is {\em infinitary regular}, as this readily implies that strong epimorphisms and regular epimorphisms coincide and are stable under products. 

Recall from \cite[Section~3.B]{AR94:libro} that $\V$ is called a {\em finitary quasivariety} if it has a strong (and hence regular) generator made of finitey presentable and projective objects (that is, objects $X$ for which $\V_0(X,-)$ preserves directed colimits and regular epimorphisms). It is well known that finitary quasivarieties are locally finitely presentable and infinitary regular (see again \cite[Section~3.B]{AR94:libro} or \cite[Section~4]{LT20:articolo}). 

In \cite[Definition 4.11]{LT20:articolo} a symmetric monoidal closed $\V$ is called a {\em symmetric monoidal finitary quasivariety} if it is a finitary quasivariety, the monoidal unit $I$ is finitely presentable, and the finitely presentable projective objects are closed under the monoidal product (note that $I$ is not required to be projective).

\begin{theorem}
	Let $\V$ be a symmetric monoidal finitary quasivariety which is also $\CMon$-enriched. Then $\V$, equipped with the (regular epi, mono) factorisation system, has the following properties:
	\begin{enumerate}[label=(\arabic*)]
		\item it is locally finitely presentable as a closed category,
		\item it is reduced product stable,
		\item it satisfies the ultraproduct covering property.
	\end{enumerate}
\end{theorem}
\begin{proof}
	Note that $\V$ is locally finitely presentable as a closed category by \cite[Remark~4.13]{LT20:articolo}; then everything follows by the arguments above.
\end{proof}

It follows that all the $\CMon$-enriched examples of \cite{LT20:articolo}, as well as all the examples of \cite[Section~4.2]{LT21:articolo}, fit this framework. Here we mention a few, where the monoidal structure is assumed to be the standard one:\begin{enumerate}
	\item the category $\CMon$ of commutative monoids;
	\item the category $\Ab$ of abelian groups;
	\item the category $\Ab_\textup{tf}$ of torsion free abelian groups;
	\item for a commutative ring $R$, the categories $R$-$\Mod$ of $R$-modules, $G$-GR$(\Mod)$ of $G$-graded $R$-modules  over an abelian group $G$, and DG($R$-$\Mod$) of differentially graded $R$-modules;
	\item more generally, for $\V$ as in \cite[Section~4.2]{LT21:articolo} and any small compact closed $\V$-category $\C$, the category $[\C,\V]_0$ of $\V$-functors and $\V$-natural transformations between them (with Day's convolution).
\end{enumerate}
In the examples above, the unit $I$ is regular projective in all cases but that of differentially graded modules and possibly (5).

\subsection{The category of small \texorpdfstring{$\V$}{V}-categories}
\label{sec:example-categories}

Let $\V$ a be symmetric monoidal closed category, then so is the category $\V$-$\Cat$ of small $\V$-categories and $\V$-functors~\cite[Section~1.4]{Kel82:libro}. The monoidal unit $\mathbf{I}$ is given by the $\V$-category with one single object $\star$ and hom-object $\mathbf{I}(\star,\star)=I$. The tensor product $\A\otimes \B$ of two $\V$-categories has $\Ob(\A)\times \Ob(\B)$ as set of objects, and 
$$ \A\otimes \B ((A,B),(A',B')):=\A(A,A')\otimes \B(B,B') $$
as hom-objects (see~\cite{Kel82:libro} for details).

If we assume moreover that $\V$ is locally finitely presentable as a closed category, then so is the category $\V$-$\Cat$ by \cite[Proposition~4.8]{KL2001:articolo}. 

Regarding factorisation systems, we have a few choices:\begin{enumerate}[label=(\roman*)]
	\item (surjective on objects, injective on objects and fully faithful);
	\item (strong epi, mono);
	\item (epi, strong mono).
\end{enumerate}
In (i) and (ii) $\M$ is closed under directed colimits, so that $\V$-$\Cat$ is reduced product ready (\cref{locally-finitely-presentable-as-closed-category}). The class $\E$ is closed under products in (i), so for this choice $\V$-$\Cat$ is reduced product stable; moreover the unit is $\E$-projective. 

\begin{remark}
	Note that there is also the factorisation system (bijective on objects, fully faithful), but since fully faithful morphisms are not necessarily monomorphic, this is not of interest for us.
\end{remark}

We now focus on the ultraproduct covering property. The initial object in $\V$-$\Cat$ is the $\V$-category with no objects. As such it forms a strictly initial object; hence the first part of the ultraproduct covering property holds by \cref{ultrapower-of-0}.

Next, recall that given $\V$-categories $\A,\B$, their coproduct $\A+\B$ in $\V$-$\Cat$ has the disjoint union of the objects of $\A$ and of $\B$ as set of objects and 
$$ 
(\A+\B)(X,Y)=
\begin{cases}
 	\ \A(X,Y)\ & \text{if $X,Y\in\A$,} \\
	\ \B(X,Y) & \text{if $X,Y\in\B$,} \\
	\ 0 & \text{otherwise.}
\end{cases} 
$$
We also fix the following notation for an object $G\in\V$: the $\V$-category $\bbtwo_G$ has two objects $x,y$ and only non-trivial hom $\bbtwo_G(x,y):=G$.

\begin{proposition}
	Assume that in $\V$ the initial object $0$ is subterminal; that is, for any object $X$ there is at most one arrow $X\to 0$. For any ultrafilter $\U$ on a set $\I$, and families $(\A_i)_{i\in \I}$ and $(\B_i)_{i\in \I}$ of $\V$-categories the comparison $\V$-functor
	$$ \rho\colon\prod_{i\in \I}\A_i/\U + \prod_{i\in \I}\B_i/\U \longrightarrow \prod_{i\in \I}(\A_i+\B_i)/\U$$
	is bijective on objects and a strong epimorphism. In particular $\rho$ lies in $\E$ for any of the factorisation systems mentioned above and thus $\V$-$\Cat$ satisfies the ultraproduct covering property.
\end{proposition}
\begin{proof}
	We have already observed that $\V$-$\Cat$ has a strict initial objects, so to show the ultraproduct covering property it is really enough to prove the conditions about $\rho$. Fix a strong generator $\G$ of $\V$ made of finitely presentable objects; by \cite[Proposition~4.7]{KL2001:articolo} the family $\{\bbtwo_G\}_{G\in\G}$ is a strong generator for $\V$-$\Cat$ made of finitely presentable $\V$-categories. To prove that $\rho$ is a strong epimorphism we apply \cref{lifts-along-f-imply-f-in-e}. 
	
	Fix then $G\in\G$ and a $\V$-functor $F\colon \bbtwo_G\to \prod_{i\in \I}(\A_i+\B_i)/\U$. By finite presentability there exists $J\in\U$ such that $F$ factors as $F_J\colon \bbtwo_G\to \prod_{j\in J}(\A_j+\B_j)$, which is then determined by a family of $\V$-functors
	$$ F_j\colon \bbtwo_G\to \A_j+\B_j $$
	for $j\in J$. By definition of $\bbtwo_G$ each $F_j$ is determined by the choice of two objects $X_j,Y_j$ in $\A_j+\B_j$ and an arrow 
	$$ f_j\colon G\to (\A_j+\B_j)(X_j,Y_j) $$
	in $\V$. We can write $J$ as the (disjoint) union of the following sets:	\begin{itemize}
		\item $J_1:=\{j\in J \mid X_j,Y_j\in\A_j \}$;
		\item $J_2:=\{j\in J \mid X_j,Y_j\in\B_j \}$;
		\item $J_3:=\{j\in J \mid X_j\in\A_j,\ Y_j\in\B_j \}$;
		\item $J_4:=\{j\in J \mid X_j\in\B_j,\ Y_j\in\A_j \}$.
	\end{itemize}
	By the ultrafilter properties exactly one of these lies in $\U$. We treat only the cases were $J_1$ or $J_3$ are in $\U$ since the other two are symmetric.
	Suppose that $J_1\in\U$, then we can define a $\V$-functor
	$$ \tilde F\colon \bbtwo_G\xrightarrow{F_{J_1}}\prod_{j\in J_1}\A_j \xrightarrow{\ a_{J_1}\ } \prod_{i\in \I}\A_i/\U\xrightarrow{\iota_1}\prod_{i\in \I}\A_i/\U + \prod_{i\in \I}\B_i/\U $$ 
	where $a_{J_1}$ is the coprojection and $\iota_1$ is the first coproduct inclusion, while $F_{J_1}$ sends $x$ to $(X_j)_{j\in J_1}$, $y$ to $(Y_j)_{j\in J_1}$, and with action on the homs
	$$ (f_j)_{j\in J_1}\colon G\to (\prod_{j\in J_1}\A_i)((X_j)_{j\in J_1},(Y_j)_{j\in J_1}). $$
	One can easily conclude that $\rho\circ \tilde F=F$.
	Suppose now that $J_3\in\U$, then we can define a $\V$-functor
	$$ \tilde F\colon \bbtwo_G\xrightarrow{F_{J_3}}\prod_{j\in J_3}\A_j + \prod_{j\in J_3}\B_j\xrightarrow{\ a_{J_3} + b_{J_3}\ } \prod_{i\in \I}\A_i/\U + \prod_{i\in \I}\B_i/\U $$ 
	where $a_{J_3}$ and $b_{J_3}$ are the coprojections, while $F_{J_3}$ sends $x$ to $(X_j)_{j\in J_3}\in \prod_{j\in J_3}\A_i$, $y$ to $(Y_j)_{j\in J_3}\in \prod_{j\in J_3}\B_i$, and with action on the homs
	$$ G\xrightarrow{\ f\ }  0=(\prod_{j\in J_3}\A_i + \prod_{j\in J_3}\B_i)((X_j)_{j\in J_3},(Y_j)_{j\in J_3}),$$
	where $f$ is the unique arrow $G \to 0$. This arrow exists, as witnessed by $f_j$ for any $j \in J_3$ and is unique by the assumption on $0$. It is easy to see that $\rho\circ \tilde F=F$.
	
	To show that $\rho$ is bijective on objects it is enough to observe that $\Ob\colon \V$-$\Cat\to\Set$ is continuous and cocontinuous (see for instance \cite[Section~5]{KL2001:articolo}) and that $\rho$ is a bijection when taking ultraproducts of sets.
\end{proof}

\begin{remark}
	If in $\V$ every arrow out of the initial object is a monomorphism, then one can show that $\rho$ is moreover faithful (the hom components are monomorphisms) and hence an isomorphism of $\V$-categories (being both a monomorphism and a strong epimorphism). 
	
	If $\V$ has a strict initial object this follows from \cref{ultraproduct-covering-property-isos} since one can easily check that the family $\{\bbtwo_G\}_{G\in\G}$ considered in the proof is a strong generator of finitely presentable and connected objects. However, this does not cover for instance the case of $\V=\Ab$.
\end{remark}

Summarising we obtain the following.

\begin{theorem}\label{V-cat-theorem}
	Let $\V$ be locally finitely presentable as a closed category and such that the initial object is subterminal. Then $\V$-$\Cat$, equipped with one of the factorisation systems (i)--(iii) has the following properties:
	\begin{enumerate}[label=(\arabic*)]
		\item it is locally finitely presentable as a closed category,
		\item it is reduced product ready,
		\item it satisfies the ultraproduct covering property.
	\end{enumerate}
	If the given factorisation system is (i), then $\V$-$\Cat$ is also reduced product stable and the unit is $\E$-projective.
\end{theorem}

\begin{example}
	Examples of bases of enrichment that can be obtained via this procedure include:\begin{itemize}
		\item the category $\Set$, by taking $\V=1$ the singleton category;
		\item the category $\mathbf{Prd}$ of preorders, by enriching over $\V=\{0,1\}$ with truncated addition as monoidal structure;
		\item the category $\Ab$-$\Cat$ of small preadditive categories;
		\item the category $\Cat$ of small categories, by enriching over $\Set$;
		\item the category $n$-$\Cat$ of small n-categories, by enriching on $(n-1)$-$\Cat$ (this is obtained recursively starting from the previous point).
	\end{itemize}
\end{example}

\begin{example}
	In the context of $2$-categories, the paper~\cite{GT26isoregular} introduces the fragment {\em isoregular logic} of enriched regular logic over $\V=\Cat$ with the (strong epi, mono) factorisation system. Its aim is to define a logical framework to present categories with structures defined by universal properties. 
	
	If we extend isoregular logic by allowing finite disjunctions, \cref{V-cat-theorem} tells us that the ultraproduct covering property holds. Moreover, since isoregular logic is defined to involve only those strong epimorphisms in $\Cat$ which are {\em surjective equivalences}, and these are stable under products, it is easy to see that all the results that involve reduced product stability also apply to this context.
\end{example}

\subsection{The category of posets}

Let us fix as base of enrichment the cartesian closed category $\Pos$ of partially ordered sets. This is well known to be locally finitely presentable as a closed category, so that conditions \ref{reduced-product-ready:finite-limits-directed-colimits-commute} and \ref{reduced-product-ready:finite-powers-directed-colimits-commute} of \cref{reduced-product-ready} are automatically satisfied (see \cref{locally-finitely-presentable-as-closed-category}).

We have two choices of meaningful factorisation systems on $\Pos$.
\begin{itemize}
	\item The (surjective, embedding) factorisation system, which coincides with the (epi, strong mono) one. Here, a surjection is simply an order preserving map that is surjective as a function, and an embedding is injective morphism of posets that reflects the order relation.
	\item The (strong epi, mono) factorisation system. See \cite[Examples~A.2(2)]{RT23EUA} for an explicit description and the fact that strong epimorphisms coincide with regular epimorphisms and are stable under products.
\end{itemize}

In both cases the class $\M$ is closed under directed colimits and $\E$ is closed under products, making $\V$ reduced product stable. The unit $1$ is $\E$-projective in both cases. Moreover, the poset $$\bbtwo:=\{\cdot < \cdot\}$$ is a finitely presentable and connected object and forms a strong generator of $\Pos$; thus the ultraproduct covering property holds for both factorisation systems by \cref{ultraproduct-covering-property-isos}.

We have thus proved the following.

\begin{theorem}
	The cartesian closed category $\Pos$ of posets equipped with either the (surjective, embedding) or the (strong epi, mono) factorisation system, has the following properties:\begin{enumerate}[label=(\arabic*)]
		\item it is locally finitely presentable as a closed category,
		\item it is reduced product stable,
		\item it satisfies the ultraproduct covering property;
		\item the unit is $\E$-projective.
	\end{enumerate}
\end{theorem}

\subsection{The category of pseudometric spaces}\label{pseudometric-section}

Also of interest in theoretical computer science, especially in quantitative reasoning, is enrichment over the category $\Met$ of (generalised) metric spaces and contractions. As we show in the next section, this is not suitable for us since directed colimits do not commute with finite limits (making $\Met$ not reduced product ready). This problem vanishes if we replace metric spaces with pseudometric spaces.

\begin{definition}
\label{pseudometric-space}
A {\em (generalised) pseudometric space} is the data of a set $X$ together with a distance $d\colon X\times X\to[0,\infty]$ which is symmetric, satisfies the triangular inequality, and such that $d(x,x)=0$ for any $x\in X$. We denote by $\PMet$ the category of pseudometric spaces and contractions.
\end{definition}
We do not require that if $d(x,y)=0$ then $x=y$; that would make a pseudometric space into a metric space.

The category $\PMet$ is locally $\aleph_1$-presentable and the forgetful functor $\PMet\to\Set$ is both continuous and cocontinuous \cite[Example~2.3(4)]{AR20:articolo}. The same proof as that appearing in \cite[Example~2.2(5)]{adamek2015finitary} shows that the finitely presentable pseudometric spaces are exactly the finite discrete spaces (that is, those finite spaces for which $d(x,y)=\infty$ whenever $x\neq y$); these are exactly the finite coproducts of the terminal object. (Note that since finite discrete spaces are not a strong generator, this implies that $\PMet$ is not locally finitely presentable.)

The monoidal closed structure which is usually considered on $\Met$ (see e.g., \cite[Remark~2.2]{AR20:articolo}) extends with no changes to $\PMet$. The monoidal unit is given by the terminal object $1$, and for any $X,Y\in\PMet$ we define $X\otimes Y:=(X\times Y,\ d_{X\otimes Y})$ where
$$ d_{X\otimes Y}((x,y),(x',y'))= d_X(x,x')+d_Y(y,y'). $$
The internal hom is defined by $[X,Y]:=(\PMet(X,Y),\ d_{[X,Y]})$ where
$$ d_{[X,Y]}(f,g)=\sup_{x\in X}d_Y(f(x),g(x)). $$
Directed colimits commute with finite limits in $\PMet$. The case of products is treated in \cite[Lemma~3.7]{adamek2015finitary} (there the distance takes value in $[0,1]$, but the proof does not change). Commutation with equalisers follows easily from the fact that these are computed in $\PMet$ as in $\Set$ via isometric subspaces, and isometries are stable under directed colimits. As for finite powers, since the finitely presentable objects are finite discrete spaces, powers by them are just finite products (of the same object multiple times) and as such commute with directed colimits.

Like in $\Met$, we have two factorisation systems of interest:
\begin{itemize}
	\item (surjective, isometry);
	\item (dense, closed isometry).
\end{itemize}
In both cases the class $\E$ is closed under products and $\M$ under directed colimits. Moreover the monoidal unit $1$ is an $\E$-generator with respect to both factorisation systems, and $\E$-projective with respect to the first. It follows by \cref{sufficient-condition-for-ultraproduct-covering-property} that $\PMet$ satisfies the ultraproduct covering property.

Summarising, even though $\PMet$ is not locally finitely presentable, we have shown the following.

\begin{theorem}
	The symmetric monoidal closed category $\PMet$ equipped with either the (surjective, isometry) or the (dense, closed isometry) factorisation system, has the following properties:\begin{enumerate}[label=(\arabic*)]
		\item it is reduced product stable,
		\item it satisfies the ultraproduct covering property,
		\item the monoidal unit is finitely presentable and finitely presentable objects are closed under the monoidal product,
		\item the monoidal unit is $\E$-projective with respect to the first factorisation system.
	\end{enumerate}
\end{theorem}

\subsection{Non-examples and counterexamples}

Enrichment over the cartesian closed category $\omegaCPO$, of $\omega$-complete partial orders, is of interest in theoretical computer science and its logic has been studied for instance in \cite{adamek1985birkhoff}. However, finitary enriched logic does not seem to behave nicely in this context. Indeed, $\omegaCPO$ is known not to be locally finitely presentable and, in addition, finite limits do not commute with directed colimits as shown in the example below. Thus it is not even reduced product ready.

\begin{example}
	Consider the $\omega$-complete poset $\omega+1=\{0<1<\cdots <\omega\}$, and for any natural $n>0$ the object $P_n=\{0<1<\cdots<n<\infty\}$. These, ordered by inclusion, form a directed family whose colimit is $\omega+2=\{0<1<\cdots <\omega<\infty\}$. 
	
	Consider, for $n>0$, the $\omega$-complete poset $Q_n$ obtained by adding to $\omega+1$ an element $\overline n$ such that $n<\overline n<\omega$, but $\overline n$ is not comparable with any other $n<k<\omega$. These form a directed family where $Q_n\to Q_{n+1}$ fixes $\omega+1$ and sends $\overline n$ to $\overline{n+1}$. The colimit of this in $\omegaCPO$ is still $\omega+1$ (since, in such a colimit $\overline n\sim \overline{n+1}$ and $\omega=\sup_n n\leq\sup_n \overline n\leq\omega$, so also $\overline n\sim\omega$).
	
	Now, for any $n$ let $f_n,g_n\colon P_n\to Q_n$ be defined by $f_n(k)=g_n(k)=k$ for any $k\leq n$, $f_n(\infty)=\omega$, and $g_n(\infty)=\bar n$. It is easy to see that $f_n$ and $g_n$ are natural with respect to the directed diagrams.
	\begin{itemize}
		\item When taking the directed colimit, the maps $f_\infty$ and $g_\infty$ coincide both with the map $\omega+2\to\omega+1$ which is the identity on the natural numbers and sends $\omega$ and $\infty$ to $\omega$. Thus the equaliser of the directed colimit trivialises, returning $\omega+2$.
		\item Given $n\in\mathbb N$, the equaliser of $f_n$ and $g_n$ coincides with $\{0<1<\cdots <n\}$. Therefore, the directed colimit of the equalisers coincides with $\omega +1$ (it would be $\omega$ in $\Pos$, but since we are in $\omegaCPO$ we need to add a top element). 
	\end{itemize}
	It follows that equalisers do not commute with directed colimits.	
\end{example} 

The same problem arises for the categories $\Met$ and $\CMet$ of (generalised) metric spaces and complete metric spaces respectively, with contractions as arrows: finite limits do not commute with directed colimits. 

\begin{example}[{\cite[Remark~2.3]{rosicky2021metric}}]
	Given $\varepsilon\in(0,\infty]$, denote by $2_\varepsilon$ the metric space with two points of distance $\varepsilon$. Both in $\Met$ and $\CMet$ we can consider the monomorphisms 
	$$ 2_\infty \to 2_{1/n},$$
	for $n\in\mathbb N$, given by the identity on the two elements. These form a directed family indexed on the natural numbers whose colimit is the unique map
	$$ 2_\infty \to 1,$$
which is not a monomorphism.	 
\end{example}

% References
\bibliography{biblio}

@inproceedings{mardare2016quantitative,
	title={Quantitative algebraic reasoning},
	author={Mardare, Radu and Panangaden, Prakash and Plotkin, Gordon},
	booktitle={Proceedings of the 31st Annual ACM/IEEE Symposium on Logic in Computer Science},
	pages={700--709},
	year={2016}
}

@book{iovino2014applications,
	title={Applications of model theory to functional analysis},
	author={Iovino, Jos{\'e}},
	year={2014},
	publisher={Courier Corporation}
}

@article{bloom1976varieties,
	title={Varieties of ordered algebras},
	author={Bloom, Stephen L},
	journal={Journal of Computer and System Sciences},
	volume={13},
	number={2},
	pages={200--212},
	year={1976},
	publisher={Elsevier}
}

@article{adamek2022categorical,
	title={A categorical view of varieties of ordered algebras},
	author={Ad{\'a}mek, Jir{\'\i} and Dost{\'a}l, Matej and Velebil, Jir{\'\i}},
	journal={Mathematical Structures in Computer Science},
	volume={32},
	number={4},
	pages={349--373},
	year={2022},
	publisher={Cambridge University Press}
}

@book{lawvere1968equality,
	title={Equality in hyperdoctrines and comprehension schema as an adjoint functor},
	author={Lawvere, Francis W},
	year={1968},
	publisher={Forschungsinst. f. Mathe.}
}

@inproceedings{seely1984locally,
	title={Locally cartesian closed categories and type theory},
	author={Seely, Robert AG},
	booktitle={Mathematical proceedings of the Cambridge philosophical society},
	volume={95},
	number={1},
	pages={33--48},
	year={1984},
	organization={Cambridge University Press}
}

@article{lambek1968deductive,
	title={Deductive systems and categories: I. Syntactic calculus and residuated categories},
	author={Lambek, Joachim},
	journal={Mathematical systems theory},
	volume={2},
	number={4},
	pages={287--318},
	year={1968},
	publisher={Springer}
}

@article{lawvere1969adjointness,
	title={Adjointness in foundations},
	author={Lawvere, F William},
	journal={Dialectica},
	volume={23},
	number={3/4},
	pages={281--296},
	year={1969},
	publisher={JSTOR}
}

@inproceedings{lawvere1970quantifiers,
	title={Quantifiers and sheaves},
	author={Lawvere, F William},
	booktitle={Actes du congres international des mathematiciens, Nice},
	volume={1},
	pages={329--334},
	year={1970}
}

@book{maclane2012sheaves,
	title={Sheaves in geometry and logic: A first introduction to topos theory},
	author={MacLane, Saunders and Moerdijk, Ieke},
	year={2012},
	publisher={Springer Science \& Business Media}
}

@article{lucyshyn2014enriched,
	title={Enriched factorization systems},
	author={Lucyshyn-Wright, Rory BB},
	journal={Theory and Applications of Categories},
	volume={29},
	number={18},
	pages={475--495},
	year={2014}
}

@misc{GT26isoregular,
	title={Isoregular theories, accessible 2-categories, and free constructions},
	author={Gambino, Nicola and Tendas, Giacomo},
	note={arXiv:2606.26717 (preprint)},
	year={2026}
}

@article{rosicky2021metric,
	title={Metric monads},
	author={Rosick{\'y}, Ji{\v{r}}{\'\i}},
	journal={Mathematical Structures in Computer Science},
	volume={31},
	number={5},
	pages={535--552},
	year={2021},
	publisher={Cambridge University Press}
}

@misc{AS26,
	title={Graduated categories of presheaves},
	author={Ad{\'a}mek, Ji{\v{r}}i and Sousa, Lurdes},
	note={arXiv:2609.30444 (preprint)},
	year={2026}
}

@article{adamek2015finitary,
	title={On finitary functors and their presentations},
	author={Ad{\'a}mek, Ji{\v{r}}{\'\i} and Milius, Stefan and Moss, Lawrence S and Urbat, Henning},
	journal={Journal of Computer and System Sciences},
	volume={81},
	number={5},
	pages={813--833},
	year={2015},
	publisher={Elsevier}
}

@misc{RT25positive,
	title={Enriched positive logic},
	author={Rosick{\'y}, Ji{\v{r}}{\'\i} and Tendas, Giacomo},
	note={arXiv:2509.20019 (preprint)},
	year={2025}
}

@article{adamek1985birkhoff,
	title={The Birkhoff variety theorem for continuous algebras},
	author={Ad{\'a}mek, Ji{\v{r}}i and Nelson, Evelyn and Reiterman, Jan},
	journal={algebra universalis},
	volume={20},
	number={3},
	pages={328--350},
	year={1985},
	publisher={Springer}
}

@book{Joh02:libro,
	author	={Johnstone, P.},
	title	={Sketches of an Elephant: A Topos Theory Compendium},
	publisher ={Oxford Logic Guides},
	year	={2002},
}

@book{AR94:libro,
	AUTHOR = {Ad\'{a}mek, J. and Rosick\'{y}, J.},
	TITLE = {Locally presentable and accessible categories},
	SERIES = {London Mathematical Society Lecture Note Series},
	VOLUME = {189},
	PUBLISHER = {Cambridge University Press, Cambridge},
	YEAR = {1994},
	PAGES = {xiv+316},
	ISBN = {0-521-42261-2},
	MRCLASS = {18Axx (18-02)},
	MRNUMBER = {1294136},
	MRREVIEWER = {J. R. Isbell},
	DOI = {10.1017/CBO9780511600579},
	URL = {https://doi.org/10.1017/CBO9780511600579},
}

@book{MP89:libro,
  title = {Accessible Categories: The Foundations of Categorical Model Theory},
  shorttitle = {Accessible Categories},
  author = {Makkai, Mih{\'a}ly and Par{\'e}, Robert},
  year = 1989,
  series = {Contemporary {{Mathematics}}},
  number = {104},
  publisher = {American Mathematical Society}
}

@book{Kel82:libro,
	author	={Kelly, G. M.},
	title	={Basic Concepts of Enriched Category Theory},
	publisher ={Cambridge University Press},
	year	={1982},
}

@book{Pre11:libro,
	author	={Prest, M.},
	title	={Definable Additive Categories: Purity and Model Theory},
	publisher ={Memoirs of the American Mathematical Society},
	year	={2011},
}

@book{Bor94:libro,
	author	={Borceux, F.},
	title	={Handbook of Categorical Algebra: Volume 2, Categories and Structures},
	publisher ={Cambridge University Press},
	year	={1994},
}

@article{Kel82:articolo,
	AUTHOR = {Kelly, G. M.},
	TITLE = {Structures defined by finite limits in the enriched context. {I}},
	JOURNAL = {Cahiers de Topologie et G\'{e}om\'{e}trie Diff\'{e}rentielle},
	VOLUME = {23},
	YEAR = {1982},
	NUMBER = {1},
	PAGES = {3--42},
	ISSN = {0008-0004},
	MRCLASS = {18D15 (18D20)},
	MRNUMBER = {648793},
	MRREVIEWER = {Walter Tholen},
}

@article{JK01:articolo,
	author	={Janelidze, G. and Kelly, G.M.},
	title	={A Note on Actions of a Monoidal Category},
	journal ={Theory and Applications of Categories},
	volume = {9},
	number = {4},
	pages 	={61-91},
	year	={2001},
}

@article{LT20:articolo,
	title = {Enriched regular theories},
	journal = {Journal of Pure and Applied Algebra},
	volume = {224},
	number = {6},
	year = {2020},
	issn = {0022-4049},
	doi = {https://doi.org/10.1016/j.jpaa.2019.106268},
	url = {http://www.sciencedirect.com/science/article/pii/S0022404919302816},
	author = {Lack, S. and Tendas, G.},

}

@article{KL2001:articolo,
	title={{V}-{C}AT IS LOCALLY PRESENTABLE OR LOCALLY BOUNDED IF {V} IS SO},
	author={Kelly, G. M. and Lack, S.},
	journal={Theory and Applications of Categories},
	year={2001},
	volume={8},
	pages={555-575}
}

@article{AR20:articolo,
  	title={Approximate injectivity and smallness in metric-enriched categories},
  	author={Ad{\'a}mek, Ji{\v{r}}{\'\i} and Rosick{\`y}, Ji{\v{r}}{\'\i}},
  	journal={Journal of Pure and Applied Algebra},
  	volume={226},
  	number={6},
  	year={2022},
  	publisher={Elsevier}
  }

@article{LT21:articolo,
	title={Flat vs. filtered colimits in the enriched context},
	author={Lack, Stephen and Tendas, Giacomo},
	journal={Advances in Mathematics},
	volume={404},
	year={2022},
	publisher={Elsevier}
}

@article{LT22:virtual,
	title={Virtual concepts in the theory of accessible categories},
	author={Lack, Stephen and Tendas, Giacomo},
	journal={Journal of Pure and Applied Algebra},
	year={2023},
	volume= {227},
	number= {2},
	doi={https://doi.org/10.1016/j.jpaa.2022.107196}
}

@article{RT23EUA,
author={Rosick{\'y}, Jiri and Tendas, Giacomo},
title={Towards enriched universal algebra},
journal={Selecta Mathematica},
volume={32},
number={2},
pages={21},
year={2026},
publisher={Springer}
}

@article{RT25ERegular,
	title={Enriched concepts of regular logic},
	author={Rosick{\'y}, Jiri and Tendas, Giacomo},
	journal={The Journal of Symbolic Logic},
	year={2025}
}

@misc{tomasicTopostheoreticViewDifference2020,
  title = {A Topos-Theoretic View of Difference Algebra},
  author = {Toma{\v s}i{\'c}, Ivan},
  year = 2020,
  note = {arXiv:2001.09075 (preprint)}
}

@article{freyd1972categories,
	title={Categories of continuous functors, {I}},
	author={Freyd, Peter J and Kelly, G Max},
	journal={Journal of pure and applied algebra},
	volume={2},
	number={3},
	pages={169--191},
	year={1972},
	publisher={North-Holland}
}

@article{sealTensorsMonadsActions2013,
  title = {Tensors, Monads and Actions},
  author = {Seal, Gavin J.},
  year = 2013,
  month = jul,
  journal = {Theory and Applications of Categories},
  volume = {28},
  number = {15},
  pages = {403--434}
}

@phdthesis{dayConstructionBiclosedCategories1970,
  title = {Construction of Biclosed Categories},
  author = {Day, Brian John},
  year = 1970,
  eprinttype = {hdl},
  doi = {10.26190/UNSWORKS/8048},
  collaborator = {{UNSW Sydney}},
  school = {UNSW Sydney}
}

@book{adamekAlgebraicTheoriesCategorical2010,
  title = {Algebraic {{Theories}}: {{A Categorical Introduction}} to {{General Algebra}}},
  shorttitle = {Algebraic {{Theories}}},
  author = {Ad{\'a}mek, J. and Rosick{\'y}, J. and Vitale, E. M.},
  year = 2010,
  series = {Cambridge {{Tracts}} in {{Mathematics}}},
  publisher = {Cambridge University Press},
  doi = {10.1017/CBO9780511760754}
}

@misc{kanalasRealizingTypesToposvalued2026,
  title = {Realizing Types in Topos-Valued Models},
  author = {Kanalas, Krist{\'o}f},
  year = 2026,
  month = jun,
  note = {Slides, Making Categorical Model Theory Accessible, Brno},
  url = {https://www.markkamsma.nl/making-categorical-model-theory-accessible/slides/kanalas.pdf}
}

@article{makkaiStoneDualityFirst1987,
  title = {Stone Duality for First Order Logic},
  author = {Makkai, M.},
  year = 1987,
  month = aug,
  journal = {Advances in Mathematics},
  volume = {65},
  number = {2},
  pages = {97--170},
  doi = {10.1016/0001-8708(87)90020-X}
}

@misc{lurieUltracategories,
  title = {Ultracategories},
  author = {Lurie, Jacob},
  url = {https://people.math.harvard.edu/~lurie/papers/Conceptual.pdf},
  year = {2018}
}
\bibliographystyle{abbrv}

\end{document}